\documentclass[reqno,11pt]{amsart}
\usepackage{amsmath,latexsym,amsfonts,amsbsy,amsthm,mathtools}
\usepackage{amssymb}
\usepackage{mathrsfs}
\usepackage{shortvrb}
\usepackage{enumerate}
\usepackage{color}
\usepackage{graphicx}

\DeclareMathOperator{\supp}{supp}
\def\endpf{\hfill\ensuremath{\square}\par}
\usepackage[hidelinks]{hyperref}
\hypersetup{
	pdftitle={Rough Singular Integrals and Marcinkiewicz Integrals on Product Spaces, with H1 Characterizations on Product Spheres},
	pdfauthor={Jiecheng Chen, Dashan Fan, Weitao Hu, and Meng Wang},
	pdfkeywords={product spaces, rough kernels, singular integral operators, Marcinkiewicz integrals, Hardy spaces}
}

\DeclareMathOperator*{\Liint}{{\mkern1mu\int\mkern-8mu\int\mkern-8mu}}

\def\ba{\begin{eqnarray}}
	\def\ea{\end{eqnarray}}

\def\R{\mathbb R}

\def\tv10{\widetilde{v_1^0}}
\def\tb10{\widetilde{b_1^0}}

\newcommand{\beq}{\begin{equation}}
	\newcommand{\eeq}{\end{equation}}
\newcommand{\ben}{\begin{eqnarray}}
	\newcommand{\een}{\end{eqnarray}}
\newcommand{\beno}{\begin{eqnarray*}}
	\newcommand{\eeno}{\end{eqnarray*}}

\newtheorem{Theorem}{Theorem}[section]

\newtheorem{Proposition}[Theorem]{Proposition}
\newtheorem{Lemma}[Theorem]{Lemma}
\newtheorem{Corollary}[Theorem]{Corollary}
\newtheorem{Remark}[Theorem]{Remark}

\numberwithin{equation}{section}

\begin{document}

	\title[Rough Integrals on Product Spaces]{%
		Rough Singular Integrals and Marcinkiewicz Integrals on Product Spaces,
		with $H^1$ Characterizations on Product Spheres
	}

	\author{Jiecheng Chen}
	\address{School of Mathematical Sciences,  Zhejiang Normal University, Jinhua, 321004, China}
	\email{jcchen@zjnu.edu.cn}

\author{Dashan Fan}
\address{School of Mathematical Sciences, Zhejiang Normal University,
	Jinhua 321004, China; Department of Mathematical Sciences,
	University of Wisconsin--Milwaukee, Milwaukee, WI 53211, USA}
\email{fan@uwm.edu}

	\author{Weitao Hu}
	\address{School of Mathematical Sciences,  Zhejiang Normal University, Jinhua, 321004, China}
	\email{huweitao@zjnu.edu.cn}

	\author{Meng Wang
	}
	\address{School of Mathematical Sciences, Zhejiang University, Hangzhou 310058, China}
	\email{mathdreamcn@zju.edu.cn}

	\date{September 26, 2026}

\begin{abstract}
We prove the $L^p(\mathbb R^n\times\mathbb R^m)$ boundedness of the rough
product singular integral $T_\Omega$, its maximal truncation $T_\Omega^*$,
and the product Marcinkiewicz integral $\mu_\Omega$ for $1<p<\infty$,
assuming that $\Omega\in H^1(S^{n-1}\times S^{m-1})$ satisfies separate
cancellation. These results include kernels that lie outside the Orlicz
classes required in earlier work.
For functions with separate cancellation, we also obtain equivalent
characterizations of the product Hardy space in terms of restricted,
conjugate, and spectral Riesz transforms, as well as radial maximal
functions, non-tangential maximal functions, and area integrals
associated with the spectral and ball Poisson extensions.	The
	operator bounds are proved using restricted Riesz transforms and the
	Calder\'on--Zygmund rotation method. For the Marcinkiewicz integral, we
	further establish integrable bounds for the radial profiles of the
	transformed kernels.

\end{abstract}
\maketitle
\tableofcontents

	\section{Introduction}\label{sec:introduction}

	Let $\R^l$ be the $l$-dimensional Euclidean space and $S^{l-1}$ its unit sphere, where $l\ge2$. Throughout the paper, $n,m\ge2$.  We denote by $z'$  the unit vector $z/|z|$ for $z\neq 0$. The Calder\'on-Zygmund singular integral operator and its maximal operator, with a homogeneous kernel on the product space $\R^n\times\R^m$, are defined by
	\begin{eqnarray}\label{operator}
		\left\{
		\begin{array}{l}
			T_{\Omega}(f)(x,y)=\operatorname{p.v.}\,\Liint_{\R^n\times\R^m}K(u,v)f(x-u,y-v)~dudv,\\
			T_{\Omega}^*(f)(x,y)=\sup_{\epsilon'>0,\epsilon''>0}\left|\Liint_{|u|>\epsilon',|v|>\epsilon''}K(u,v)f(x-u,y-v)dudv\right|,\\
			K(u,v)=\Omega(u',v')|u|^{-n}|v|^{-m},
		\end{array}\right.
	\end{eqnarray}
	where $\Omega\in L^1(S^{n-1}\times S^{m-1})$ satisfies the separate cancellation conditions
	\ba\label{cancellation}
\begin{aligned}
\int_{S^{n-1}}\Omega(x',y')\,dx'&=0
&&\text{for a.e. }y'\in S^{m-1},\\
\int_{S^{m-1}}\Omega(x',y')\,dy'&=0
&&\text{for a.e. }x'\in S^{n-1}.
\end{aligned}
	\ea
	For $1\le r\le +\infty$, we denote
	\begin{eqnarray}\label{def:l01}
		L_0^r(S^{n-1}\times S^{m-1})=\{\Omega\in L^r(S^{n-1}\times S^{m-1}):~\Omega~\mbox{satisfies~} (\ref{cancellation})\}.
	\end{eqnarray}
	The Marcinkiewicz integral operator on the product space $\R^n\times\R^m$ is defined by
	\ba\label{marcinkiewicz}
	\mu_{\Omega}(f)(x,y):=\left(\Liint_{(\R_{+}^{1})^2}|\psi_{t,s}^{\Omega}\ast f(x,y)|^2~\frac{dtds}{ts}\right)^{\frac{1}{2}},
	\ea
	where $\Omega\in L_0^1(S^{n-1}\times S^{m-1})$, $\psi_{t,s}^{\Omega}(x,y)=t^{-n}s^{-m}\psi^{\Omega}(\frac{x}{t},\frac{y}{s})$ and
	\ba\label{4}
	\psi^{\Omega}(x,y)=\Omega(x',y')|x|^{1-n}|y|^{1-m}\chi_{|x|\le 1}(x)\chi_{|y|\le 1}(y).
	\ea

	We use a Poisson maximal definition of \(H^1\), following the scalar
	manifold setting of \cite{CL,CWn2}. Let \({\mathcal H}\) be a Hilbert space,
	and let \(D\) be a
	complete Riemannian manifold of dimension \(l_D\) with non-negative Ricci
	curvature. We write \(L^1(D;{\mathcal H})\) for the Bochner \(L^1\)-space of
	\({\mathcal H}\)-valued functions on \(D\), equipped with the norm
	\[
	\|f\|_{L^1(D;{\mathcal H})}
	=
	\int_D \|f(x)\|_{\mathcal H}\,d_Dx.
	\]
	When \({\mathcal H}=\mathbb C\), we simply write \(L^1(D)\).

	Let \(\Delta_D\) denote the nonpositive Laplace--Beltrami operator on \(D\),
	\(\nabla_D\) its gradient, \(P_D=\{P_t^D\}_{t>0}\) its Poisson semigroup,
	\(\{p_t^D\}_{t>0}\) the corresponding Poisson kernels, and \(d_Dx\) the
	Riemannian volume element. For \(f\in L^1(D;{\mathcal H})\), define its radial
	maximal function by
	\[
	P_D^+(f)(x)
	=
	\sup_{t>0}\|P_t^D f(x)\|_{\mathcal H}.
	\]
	We say that \(f\in H^1(D;{\mathcal H})\) if
	\begin{eqnarray}\label{def:h1manifoldd}
		\|f\|_{H^1(D;{\mathcal H})}
		:=
		\|P_D^+(f)\|_{L^1(D)}
		<+\infty.
	\end{eqnarray}
	In the scalar case we write \(H^1(D)\). For mean-zero scalar data,
	this is the definition in \cite[p.~23]{CL}; on compact manifolds we also
	retain the constant functions.

	On \(\mathbb R^l\), this agrees with the classical Hardy space
	\cite{FS1}, and every \(H^1\) function has integral zero. On the sphere,
	it agrees with the atomic definition in \cite{Colzanithesis,Colzani},
	with constants included; see \cite[Theorem 1, p.~24]{CL} on the
	mean-zero subspace. Cancellation will be imposed explicitly when needed.

	Let \(N\) and \(M\) be two complete Riemannian manifolds with non-negative
	Ricci curvatures. For \(f\in L^1(N\times M)\), define
	\[
	P^+(f)(x,y)
	:=
	\sup_{t,s>0}
	\left|
	\left(P_t^N\otimes P_s^M\right)f(x,y)
	\right|.
	\]
	We define the scalar-valued product Hardy space \(H^1(N\times M)\) by
	\begin{eqnarray}\label{def:h1manifoldp}
		H^1(N\times M)
		:=
		\left\{
		f\in L^1(N\times M):
		\|f\|_{H^1(N\times M)}
		=
		\|P^+(f)\|_{L^1(N\times M)}
		<+\infty
		\right\}.
	\end{eqnarray}
	This product definition follows \cite{CWn1}. Hilbert-valued
one-parameter spaces will be used to iterate estimates; the product
space itself is scalar-valued. For the classical theory on the bidisc
and products of upper half-spaces, see
\cite{GunS,CF1,CF,Chenshuo,Merryfield,Sato}.

	In what follows, when the underlying space is clear, we write simply
	\(\|\cdot\|_1\) for the corresponding \(L^1\)-norm. In particular, on the
	product sphere \(S^{n-1}\times S^{m-1}\), the notation \(H^1\) means
	\(H^1(S^{n-1}\times S^{m-1})\). For finite-dimensional vector- or
	tensor-valued functions, \(|\cdot|\) denotes the standard Euclidean, respectively
	Hilbert--Schmidt, norm in the fiber, and the corresponding \(L^1\)-norms are
	taken with respect to these fiber norms.

	For simplicity, we shall write
	\[
	\Omega\in H^1\wedge \eqref{cancellation}
	\]
	to mean that
	\[
	\Omega\in H^1(S^{n-1}\times S^{m-1})
	\quad\text{and}\quad
	\Omega\ \text{satisfies}\ \eqref{cancellation}.
	\]

	Constants may depend on the dimensions, \(p\), and fixed cutoffs,
	but not on the functions or truncation parameters. We write
	\(A\lesssim B\) for \(A\le CB\), and \(A\cong B\) when both
	inequalities hold. Additional dependence is indicated by a subscript.

The classical theory of singular integrals and square functions begins
with \cite{CZ,Stein,ST0}. Product singular integrals and Fourier-transform
methods are developed in \cite{Fefferman,FS2,Duoandikoetxea,DF}.
We focus here on angular conditions that give boundedness for every
$1<p<\infty$.

In the non-product case, the condition
$\Omega\in H^1(S^{l-1})\cap L_0^1(S^{l-1})$ gives the $L^p$-boundedness
of $T_\Omega$ and $T_\Omega^*$ for every $1<p<\infty$, where
\begin{eqnarray}\label{con:cancellationr}
L_0^r(S^{l-1})
:=\left\{f\in L^r(S^{l-1}):\int_{S^{l-1}}f(x')\,dx'=0\right\},
\qquad 1\le r\le\infty.
\end{eqnarray}
For $T_\Omega$, this is due to Connett \cite{Connett} and
Ricci--Weiss \cite{RW}; see also \cite{1977CW,2000stefanove}.
For $T_\Omega^*$, see Fan--Pan \cite{FP} and Grafakos--Stefanov
\cite{2000GS}. These results and a different sufficient kernel condition
are discussed in \cite[Introduction]{GraS}. 
In the product case, the $L^p$-boundedness of $T_\Omega$ for
$1<p<\infty$ under the condition
\[
\Omega\in L(\ln^+L)^2(S^{n-1}\times S^{m-1})
\]
with separate cancellation was obtained in \cite{Chen,HMS}.
Under the same condition, the boundedness of both $T_\Omega$ and
$T_\Omega^*$ was proved in \cite{SQP}; see also \cite{Wang}.
Since $L(\ln^+L)^2(S^{n-1}\times S^{m-1})$ is properly contained in
the product Hardy space, we ask whether the same bounds hold for
$\Omega\in H^1$ satisfying \eqref{cancellation}.
The product atomic theory \cite{CF1,CF} involves more complicated
geometry than its one-parameter counterpart. We instead use restricted
Riesz transforms to control the kernels produced by rotations, following
the one-parameter approach of Ricci--Weiss \cite{RW}; see
Lemma \ref{lemma7}. This gives our first main result.
	\begin{Theorem}\label{thm1}
		Suppose $\Omega\in H^1(S^{n-1}\times S^{m-1})\wedge (\ref{cancellation})$. Then $T_{\Omega}$ and $T_{\Omega}^*$
		are
		bounded on $L^p(\mathbb R^n\times\mathbb R^m)$ for $1<p<+\infty$.
	\end{Theorem}

We prove Theorem \ref{thm1} in Section
\ref{Singular Integrals on Product spaces}. For $T_\Omega$, Lemma
\ref{lemma7} and the rotation method reduce the proof to directional
Hilbert transforms. For $T_\Omega^*$, we use the odd-odd functions
$\omega_{\alpha,\beta}$ in \eqref{def:omegaab} and estimate the
additional terms arising from the truncations. The principal terms are
controlled by maximal directional Hilbert transforms on $\mathbb R^2$.

For the non-product Marcinkiewicz integral, the mean-zero
$H^1(S^{l-1})$ kernel condition gives boundedness for every
$1<p<\infty$ \cite{DFP}; see \cite{XCY} for another proof.
The condition $\Omega\in L(\ln^+L)^{1/2}(S^{l-1})\cap L_0^1(S^{l-1})$
also suffices \cite{SQChP}; Walsh \cite{Walsh} established the $L^2$ case.
For related Marcinkiewicz operators, see \cite{CFP,Ding,A}.
In the product case, \cite[Theorem 2]{QSChP} gives the full $L^p$ range
for $L\ln^+L$ kernels with separate cancellation. Earlier developments
appear in \cite{CDF,CFY2001,CFY,Choi,YCF,CWn3}; the history of the
conjecture in \cite{Ding2001} is described in \cite[p.~228]{QSChP}.
Since $L\ln^+L$ does not contain the product Hardy space, we also prove
the following.
	\begin{Theorem}\label{thm2}
		Suppose $\Omega\in H^1(S^{n-1}\times S^{m-1})\wedge (\ref{cancellation})$. Then $\mu_{\Omega}$
		is
		bounded on $L^p(\mathbb R^n\times\mathbb R^m)$ for $1<p<+\infty$.
	\end{Theorem}

In Section \ref{Marcinkiewicz integrals on product spaces}, rotations
reduce Theorem \ref{thm2} to directional Littlewood--Paley functions.
The radial profiles of the transformed kernels are controlled by
$\|\Omega\|_{H^1}$ in Lemma \ref{lemma10}; see also
Remark \ref{rmk:4.5}.

Sections \ref{sec:h1one} and \ref{sec:product} compare the one-parameter
and product spherical Hardy-space quantities. These comparisons give
the following characterization, which supplies the restricted Riesz
estimates used in the two operator theorems.
	\begin{Theorem}\label{thm3} For every \(f\in L_0^1(S^{n-1}\times S^{m-1})\), one has
		\[
\begin{aligned}
\|f\|_{H^1}
&\cong\|P^*(f)\|_1\cong\|A(f)\|_1\\
&\cong\|(\mathfrak R_{S^{n-1}}'\otimes\mathfrak R_{S^{m-1}}')f\|_1\\
&\cong\|(\mathcal R_{S^{n-1}}'\otimes\mathcal R_{S^{m-1}}')f\|_1
\cong\|(\widetilde R_{S^{n-1}}'\otimes\widetilde R_{S^{m-1}}')f\|_1\\
&\cong\|C_+(f)\|_1\cong\|C_*(f)\|_1\cong\|S(f)\|_1.
\end{aligned}
\]
	\end{Theorem}

Here the spectral Poisson and Riesz quantities are defined in
\eqref{def:h1manifoldp}, \eqref{def:prona}, \eqref{def:maninar}, and
\eqref{def:rieszp}. The restricted and conjugate transforms are defined
in \eqref{def:widetilder} and \eqref{def:conjugate}; primes denote
augmentation by $Id$. The ball-Poisson quantities are defined in
\eqref{def:procs}.

We use \cite{chineseequivalence} for its
one-parameter ball-Poisson area--maximal equivalence for mean-zero data
and its product estimate
$\|S(f)\|_1\lesssim\|C_*(f)\|_1$
(Theorems 1.1 and 1.2 there, respectively). Under separate cancellation,
we prove the reverse product estimate and connect these ball-Poisson
quantities with the Poisson semigroup and the three Riesz transform
families in Theorem \ref{thm3}. The restricted Riesz transforms then
provide the kernel estimates used in Theorems \ref{thm1} and \ref{thm2}.
The rough singular integral, maximal truncation, and Marcinkiewicz
integral bounds in those theorems are not contained in
\cite{chineseequivalence}.

	We further note the following related works.
	\begin{Remark}
		(i)	Earlier treatments of multiparameter Hardy spaces and area--maximal
comparisons include \cite{Chenshuo,Sato}.

		(ii) For classical product atomic decompositions, see \cite{CF1,RFefferman}.
Spline characterizations are studied by Chang--Ciesielski \cite{CF0}.

		(iii) Product Hardy spaces associated with operators and with spaces
of homogeneous type are studied in \cite{CDLWY,HLL,HLPW}. Related
characterizations by Riesz transforms and maximal functions in the
Bessel setting are established in \cite{DLWY2021}. Our comparisons
concern the concrete spherical operators required for the rough-kernel
applications above.

		(iv) For spherical Hardy spaces, see Connett \cite{Connett},
Ricci--Weiss \cite{RW}, and Colzani \cite{Colzanithesis,Colzani}.
Related compact-manifold results appear in \cite{Colzani1988,CT1990}.

		(v) For Hardy spaces and area--maximal comparisons on manifolds
with non-negative Ricci curvature, see \cite{CL,CW,CWn1,CWn2}.
For BMO decompositions on normal Lie groups, see \cite{WC}.

	\end{Remark}

	\section{Proof of Theorem \ref{thm1}}\label{Singular Integrals on Product spaces}

Throughout this section, we identify homogeneous principal value distributions
with their associated convolution operators. Thus, for
$
\sigma\in L_0^1(S^{l-1}),
$
we use
\[
\operatorname{p.v.}\,\frac{\sigma(\cdot)}{|\cdot|^l}
\]
to denote both the principal value distribution and the convolution operator
with this kernel. With the same convention, the Dirac mass \(\delta_0\)
corresponds to the identity convolution operator, since \(\delta_0*f=f\).
All identities below involving such kernels are therefore understood either in
\({\mathcal S}'\) or as identities of convolution operators.

 For \(l\ge2\),
write
\begin{eqnarray}\label{def:rieszkernel}
K_l(z)
=
\gamma_l\frac{z}{|z|^{l+1}},
\qquad z\in\mathbb R^l\setminus\{0\},
\end{eqnarray}
 and \(\operatorname{p.v.}\,\,\,K_l\) denotes the corresponding vector-valued principal
value distribution. We denote by
\[
{\mathfrak R}_{\mathbb R^l}
=
(R_1,\ldots,R_l)
\]
the vector-valued Euclidean Riesz transform, namely
$
{\mathfrak R}_{\mathbb R^l}f
=
(\operatorname{p.v.}\,\,\,K_l)*f.
$
Here \(\gamma_l\) is the usual dimensional normalizing constant. Its precise value
is irrelevant for our estimates. Constants denoted by \(C_l\) or \(C_{n,m}\)
may change from line to line and depend only on the dimensions.

For any function \(\Omega\in L_0^1(S^{n-1}\times S^{m-1})\), we have the following
parity decomposition:
\ba\label{omegade}
\Omega=\Omega_{o,o}+\Omega_{o,e}+\Omega_{e,o}+\Omega_{e,e},
\ea
where the components are defined explicitly as
\ba\label{7}
\left\{
\begin{array}{l}
	\Omega_{o,o}(x',y')=\frac{1}{4}\left(\Omega(x',y')-\Omega(-x',y')-\Omega(x',-y')+\Omega(-x',-y')\right),\\
	\Omega_{o,e}(x',y')=\frac{1}{4}\left(\Omega(x',y')-\Omega(-x',y')+\Omega(x',-y')-\Omega(-x',-y')\right),\\
	\Omega_{e,o}(x',y')=\frac{1}{4}\left(\Omega(x',y')+\Omega(-x',y')-\Omega(x',-y')-\Omega(-x',-y')\right),\\
	\Omega_{e,e}(x',y')=\frac{1}{4}\left(\Omega(x',y')+\Omega(-x',y')+\Omega(x',-y')+\Omega(-x',-y')\right).
\end{array}\right.
\ea
Each component \(\Omega_{\alpha,\beta}\), \(\alpha,\beta\in\{o,e\}\), belongs to
\(L_0^1(S^{n-1}\times S^{m-1})\). Here the subscripts \(o\) and \(e\) denote
odd parity and even parity, respectively. For example, \(\Omega_{o,e}\) is
odd with respect to the first variable \(x'\) and even with respect to the
second variable \(y'\).
The purpose of the following transformation is to reduce every parity component
of the kernel to an odd-odd kernel, so that the rotation method can be applied.

	For \(\alpha,\beta\in\{o,e\}\), define
	\[
	K_{\alpha,\beta}(x,y)
	=
	\Omega_{\alpha,\beta}(x',y')|x|^{-n}|y|^{-m},
	\]
	and
	\[
	{\mathcal K}_{\alpha,\beta}(x,y)
	=
	\omega_{\alpha,\beta}(x',y')|x|^{-n}|y|^{-m},
	\]
	where
\begin{eqnarray}\label{def:omegaab}
	\omega_{\alpha,\beta}
	=
	\widetilde R_{S^{n-1}}^{\alpha}
	\otimes
	\widetilde R_{S^{m-1}}^{\beta}
	(\Omega_{\alpha,\beta}).
\end{eqnarray}
Here
\[
\widetilde R_{S^{l-1}}^{e}
=
\widetilde R_{S^{l-1}},
\qquad
\widetilde R_{S^{l-1}}^{o}
=
Id,
\]
where \(Id\) denotes the identity operator acting on functions on
\(S^{l-1}\), and \(\widetilde R_{S^{l-1}}\) is the restricted Riesz transform
on \(S^{l-1}\) defined in \cite{RW}; see also \eqref{def:widetilder} below.
Thus \(\omega_{o,o}\) is scalar-valued, \(\omega_{e,o}\) is
\(\mathbb C^n\)-valued, \(\omega_{o,e}\) is \(\mathbb C^m\)-valued, and
\(\omega_{e,e}\) is \(\mathbb C^n\otimes\mathbb C^m\)-valued.

According to the finite-dimensional fiber convention stated in the
introduction,
\[
\|\omega_{\alpha,\beta}\|_1
=
\int_{S^{n-1}\times S^{m-1}}
|\omega_{\alpha,\beta}(x',y')|
\,dx'dy'.
\]
The same convention will be used for \(\Omega_{e,e}'\), \(\Omega_{e,e}''\), and
all other finite-dimensional vector- or tensor-valued spherical functions.

	We define the corresponding operators by
	\[
	T_{\alpha,\beta}f
	=
	(\operatorname{p.v.}\,\,K_{\alpha,\beta})*f,
	\qquad
	{\mathcal T}_{\alpha,\beta}f
	=
	(\operatorname{p.v.}\,\,{\mathcal K}_{\alpha,\beta})*f.
	\]
Thus \(T_{\alpha,\beta}\) is the operator associated with the original parity
component \(\Omega_{\alpha,\beta}\), while
\({\mathcal T}_{\alpha,\beta}\) is the operator associated with the transformed
odd-odd spherical function \(\omega_{\alpha,\beta}\). The latter operator is the
one to which the rotation method will be applied.
 By the product identities recalled below, we first have
 \[
 \big(
 \mathfrak R_{\mathbb R^n}^{\alpha}
 \otimes
 \mathfrak R_{\mathbb R^m}^{\beta}
 \big)
 T_{\alpha,\beta}
 =
 {\mathcal T}_{\alpha,\beta}.
 \]Since
\[
\mathfrak R_{\mathbb R^l}\cdot \mathfrak R_{\mathbb R^l}
=
-Id,
\qquad l=n,m,
\]
we obtain
\[
T_{\alpha,\beta}
=
(-1)^{\mathbf{1}_{\alpha\neq\beta}}
\big(
\mathfrak R_{\mathbb R^n}^{\alpha}
\otimes
\mathfrak R_{\mathbb R^m}^{\beta}
\big)
\cdot
{\mathcal T}_{\alpha,\beta},
\]
where \(\mathbf{1}_{\alpha\neq\beta}=0\) if \(\alpha=\beta\) and
	\(\mathbf{1}_{\alpha\neq\beta}=1\) otherwise, and
	\[
	\mathfrak R_{\mathbb R^l}^{e}
	=
	\mathfrak R_{\mathbb R^l},
	\qquad
	\mathfrak R_{\mathbb R^l}^{o}
	=
	Id,
	\qquad l=n,m.
	\]
	The dot denotes the natural contraction when the corresponding Riesz
	transform is vector-valued.

In the even-even case, we shall also use the following partially transformed
spherical functions:
	\begin{eqnarray}\label{def:omega}
	\left\{
	\begin{array}{l}
		\Omega_{e,e}'(x',y')=
		\widetilde R_{S^{n-1}}(\Omega_{e,e}(\cdot,y'))(x'),\\
		\Omega_{e,e}''(x',y')=
		\widetilde R_{S^{m-1}}(\Omega_{e,e}(x',\cdot))(y').
	\end{array}
	\right.
\end{eqnarray}
Both \(\Omega_{e,e}'\) and \(\Omega_{e,e}''\) are vector-valued functions.

We now recall the restricted Riesz transform on the sphere.

For
$
\sigma\in C^\infty(S^{l-1})\cap L_0^1(S^{l-1}),
$ and for each \(j=1,\ldots,l\),
by the definition of the restricted Riesz transform, there exist a unique
function
\[
\sigma_j\in C^\infty(S^{l-1})\cap L_0^1(S^{l-1})
\]
and a unique constant \(a_j(\sigma)\in\mathbb C\) such that
\begin{eqnarray}\label{def:tildeR:dist}
	R_j\circ
	\left(
	\operatorname{p.v.}\,\,\frac{\sigma(\cdot)}{|\cdot|^l}
	\right)
	=
	a_j(\sigma)\delta_0
	+
	\operatorname{p.v.}\,\,\frac{\sigma_j(\cdot)}{|\cdot|^l}.
\end{eqnarray}
More explicitly, under the convention above, this means that, for every \(f\in{\mathcal S}(\mathbb R^l)\),
\[
R_j\left[
\left(
\operatorname{p.v.}\,\,\frac{\sigma(\cdot)}{|\cdot|^l}
\right)*f
\right]
=
a_j(\sigma)f
+
\left(
\operatorname{p.v.}\,\,\frac{\sigma_j(\cdot)}{|\cdot|^l}
\right)*f.
\]
Thus \eqref{def:tildeR:dist} is an identity of translation-invariant
operators, or equivalently of their distribution kernels.
We write
$$
\widetilde R_j\sigma=\sigma_j.
$$

In the applications below, \(\widetilde R_j\) is applied only to functions which
are even in the corresponding variable. In that case the parity implies
\(a_j(\sigma)=0\), and hence
\[
R_j\circ
\left(
\operatorname{p.v.}\,\,\frac{\sigma(\cdot)}{|\cdot|^l}
\right)
=
\operatorname{p.v.}\,\,\frac{\widetilde R_j\sigma(\cdot)}{|\cdot|^l}.
\]
The vector-valued restricted Riesz transform on \(S^{l-1}\) is defined by
\begin{eqnarray}\label{def:widetilder}
	\widetilde R_{S^{l-1}}f
	=
	(\widetilde R_1f,\ldots,\widetilde R_lf).
\end{eqnarray}

We shall use the preceding one-parameter identity in the product setting only
after the parity decomposition \eqref{7}. More precisely, the Euclidean Riesz
transform is applied only in those variables in which the corresponding parity
component is even, while the odd variables are left unchanged. Hence the
following identities hold in
\({\mathcal S}'(\mathbb R^n\times\mathbb R^m)\):
\begin{eqnarray}\label{transfor:rieszproduct}
	\left\{\begin{aligned}
		({\mathfrak R}_{\R^n}\otimes Id_{\R^m})
		\left(
		\operatorname{p.v.}\,\frac{\Omega_{e,o}(x',y')}{|x|^n|y|^m}
		\right)
		&=
		\operatorname{p.v.}\,\frac{\omega_{e,o}(x',y')}{|x|^n|y|^m},\\
		(Id_{\R^n}\otimes{\mathfrak R}_{\R^m})
		\left(
		\operatorname{p.v.}\,\frac{\Omega_{o,e}(x',y')}{|x|^n|y|^m}
		\right)
		&=
		\operatorname{p.v.}\,\frac{\omega_{o,e}(x',y')}{|x|^n|y|^m},\\
		({\mathfrak R}_{\R^n}\otimes{\mathfrak R}_{\R^m})
		\left(
		\operatorname{p.v.}\,\frac{\Omega_{e,e}(x',y')}{|x|^n|y|^m}
		\right)
		&=
		\operatorname{p.v.}\,\frac{\omega_{e,e}(x',y')}{|x|^n|y|^m}.
	\end{aligned}
	\right.
\end{eqnarray}
Moreover, in the even-even case we shall also use the corresponding partial
identities
\begin{eqnarray}\label{transfor:rieszpartial}
	\left\{\begin{aligned}
		({\mathfrak R}_{\R^n}\otimes Id_{\R^m})
		\left(
		\operatorname{p.v.}\,\frac{\Omega_{e,e}(x',y')}{|x|^n|y|^m}
		\right)
		&=
		\operatorname{p.v.}\,\frac{\Omega_{e,e}'(x',y')}{|x|^n|y|^m},\\
		(Id_{\R^n}\otimes{\mathfrak R}_{\R^m})
		\left(
		\operatorname{p.v.}\,\frac{\Omega_{e,e}(x',y')}{|x|^n|y|^m}
		\right)
		&=
		\operatorname{p.v.}\,\frac{\Omega_{e,e}''(x',y')}{|x|^n|y|^m}.
	\end{aligned}
	\right.
\end{eqnarray}
Consequently,
\[
(Id_{\R^n}\otimes{\mathfrak R}_{\R^m})
\left(
\operatorname{p.v.}\,\frac{\Omega_{e,e}'(x',y')}{|x|^n|y|^m}
\right)
=
\operatorname{p.v.}\,\frac{\omega_{e,e}(x',y')}{|x|^n|y|^m},
\]
and
\[
({\mathfrak R}_{\R^n}\otimes Id_{\R^m})
\left(
\operatorname{p.v.}\,\frac{\Omega_{e,e}''(x',y')}{|x|^n|y|^m}
\right)
=
\operatorname{p.v.}\,\frac{\omega_{e,e}(x',y')}{|x|^n|y|^m}.
\]

No additional \(\delta_0\)-terms appear in these identities, because each Riesz
transform is applied only in a variable with respect to which the corresponding
spherical function is even.
We shall often use these product distribution identities below to identify the
principal singular parts of the kernels.

	Thus, the following lemma is a corollary of Theorem \ref{thm3}.

	\begin{Lemma}\label{lemma7}
		Let
	$\Omega\in H^1(S^{n-1}\times S^{m-1})
		\wedge \eqref{cancellation}$.
		Then for every \(\alpha,\beta\in\{o,e\}\),
	$\omega_{\alpha,\beta}\in L^1(S^{n-1}\times S^{m-1})$ and
	$$
	\sum_{\alpha,\beta\in\{o,e\}}\|\omega_{\alpha,\beta}\|_1\le C_{n,m}\|\Omega\|_{H^1}.
	$$
Moreover, the partially transformed functions in the even-even case satisfy $\Omega_{e,e}'$, $\Omega_{e,e}''\in L^1(S^{n-1}\times S^{m-1})$, and
$$
\|\Omega_{e,e}'\|_1+\|\Omega_{e,e}''\|_1\le C_{n,m}\|\Omega\|_{H^1}.
$$
	\end{Lemma}
	
	{\bf Proof:}
	Although Theorem \ref{thm3} is proved later in Sections
	\ref{sec:h1one} and \ref{sec:product}, its proof is independent of the
	present section. Antipodal reflections commute with the product Poisson
	semigroup. Hence the parity projections in \eqref{7} preserve separate
	cancellation and satisfy
	\[
	\|\Omega_{\alpha,\beta}\|_{H^1}
	\le
	\|\Omega\|_{H^1},
	\qquad
	\alpha,\beta\in\{o,e\}.
	\]
	It follows from Theorem \ref{thm3} that
	\[
	\sum_{\alpha,\beta\in\{o,e\}}
	\left\|
	(\widetilde R_{S^{n-1}}'
	\otimes
	\widetilde R_{S^{m-1}}')
	\Omega_{\alpha,\beta}
	\right\|_1
	\le
	C_{n,m}\|\Omega\|_{H^1}.
	\]
	By their definitions, the functions
	\(\omega_{\alpha,\beta}\) occur among the corresponding identity,
	single-transform, or double-transform blocks. For
	\(\Omega_{e,e}\), the two single-transform blocks are
	\(\Omega_{e,e}'\) and \(\Omega_{e,e}''\). The asserted estimates now
	follow.
	\endpf


	{\bf Proof of the  boundedness of $T_{\Omega}$:}  It is enough to prove the estimate first for smooth $\Omega$ satisfying the cancellation condition. The general case follows by approximation.

	Noting that  ${\mathfrak R}_{\R^l}\cdot{\mathfrak R}_{\R^l}=-Id$,  $l=n,m$, and $({\mathfrak R}_{\R^n}\otimes {\mathfrak R}_{\R^m})\cdot({\mathfrak  R}_{\R^n}\otimes {\mathfrak R}_{\R^m})=Id$, we obtain, by \eqref{operator}, \eqref{omegade}, and \eqref{7},
	\begin{eqnarray*}
		T_{\Omega}(f)={\mathcal T}_{o,o}(f)-{\mathfrak R}_{\R^n}\cdot{\mathcal T}_{e,o}(f)-{\mathfrak R}_{\R^m}\cdot{\mathcal T}_{o,e}(f)+({\mathfrak R}_{\R^n}\otimes{\mathfrak R}_{\R^m})\cdot{\mathcal T}_{e,e}(f)
	\end{eqnarray*}
Here
$$
{\mathcal T}_{\alpha,\beta}(f)(x,y)
=
\operatorname{p.v.}\,\iint_{\mathbb R^n\times\mathbb R^m}
{\mathcal K}_{\alpha,\beta}(u,v)
f(x-u,y-v)\,du\,dv,
\qquad
\alpha,\beta\in\{o,e\}.
$$
Here and below, the $L^p$-norm of vector-valued or tensor-valued functions is understood with respect to the Euclidean norm of the fiber.

By the \(L^p\)-boundedness of the Riesz transforms, it is enough to prove the
\(L^p(\mathbb R^n\times\mathbb R^m)\)-boundedness of all
\({\mathcal T}_{\alpha,\beta}\), \(1<p<\infty\).
	{{By the definition of \(\omega_{\alpha,\beta}\), the function \(\omega_{\alpha,\beta}\) is odd in both variables. Indeed, if one of the original parity components is already odd in a variable, no restricted Riesz transform is applied in that variable; if it is even, the restricted Riesz transform changes its parity from even to odd. Hence \({\mathcal K}_{\alpha,\beta}\) is odd in both \(u\) and \(v\), and is homogeneous of degree \(-n\) in \(u\) and degree \(-m\) in \(v\). Therefore, by the rotation method,}	}
		\[
		{\mathcal T}_{\alpha,\beta}(f)(x,y)
		=
		\frac14
		\iint_{S^{n-1}\times S^{m-1}}
		\omega_{\alpha,\beta}(u',v')
		\big(H_{u'}^{\mathbb R^n}\otimes H_{v'}^{\mathbb R^m}\big)f(x,y)
		\,du'\,dv',
		\]
		where
		\[
		H_{u'}^{\mathbb R^n}f(x,y)
		=
		\operatorname{p.v.}\,\int_{\mathbb R}
		f(x-tu',y)\,\frac{dt}{t},
		\]
		and
		\[
		H_{v'}^{\mathbb R^m}f(x,y)
		=
		\operatorname{p.v.}\,\int_{\mathbb R}
		f(x,y-sv')\,\frac{ds}{s}
		\]
		are directional Hilbert transforms.

		It follows from Minkowski's inequality and the uniform \(L^p\)-boundedness of directional Hilbert transforms that \[ \|{\mathcal T}_{\alpha,\beta}(f)\|_{L^p(\mathbb R^n\times\mathbb R^m)} \le C_p \|\omega_{\alpha,\beta}\|_{L^1(S^{n-1}\times S^{m-1})} \|f\|_{L^p(\mathbb R^n\times\mathbb R^m)}. \] Since \(\omega_{\alpha,\beta}\in L^1(S^{n-1}\times S^{m-1})\) by Lemma \ref{lemma7},  each \({\mathcal T}_{\alpha,\beta}\) is bounded on \(L^p(\mathbb R^n\times\mathbb R^m)\).
	 Consequently, \(T_\Omega\) is bounded on \(L^p(\mathbb R^n\times\mathbb R^m)\): $$
	 \|T_{\Omega}(f)\|_{L^p(\R^n\times\R^m)}\le C_{p,n,m}\|\Omega\|_{H^1(S^{n-1}\times S^{m-1})}\|f\|_{L^p(\R^n\times\R^m)}.
	 $$

	{\bf Proof of the boundedness of $T_{\Omega}^{*}$:}

	It is easy to check that
	\ba\label{tomega*}
	T_{\Omega}^{*}(f)(x,y)
	&\le&
	\sum_{\alpha,\beta\in\{o,e\}}
	\bigg(
	T_{*}^{\alpha,\beta}(f)(x,y)
	+
	T_{*}^{\alpha,\beta,1}(f)(x,y)
	+
	T_{*}^{\alpha,\beta,2}(f)(x,y)
	\nonumber\\
	&&\qquad\qquad
	+
	16
	\Liint_{S^{n-1}\times S^{m-1}}
	|\Omega_{\alpha,\beta}(u',v')|
	M_{u',v'}(f)(x,y)
	\,du'dv'
	\bigg),
	\ea
	where
	\ba\label{13}
	\begin{aligned}
		M_{u',v'}(f)(x,y)
		&=
		\sup_{t,s>0}
		\frac{1}{ts}
		\Liint_{(0,t)\times(0,s)}
		|f(x-\delta u',y-\eta v')|
		\,d\delta d\eta,
		\\
		M_{u'}^{\R^n}(f)(x,y)
		&=
		\sup_{t>0}
		\frac{1}{t}
		\int_{0}^{t}
		|f(x-\delta u',y)|
		\,d\delta,
		\\
		M_{v'}^{\R^m}(f)(x,y)
		&=
		\sup_{s>0}
		\frac{1}{s}
		\int_{0}^{s}
		|f(x,y-\eta v')|
		\,d\eta .
	\end{aligned}
	\ea
	Moreover,
	\begin{eqnarray*}
		\begin{aligned}
			T_{*}^{\alpha,\beta}(f)(x,y)
			&=&
			\sup_{t,s>0}
			\left|
			\Liint_{\R^n\times\R^m}
			L_{t,s}^{\alpha,\beta}(u,v)
			f(x-u,y-v)
			\,dudv
			\right|,
			\\
			T_{*}^{\alpha,\beta,1}(f)(x,y)
			&=&
			\sup_{t,s>0}
			\left|
			\Liint_{\{|u|<t\}\times\R^m}
			L_{t,s}^{\alpha,\beta}(u,v)
			f(x-u,y-v)
			\,dudv
			\right|,
			\\
			T_{*}^{\alpha,\beta,2}(f)(x,y)
			&=&
			\sup_{t,s>0}
			\left|
			\Liint_{\R^n\times\{|v|<s\}}
			L_{t,s}^{\alpha,\beta}(u,v)
			f(x-u,y-v)
			\,dudv
			\right|.
		\end{aligned}
	\end{eqnarray*}
	Here
	\ba\label{operatorL}
	L_{t,s}^{\alpha,\beta}(x,y)
	=
	t^{-n}s^{-m}
	L^{\alpha,\beta}
	\left(
	\frac{x}{t},
	\frac{y}{s}
	\right),
	\ea
	where
	\[
	L^{\alpha,\beta}(x,y)
	=
	\Omega_{\alpha,\beta}(x',y')
	|x|^{-n}|y|^{-m}
	\widetilde{\chi}(|x|)
	\widetilde{\chi}(|y|).
	\]
	The cutoff function \(\widetilde{\chi}\) satisfies
	\ba\label{tildechi}
	\widetilde{\chi}\in C^{\infty}(0,+\infty),
	\qquad
	\widetilde{\chi}(r)=0\quad\text{for }r\le \frac12,
	\qquad
	\widetilde{\chi}(r)=1\quad\text{for }r\ge1,
	\qquad
	0\le \widetilde{\chi}\le1 .
	\ea

	We shall repeatedly use the following elementary cutoff estimates. Let
	\(l\ge2\), \(r>0\), and \(X,Z\in\mathbb R^l\). If \(K_l\) denotes the Riesz
	kernel in \(\mathbb R^l\), then the following estimates hold.

	First, if
	\[
	|X|\le \frac r4,
	\qquad
	\widetilde\chi\left(\frac{|Z|}{r}\right)\neq0,
	\]
	then \(|Z|\ge r/2\), \(|X-Z|\sim |Z|\), and
	\begin{equation*}\tag{E1}\label{est:E1}
		|K_l(X-Z)|
		\le
		C_l |Z|^{-l}.
	\end{equation*}

	Second, if
	\[
	|X|\ge 2r,
	\qquad
	1-\widetilde\chi\left(\frac{|Z|}{r}\right)\neq0,
	\]
	then \(|Z|\le r\), \(|X-Z|\sim |X|\), and
	\begin{equation*}\tag{E2}\label{est:E2}
		|K_l(X-Z)-K_l(X)|
		\le
		C_l\frac{|Z|}{|X|^{l+1}}.
	\end{equation*}

	Third, by the smoothness of \(\widetilde\chi\),
	\begin{equation*}\tag{E3}\label{est:E3}
		\left|
		\widetilde\chi\left(\frac{|Z|}{r}\right)
		-
		\widetilde\chi\left(\frac{|X|}{r}\right)
		\right|
		\le
		C\frac{|X-Z|}{r}.
	\end{equation*}

	Consequently, if
	\[
	\frac r4<|X|<2r,
	\qquad
	\frac r8\le |Z|\le 3r,
	\]
	then
	\begin{equation*}\tag{E4}\label{est:E4}
		\left|
		K_l(X-Z)
		\left[
		\widetilde\chi\left(\frac{|Z|}{r}\right)
		-
		\widetilde\chi\left(\frac{|X|}{r}\right)
		\right]
		\frac1{|Z|^l}
		\right|
		\le
		C_l
		\frac1{r^{l+1}|X-Z|^{l-1}}.
	\end{equation*}
	The portions with \(|Z|<r/8\) or \(|Z|>3r\) in the annular decompositions are
	smoother and will be dominated by the radial maximal kernels
	\(\Phi_r^{\mathbb R^l}\). These estimates will be used below with
	\((X,Z,r,l)=(x,z,t,n)\) or \((X,Z,r,l)=(y,w,s,m)\).

	We shall only deal with \(T_{*}^{\alpha,\beta}(f)\) below. The terms
	\(T_{*}^{\alpha,\beta,j}(f)\), \(j=1,2\), are local error terms and will be
	estimated at the end. Define
	\[
	{\mathcal T}_{*}^{\alpha,\beta}(f)(x,y)
	=
	\sup_{t,s>0}
	\left|
	{\mathcal L}_{t,s}^{\alpha,\beta}*f(x,y)
	\right|,
	\]
	where
	\[
	{\mathcal L}_{t,s}^{\alpha,\beta}(x,y)
	=
	t^{-n}s^{-m}
	{\mathcal L}^{\alpha,\beta}
	\left(
	\frac{x}{t},
	\frac{y}{s}
	\right),
	\]
	and
	\[
	{\mathcal L}^{\alpha,\beta}
	=
	\left(
	\mathfrak R_{\R^n}^{\alpha}
	\otimes
	\mathfrak R_{\R^m}^{\beta}
	\right)
	(L^{\alpha,\beta}).
	\]
	Here
	\[
	\mathfrak R_{\R^l}^{e}
	=
	\mathfrak R_{\R^l},
	\qquad
	\mathfrak R_{\R^l}^{o}
	=
	Id,
	\qquad l=n,m.
	\]
	Since
	\[
	\mathfrak R_{\R^l}\cdot \mathfrak R_{\R^l}
	=
	-Id,
	\qquad l=n,m,
	\]
	we have, up to harmless signs,
	\[
	L_{t,s}^{\alpha,\beta}*f
	=
	{\mathcal L}_{t,s}^{\alpha,\beta}
	*
	f^{\alpha,\beta},
	\]
	where
	\[
	f^{\alpha,\beta}
	=
	\left(
	\mathfrak R_{\R^n}^{\alpha}
	\otimes
	\mathfrak R_{\R^m}^{\beta}
	\right)f .
	\]
	Here the convolution on the right-hand side is understood with the natural
	vector or tensor contraction. Consequently,
	\[
	T_{*}^{\alpha,\beta}(f)(x,y)
	=
	{\mathcal T}_{*}^{\alpha,\beta}
	(f^{\alpha,\beta})(x,y).
	\]
	Therefore, by the \(L^p\)-boundedness of the Riesz transforms, it remains to
	prove the \(L^p\)-boundedness of
	\({\mathcal T}_{*}^{\alpha,\beta}\). In what follows we estimate
	\({\mathcal T}_{*}^{\alpha,\beta}(f)\).

	{\bf Case I:} \((\alpha,\beta)=(o,o)\).

	Since, by \eqref{def:omegaab},
	 \(\omega_{o,o}=\Omega_{o,o}\), by the rotation method we have
	\ba\label{17}
	{\mathcal T}_{*}^{o,o}(f)(x,y)
	&\le&
	\frac14
	\Liint_{S^{n-1}\times S^{m-1}}
	|\omega_{o,o}(u',v')|
	\widetilde H_{u',v'}^{*,0}(f)(x,y)
	\,du'dv'.
	\ea
	Here and in what follows, we introduce the auxiliary maximal operators defined below:
	\ba\label{18}
	\begin{aligned}
		H_{u',v'}^{*,0}(f)(x,y)
		&=
		\sup_{t,s>0}
		\left|
		H_{u',t}^{\R^n}
		\circ
		H_{v',s}^{\R^m}(f)(x,y)
		\right|,
		\\
		\widetilde H_{u',v'}^{*,0}(f)(x,y)
		&=
		\sup_{t,s>0}
		\left|
		\widetilde H_{u',t}^{\R^n}
		\circ
		\widetilde H_{v',s}^{\R^m}(f)(x,y)
		\right|,
		\\
		\widetilde H_{u',v'}^{*,1}(f)(x,y)
		&=
		\sup_{t,s>0}
		\left|
		\widetilde H_{u',t}^{\R^n}
		\circ
		H_{v',s}^{\R^m}(f)(x,y)
		\right|,
		\\
		\widetilde H_{u',v'}^{*,2}(f)(x,y)
		&=
		\sup_{t,s>0}
		\left|
		H_{u',t}^{\R^n}
		\circ
		\widetilde H_{v',s}^{\R^m}(f)(x,y)
		\right|.
		\end{aligned}
		\ea

	By iterating the one-parameter Cotlar inequality and using the
	commutation of operators acting in different ambient variables, the four
	maximal operators in \eqref{18} are bounded on
	\(L^p(\R^n\times\R^m)\), \(1<p<\infty\), uniformly in
	\(u'\in S^{n-1}\) and \(v'\in S^{m-1}\).

	The one-parameter directional operators are defined by
	\begin{eqnarray*}
		\widetilde H_{u',t}^{\R^n}(f)(x,y)
		=
		\int_{\R}
		f(x-\delta u',y)
		\widetilde\chi\left(\frac{|\delta|}{t}\right)
		\frac{d\delta}{\delta},\\
		H_{u',t}^{\R^n}(f)(x,y)
		=
		\int_{\R}
		f(x-\delta u',y)
		\chi\left(\frac{|\delta|}{t}\right)
		\frac{d\delta}{\delta},
		\\
		\widetilde H_{v',s}^{\R^m}(f)(x,y)
		=
		\int_{\R}
		f(x,y-\eta v')
		\widetilde\chi\left(\frac{|\eta|}{s}\right)
		\frac{d\eta}{\eta},\\
		H_{v',s}^{\R^m}(f)(x,y)
		=
		\int_{\R}
		f(x,y-\eta v')
		\chi\left(\frac{|\eta|}{s}\right)
		\frac{d\eta}{\eta}.
	\end{eqnarray*}
	Here \(\widetilde\chi\) is defined in \eqref{tildechi}, and
	\(\chi=\chi_{[1,+\infty)}\).

	{\bf Case II:} \((\alpha,\beta)=(o,e)\)
	and \((\alpha,\beta)=(e,o)\).
	We recall the definition for Riesz kernel $K_l(z)$ from \eqref{def:rieszkernel}.

	We first estimate \({\mathcal T}_{*}^{o,e}\).
In this case the main transformed kernel is
\[
{\mathcal K}_{o,e}(x,y)
=
\omega_{o,e}(x',y')|x|^{-n}|y|^{-m}.
\]
The error terms below still contain the original spherical function
\(\Omega_{o,e}\).

	By the definition of \({\mathcal L}_{t,s}^{o,e}\), we have
	\begin{eqnarray*}
		{\mathcal L}_{t,s}^{o,e}(x,y)
		&=&
		\widetilde{\chi}\left(\frac{|x|}{t}\right)
		\operatorname{p.v.}\,\int_{\R^m}
	K_{m}(y-w)
		\frac{\Omega_{o,e}(x',w')}{|x|^n|w|^m}
		\widetilde{\chi}\left(\frac{|w|}{s}\right)\,dw
		\\
		&=&
		\widetilde{\chi}\left(\frac{|x|}{t}\right)
		\bigg\{
		\chi_{\{|y|>2s\}}(y)
		\left(
		{\mathcal K}_{o,e}(x,y)
		+
		{\mathcal L}_{t,s}^{o,e,1}(x,y)
		\right)
		\\
		&&\qquad\qquad\qquad
		+
		\chi_{\{|y|<s/4\}}(y)
		{\mathcal L}_{t,s}^{o,e,2}(x,y)
		\\
		&&\qquad\qquad\quad
		+
		\chi_{\{s/4<|y|<2s\}}(y)
		\left(
		\widetilde{\chi}\left(\frac{|y|}{s}\right)
		{\mathcal K}_{o,e}(x,y)
		+
		{\mathcal L}_{t,s}^{o,e,3}(x,y)
		\right)
		\bigg\},
	\end{eqnarray*}
	where
	\begin{eqnarray*}
		{\mathcal L}_{t,s}^{o,e,1}(x,y)
		&=&
		\int_{\R^m}
		K_m(y-w)
		\left(
		\widetilde{\chi}\left(\frac{|w|}{s}\right)-1
		\right)
		\frac{\Omega_{o,e}(x',w')}{|x|^n|w|^m}
		\,dw,
		\\
		{\mathcal L}_{t,s}^{o,e,2}(x,y)
		&=&
		\int_{\R^m}
		K_m(y-w)
		\widetilde{\chi}\left(\frac{|w|}{s}\right)
		\frac{\Omega_{o,e}(x',w')}{|x|^n|w|^m}
		\,dw,
		\\
		{\mathcal L}_{t,s}^{o,e,3}(x,y)
		&=&
		\int_{\R^m}
	K_m(y-w)
		\frac{\Omega_{o,e}(x',w')}{|x|^n|w|^m}
		\left(
		\widetilde{\chi}\left(\frac{|w|}{s}\right)
		-
		\widetilde{\chi}\left(\frac{|y|}{s}\right)
		\right)
		\,dw.
	\end{eqnarray*}

	The terms containing \({\mathcal K}_{o,e}\) are controlled by the rotation
	method, while the remaining terms are treated as error terms. Thus
	\ba\label{19}
	{\mathcal T}_{*}^{o,e}(f)(x,y)
	&\le&
	\sup_{s,t>0}
	\frac14
	\left|
	\Liint_{S^{n-1}\times S^{m-1}}
	\omega_{o,e}(u',v')
	\left(
	\widetilde H_{u',t}^{\R^n}
	\circ
	H_{v',2s}^{\R^m}
	\right)f(x,y)
	\,du'dv'
	\right|
	\nonumber\\
	&&+
	\sup_{s,t>0}
	\frac12
	\left|
	\int_{S^{n-1}}
	\int_{|v|\ge 2s}
	{\mathcal L}_{t,s}^{o,e,1}(u',v)
	\widetilde H_{u',t}^{\R^n}f(x,y-v)
	\,dvdu'
	\right|
	\nonumber\\
	&&+
	\sup_{s,t>0}
	\frac12
	\left|
	\int_{S^{n-1}}
	\int_{|v|\le s/4}
	{\mathcal L}_{t,s}^{o,e,2}(u',v)
	\widetilde H_{u',t}^{\R^n}f(x,y-v)
	\,dvdu'
	\right|
	\nonumber\\
	&&+
	\sup_{s,t>0}
	\frac12
	\left|
	\Liint_{S^{n-1}\times S^{m-1}}
	\omega_{o,e}(u',v')
	\int_{s/4}^{2s}
	\widetilde H_{u',t}^{\R^n}f(x,y-\eta v')
	\widetilde{\chi}\left(\frac{\eta}{s}\right)
	\frac{d\eta}{\eta}
	\,dv'du'
	\right|
	\nonumber\\
	&&+
	\sup_{s,t>0}
	\frac12
	\left|
	\int_{S^{n-1}}
	\int_{s/4\le |v|\le 2s}
	{\mathcal L}_{t,s}^{o,e,3}(u',v)
	\widetilde H_{u',t}^{\R^n}f(x,y-v)
	\,dvdu'
	\right|
	\nonumber\\
	&:=&
	\sum_{j=0}^{4}
	{\mathcal T}_{*,j}^{o,e}(f)(x,y).
	\ea

	For the term ${\mathcal T}_{*,0}^{o,e}(f)$, we have
	\ba\label{20}
	{\mathcal T}_{*,0}^{o,e}(f)(x,y)
	\le
	C_{n,m}
	\Liint_{S^{n-1}\times S^{m-1}}
	|\omega_{o,e}(u',v')|
	\widetilde H_{u',v'}^{*,1}(f)(x,y)
	\,du'dv',
	\ea
	where \(\widetilde H_{u',v'}^{*,1}\) is defined in \eqref{18}.

	We shall also use the following one-parameter maximal directional Hilbert
	transforms:
	\begin{eqnarray}\label{22}
		\widetilde H_{u'}^{*,\R^n}(f)(x,y)
		&=&
		\sup_{t>0}
		\left|
		\widetilde H_{u',t}^{\R^n}(f)(x,y)
		\right|,
		\nonumber\\
		H_{u'}^{*,\R^n}(f)(x,y)
		&=&
		\sup_{t>0}
		\left|
		H_{u',t}^{\R^n}(f)(x,y)
		\right|,
		\nonumber\\
		\widetilde H_{v'}^{*,\R^m}(f)(x,y)
		&=&
		\sup_{s>0}
		\left|
		\widetilde H_{v',s}^{\R^m}(f)(x,y)
		\right|,
		\\
		H_{v'}^{*,\R^m}(f)(x,y)
		&=&
		\sup_{s>0}
		\left|
		H_{v',s}^{\R^m}(f)(x,y)
		\right|.
		\nonumber
	\end{eqnarray}
	Moreover, define
	$$
		\Phi^{\R^m}(y)
	=
	|y|^{-m+1}(1+|y|)^{-2},
	\qquad
	\Phi^{\R^n}(x)
	=
	|x|^{-n+1}(1+|x|)^{-2}.
	$$
	For $r>0$, set
	$$
	\Phi_{r}^{\R^m}(y)=r^{-m}\Phi^{\R^m}(\frac{y}{r}),\quad \Phi_{r}^{\R^n}(x)=r^{-n}\Phi^{\R^n}(\frac{x}{r}).
	$$
	Then define
	\ba\label{23}
	\begin{aligned}
		\Phi_{\R^m}^{*}(f)(x,y)
		&=
		\sup_{r>0}
		\left|
		\int_{\R^m}
		\Phi_{r}^{\R^m}(v)
		|f(x,y-v)|\,dv
		\right|,
		\\
		\Phi_{\R^n}^{*}(f)(x,y)
		&=
		\sup_{r>0}
		\left|
		\int_{\R^n}
		\Phi_{r}^{\R^n}(u)
		|f(x-u,y)|\,du
		\right|.
	\end{aligned}
	\ea

We shall also use translated versions of these maximal operators. Let
\[
A_n
=
\{u\in\R^n:1/8\le |u|\le 3\},
\qquad
A_m
=
\{v\in\R^m:1/8\le |v|\le 3\}.
\]
For \(u\in A_n\), define
\begin{eqnarray}\label{def:tmon}
\Phi_{u,\R^n}^{*}(f)(x,y)
:=
\sup_{r>0}
\int_{\R^n}
\Phi_{r}^{\R^n}(z)
|f(x-r u-z,y)|\,dz.
\end{eqnarray}
Similarly, for \(v\in A_m\), define
\begin{eqnarray}\label{def:tmom}
\Phi_{v,\R^m}^{*}(f)(x,y)
:=
\sup_{r>0}
\int_{\R^m}
\Phi_{r}^{\R^m}(w)
|f(x,y-r v-w)|\,dw.
\end{eqnarray}
These operators are uniformly bounded on $L^p$, $1<p<\infty$.
To see this in dimension $l$, fix $|u|\le3$ and put
\[
\mathcal A_{k,u}g(x)
=\sup_{r>0}\frac{1}{|B(0,2^{-k}r)|}
\int_{B(x-ru,2^{-k}r)}|g(z)|\,dz,
\qquad k\ge0.
\]
An annular decomposition of $\Phi^{\mathbb R^l}$ gives
\[
\Phi_{u,\mathbb R^l}^*g
\le C M g+C\sum_{k\ge0}2^{-k}\mathcal A_{k,u}g,
\]
where $M$ is the Hardy--Littlewood maximal operator.
Each averaging ball in $\mathcal A_{k,u}$ lies in a rectangle centered
at $x$, with length $O(r)$ in the direction of $u$ and transverse
width $O(2^{-k}r)$. The volume ratio is at most $C2^k$.
The maximal operator over dilates of this fixed rectangle is weak
$(1,1)$, uniformly under linear changes of coordinates. Hence
$\|\mathcal A_{k,u}\|_{L^1\to L^{1,\infty}}\le C2^k$.
Since its $L^\infty$ norm is at most $1$, interpolation gives
$\|\mathcal A_{k,u}\|_{L^p\to L^p}\le C_p2^{k/p}$.
Summing $2^{-k(1-1/p)}$ proves the asserted uniform bounds.

Now we estimate \({\mathcal T}_{*,1}^{o,e}(f)\). Since \(|v|\ge 2s\), and
since
$
1-\widetilde\chi\left(\frac{|w|}{s}\right)\neq0
$ implies that
$|w|\le s,
$
we may use the cancellation of \(\Omega_{o,e}\) in the second variable and
\eqref{est:E2}, with \((X,Z,r,l)=(v,w,s,m)\). Thus, we have
	\begin{eqnarray*}
		\left|
		{\mathcal L}_{t,s}^{o,e,1}(u',v)
		\right|
		&=&
		\left|
		\int_{\R^m}
		\left(
		K_m(v-w)-K_m(v)
		\right)
		\left(
		\widetilde{\chi}\left(\frac{|w|}{s}\right)-1
		\right)
		\frac{\Omega_{o,e}(u',w')}{|w|^m}
		\,dw
		\right|
		\\
		&\le&
		C_m
		\int_{|w|\le s}
		\frac{1}{|v|^{m+1}}
		\frac{|\Omega_{o,e}(u',w')|}{|w|^{m-1}}
		\,dw
		\\
		&\le&
		C_m
		\Phi_{s}^{\R^m}(v)
		\int_{S^{m-1}}
		|\Omega_{o,e}(u',w')|
		\,dw',
	\end{eqnarray*}
	as when $|v|\ge 2s$,
	$$
	\Phi_{s}^{\R^m}(v)=s^{-m}\Phi^{\R^m}(\frac{v}{s})\sim \frac{s}{|v|^{m+1}}.
	$$
	Hence
	\ba\label{21}
	{\mathcal T}_{*,1}^{o,e}(f)(x,y)
	\le
	C_{n,m}
	\Liint_{S^{n-1}\times S^{m-1}}
	|\Omega_{o,e}(u',v')|
	\left(
	\Phi_{\R^m}^{*}\circ
	\widetilde H_{u'}^{*,\R^n}
	\right)(f)(x,y)
	\,du'dv'.
	\ea

For \({\mathcal T}_{*,2}^{o,e}(f)\), when \(|v|\le s/4\), the support
condition
$
\widetilde\chi\left(\frac{|w|}{s}\right)\neq0
$
implies \(|w|\ge s/2\). Hence by \eqref{est:E1}, with
\((X,Z,r,l)=(v,w,s,m)\), we have
$
|v-w|\sim |w|.
$
Therefore
\[
\begin{aligned}
	\left|
	{\mathcal L}_{t,s}^{o,e,2}(u',v)
	\right|
	&\le
	C_m
	\int_{|w|\ge s/2}
	\frac{|\Omega_{o,e}(u',w')|}{|w|^{2m}}
	\,dw  \\
	&\le
	C_m s^{-m}
	\int_{S^{m-1}}
	|\Omega_{o,e}(u',w')|
	\,dw'.
\end{aligned}
\]
Moreover,
\[
s^{-m}\chi_{\{|v|\le s/4\}}(v)
\lesssim
\Phi_s^{\mathbb R^m}(v).
\]
Consequently,
\begin{eqnarray}\label{24}
	{\mathcal T}_{*,2}^{o,e}(f)(x,y)
	&\le&
	C_{n,m}
	\Liint_{S^{n-1}\times S^{m-1}}
	|\Omega_{o,e}(u',v')|
	\left(
	\Phi_{\mathbb R^m}^{*}\circ
	\widetilde H_{u'}^{*,\mathbb R^n}
	\right)(f)(x,y)
	\,du'dv'.
\end{eqnarray}

	For ${\mathcal T}_{*,3}^{o,e}(f)$, it is obvious that
	\ba\label{25}
	{\mathcal T}_{*,3}^{o,e}(f)(x,y)
	\le
	C_{n,m}
	\Liint_{S^{n-1}\times S^{m-1}}
	|\omega_{o,e}(u',v')|
	\left(
	M_{v'}^{\R^m}\circ
	\widetilde H_{u'}^{*,\R^n}
	\right)(f)(x,y)
	\,du'dv',
	\ea
	where \(M_{v'}^{\R^m}\) is defined in \eqref{13}.

For ${\mathcal T}_{*,4}^{o,e}$, split the $w$-integration into
$|w|<s/8$, $s/8\le |w|\le3s$, and $|w|>3s$.
In the first region, subtract $K_m(v)$ using cancellation; in the
last region use $|v-w|\cong|w|$. For $s/4\le |v|\le2s$, these two
parts are bounded by
\[
C_m\Phi_s^{\mathbb R^m}(v)
\int_{S^{m-1}}|\Omega_{o,e}(u',\eta)|\,d\eta.
\]
For the annular part, \eqref{est:E4} gives
\[
C_m\int_{s/8\le |w|\le3s}
\frac{|\Omega_{o,e}(u',w')|}
{s^{m+1}|v-w|^{m-1}}\,dw
\le C_m\int_{A_m}|\Omega_{o,e}(u',z')|
\Phi_s^{\mathbb R^m}(v-sz)\,dz.
\]
Here we put $w=sz$ and used $|v-sz|\le5s$ on the region in question.
Convolving with $f$ and taking the supremum yields
\ba\label{26}
{\mathcal T}_{*,4}^{o,e}(f)(x,y)
&\le& C_{n,m}\Liint_{S^{n-1}\times S^{m-1}}
|\Omega_{o,e}(u',v')|
\left(\Phi_{\mathbb R^m}^{*}\circ
\widetilde H_{u'}^{*,\mathbb R^n}\right)(f)(x,y)\,du'dv'
\nonumber\\
&&+C_{n,m}\Liint_{S^{n-1}\times A_m}
|\Omega_{o,e}(u',z')|
\left(\Phi_{z,\mathbb R^m}^{*}\circ
\widetilde H_{u'}^{*,\mathbb R^n}\right)(f)(x,y)\,du'dz.
\ea

	Combining \eqref{19}--\eqref{20} and \eqref{21}--\eqref{26}, we obtain
	\ba\label{27}
	{\mathcal T}_{*}^{o,e}(f)(x,y)
	&\le&
	C_{n,m}
	\Liint_{S^{n-1}\times S^{m-1}}
	|\omega_{o,e}(u',v')|
	\widetilde H_{u',v'}^{*,1}(f)(x,y)
	\,du'dv'
	\nonumber\\
	&&+
	C_{n,m}
	\Liint_{S^{n-1}\times S^{m-1}}
	\left(
	|\Omega_{o,e}|+|\omega_{o,e}|
	\right)(u',v')
	\nonumber\\
	&&\qquad\qquad\cdot
	\left(
	(\Phi_{\R^m}^{*}+M_{v'}^{\R^m})
	\circ
	\widetilde H_{u'}^{*,\R^n}
	\right)(f)(x,y)
	\,du'dv'
\nonumber\\
&&+C_{n,m}\Liint_{S^{n-1}\times A_m}
|\Omega_{o,e}(u',z')|
\left(\Phi_{z,\mathbb R^m}^{*}\circ
\widetilde H_{u'}^{*,\mathbb R^n}\right)(f)(x,y)\,du'dz.
	\ea

	Similarly, for \((\alpha,\beta)=(e,o)\),
	the symmetric estimate is
	\ba\label{28}
	{\mathcal T}_{*}^{e,o}(f)(x,y)
	&\le&
	C_{n,m}
	\Liint_{S^{n-1}\times S^{m-1}}
	|\omega_{e,o}(u',v')|
	\widetilde H_{u',v'}^{*,2}(f)(x,y)
	\,du'dv'
	\nonumber\\
	&&+
	C_{n,m}
	\Liint_{S^{n-1}\times S^{m-1}}
	\left(
	|\Omega_{e,o}|+|\omega_{e,o}|
	\right)(u',v')
	\nonumber\\
	&&\qquad\qquad\cdot
	\left(
	(\Phi_{\R^n}^{*}+M_{u'}^{\R^n})
	\circ
	\widetilde H_{v'}^{*,\R^m}
	\right)(f)(x,y)
	\,du'dv'
\nonumber\\
&&+C_{n,m}\Liint_{A_n\times S^{m-1}}
|\Omega_{e,o}(z',v')|
\left(\Phi_{z,\mathbb R^n}^{*}\circ
\widetilde H_{v'}^{*,\mathbb R^m}\right)(f)(x,y)\,dzdv'.
	\ea

	{\bf Case III:} \((\alpha,\beta)=(e,e)\).

In this case both variables are transformed. We use the notation
\(\omega_{e,e}\), \(\Omega_{e,e}'\),
\(\Omega_{e,e}''\) ( see \eqref{def:omegaab}, \eqref{def:omega}) and \({\mathcal K}_{e,e}=\omega_{e,e}(x',y')|x|^{-n}|y|^{-m}\) introduced above.
Thus \(\Omega_{e,e}'\) corresponds to transforming only the first variable,
	\(\Omega_{e,e}''\) corresponds to transforming only the second variable, and
	\(\omega_{e,e}\) corresponds to transforming both variables.

	For \(l=n,m\), let
	\[
	E_t^{\R^l,0}
	=
	\left\{z\in\R^l: |z|\le \frac{t}{4}\right\},
	\qquad
	E_t^{\R^l,1}
	=
	\left\{z\in\R^l: \frac{t}{4}< |z|<2t\right\},
	\]
	and
	\[
	E_t^{\R^l,\infty}
	=
	\left\{z\in\R^l: |z|\ge 2t\right\}.
	\]
	For \(i,j\in\{0,1,\infty\}\), set
	\[
	E_{t,s}^{i,j}
	=
	E_t^{\R^n,i}\times E_s^{\R^m,j},
	\]
	and define
	\[
	{\mathcal L}_{i,j,t,s}^{e,e}(x,y)
	=
	\chi_{E_{t,s}^{i,j}}(x,y)
	{\mathcal L}_{t,s}^{e,e}(x,y).
	\]
	Then
	\[
	{\mathcal L}_{t,s}^{e,e}
	=
	\sum_{i,j\in\{0,1,\infty\}}
	{\mathcal L}_{i,j,t,s}^{e,e},
	\]
	and consequently
	\ba\label{29}
	{\mathcal T}_{*}^{e,e}(f)(x,y)
	&\le&
	\sum_{i,j\in\{0,1,\infty\}}
	\sup_{t,s>0}
	\left|
	{\mathcal L}_{i,j,t,s}^{e,e}*f(x,y)
	\right|
	\nonumber\\
	&:=&
	\sum_{i,j\in\{0,1,\infty\}}
	{\mathcal T}_{i,j,*}^{e,e}(f)(x,y).
	\ea

For the term
${\mathcal L}_{\infty,\infty,t,s}^{e,e}$, $|x|>2t$, $|y|>2s$, we have
\begin{eqnarray*}
	{\mathcal L}_{\infty,\infty,t,s}^{e,e}(x,y)
	&=&
	{\mathcal K}_{e,e}(x,y)
	\\
	&&-
	|y|^{-m}
	\int_{\R^n}
	\left(
	K_n(x-z)-K_n(x)
	\right)
	\frac{\Omega_{e,e}''(z',y')}{|z|^n}
	\left(1-\widetilde{\chi}\left(\frac{|z|}{t}\right)\right)\,dz
	\\
	&&-
	|x|^{-n}
	\int_{\R^m}
	\left(
	K_m(y-w)-K_m(y)
	\right)
	\frac{\Omega_{e,e}'(x',w')}{|w|^m}
	\left(1-\widetilde{\chi}\left(\frac{|w|}{s}\right)\right)\,dw
	\\
	&&+
	\Liint_{\R^n\times\R^m}
	\left(
	K_n(x-z)-K_n(x)
	\right)
	\left(
	K_m(y-w)-K_m(y)
	\right)
	\\
	&&\qquad\qquad\qquad\cdot
	\frac{\Omega_{e,e}(z',w')}{|z|^n|w|^m}
\left(1-\widetilde{\chi}\left(\frac{|z|}{t}\right)\right)\left(1-\widetilde{\chi}\left(\frac{|w|}{s}\right)\right)\,dzdw .
\end{eqnarray*}
The error terms are estimated by using \eqref{est:E2} in the variables where
\(1-\widetilde\chi\) appears. More precisely, the \(z\)-error is controlled by
\eqref{est:E2} with \((X,Z,r,l)=(x,z,t,n)\), the \(w\)-error by
\eqref{est:E2} with \((X,Z,r,l)=(y,w,s,m)\), and the double error by applying
\eqref{est:E2} in both variables. Therefore
we obtain
\begin{eqnarray}\label{teeiis}
	{\mathcal T}_{\infty,\infty,*}^{e,e}(f)(x,y)
	&\le&
	C_{n,m}
	\Liint_{S^{n-1}\times S^{m-1}}
	|\omega_{e,e}(u',v')|
H_{u',v'}^{*,0}(f)(x,y)
	\,du'dv'
\nonumber	\\
	&&+
	C_{n,m}
	\int_{S^{m-1}}
	\left\|
	\Omega_{e,e}''(\cdot,v')
	\right\|_{L^1(S^{n-1})}
	\left(
	\Phi_{\R^n}^{*}\circ H_{v'}^{*,\R^m}
	\right)(f)(x,y)
	\,dv'
\nonumber	\\
	&&+
	C_{n,m}
	\int_{S^{n-1}}
	\left\|
	\Omega_{e,e}'(u',\cdot)
	\right\|_{L^1(S^{m-1})}
	\left(
	\Phi_{\R^m}^{*}\circ H_{u'}^{*,\R^n}
	\right)(f)(x,y)
	\,du'
\nonumber	\\
	&&+
	C_{n,m}
	\|\Omega_{e,e}\|_{L^1(S^{n-1}\times S^{m-1})}
	\left(
	\Phi_{\R^m}^{*}\circ \Phi_{\R^n}^{*}
	\right)(f)(x,y),
\end{eqnarray}
where the maximal operators appearing on the right-hand side are defined in
\eqref{18}, \eqref{22}, and \eqref{23}.

For the term \({\mathcal L}_{0,0,t,s}^{e,e}\), \(|x|\le t/4\) and
\(|y|\le s/4\). By \eqref{est:E1} in both variables, we have
$
|x-z|\sim |z|,
|y-w|\sim |w|,
$
whenever
$
\widetilde\chi\left(\frac{|z|}{t}\right)
\widetilde\chi\left(\frac{|w|}{s}\right)\neq0.
$
Thus
\begin{eqnarray*}
\left|
{\mathcal L}_{0,0,t,s}^{e,e}(x,y)
\right|
&\le&
C_{n,m}
\Liint_{|z|\ge t/2,\ |w|\ge s/2}
\frac{|\Omega_{e,e}(z',w')|}{|z|^{2n}|w|^{2m}}
\,dzdw \\
&
\le&
C_{n,m}
t^{-n}s^{-m}
\|\Omega_{e,e}\|_{L^1(S^{n-1}\times S^{m-1})}.
\end{eqnarray*}
Therefore
\begin{eqnarray}\label{es:ee00s}
	{\mathcal T}_{0,0,*}^{e,e}(f)(x,y)
	\le
	C_{n,m}
	\|\Omega_{e,e}\|_{L^1(S^{n-1}\times S^{m-1})}
	\left(
	\Phi_{\R^m}^{*}\circ \Phi_{\R^n}^{*}
	\right)(f)(x,y).
\end{eqnarray}

For \({\mathcal L}_{\infty,0,t,s}^{e,e}\),
\(|x|\ge 2t\) and \(|y|<s/4\), we have
\[
{\mathcal L}_{\infty,0,t,s}^{e,e}
=
I_{\infty,0,t,s}^{1}
-
I_{\infty,0,t,s}^{2},
\]
where
\begin{eqnarray*}
	I_{\infty,0,t,s}^{1}(x,y)
	&=&
	|x|^{-n}
	\int_{\R^m}
	K_m(y-w)
	\frac{\Omega_{e,e}'(x',w')}{|w|^m}
	\widetilde{\chi}\left(\frac{|w|}{s}\right)
	\,dw,
\end{eqnarray*}
and
\begin{eqnarray*}
	I_{\infty,0,t,s}^{2}(x,y)
	&=&
	\Liint_{\R^n\times\R^m}
	\left(
	K_n(x-z)-K_n(x)
	\right)
K_m(y-w)
	\\
	&&\qquad\qquad\cdot
	\frac{\Omega_{e,e}(z',w')}{|z|^n|w|^m}
\left(1-\widetilde{\chi}\left(\frac{|z|}{t}\right)\right)
	\widetilde{\chi}\left(\frac{|w|}{s}\right)
	\,dzdw .
\end{eqnarray*}
By \eqref{est:E1}, with \((X,Z,r,l)=(y,w,s,m)\), we have
$
|y-w|\sim |w|
$
on the support of \(\widetilde\chi(|w|/s)\). Moreover, by \eqref{est:E2}, with
\((X,Z,r,l)=(x,z,t,n)\), the \(z\)-error satisfies
\[
|K_n(x-z)-K_n(x)|
\le
C_n\frac{|z|}{|x|^{n+1}}.
\]
Therefore, the estimates
\[
|I_{\infty,0,t,s}^{1}(x,y)|
\le
C_{n,m}
\int_{|w|\ge s/2}
\frac{|\Omega_{e,e}'(x',w')|}
{|x|^n|w|^{2m}}
\,dw,
\]
and\[
|I_{\infty,0,t,s}^{2}(x,y)|
\le
C_{n,m}
\Liint_{|z|\le t,\ |w|\ge s/2}
\frac{|\Omega_{e,e}(z',w')|}
{|x|^{n+1}|z|^{n-1}|w|^{2m}}
\,dzdw
\]
hold.
Consequently,
\begin{eqnarray*}
	\left|
	{\mathcal L}_{\infty,0,t,s}^{e,e}(x,y)
	\right|
	&\le&
	C_{n,m}
	\int_{|w|\ge s/2}
	\frac{|\Omega_{e,e}'(x',w')|}
	{|x|^n|w|^{2m}}
	\,dw
	\\
	&&+
	C_{n,m}
	\Liint_{|z|\le t,\ |w|\ge s/2}
	\frac{|\Omega_{e,e}(z',w')|}
	{|x|^{n+1}|z|^{n-1}|w|^{2m}}
	\, dzdw.
\end{eqnarray*}
It follows from the preceding pointwise estimate that
\begin{eqnarray}\label{es:eei0s}
	{\mathcal T}_{\infty,0,*}^{e,e}(f)(x,y)
	&\le&
	C_{n,m}
	\int_{S^{n-1}}
	\left\|
	\Omega_{e,e}'(u',\cdot)
	\right\|_{L^1(S^{m-1})}
	\left(
	\Phi_{\R^m}^{*}\circ H_{u'}^{*,\R^n}
	\right)(f)(x,y)
	\,du'\nonumber
	\\
	&&+
	C_{n,m}
	\|\Omega_{e,e}\|_{L^1(S^{n-1}\times S^{m-1})}
	\left(
	\Phi_{\R^m}^{*}\circ \Phi_{\R^n}^{*}
	\right)(f)(x,y).
\end{eqnarray}

The estimate for \({\mathcal L}_{0,\infty,t,s}^{e,e}\) is symmetric. We have
\begin{eqnarray}\label{es:ee0is}
	{\mathcal T}_{0,\infty,*}^{e,e}(f)(x,y)
	&\le&
	C_{n,m}
	\int_{S^{m-1}}
	\left\|
	\Omega_{e,e}''(\cdot,v')
	\right\|_{L^1(S^{n-1})}
	\left(
	\Phi_{\R^n}^{*}\circ H_{v'}^{*,\R^m}
	\right)(f)(x,y)
	\,dv'\nonumber
	\\
	&&+
	C_{n,m}
	\|\Omega_{e,e}\|_{L^1(S^{n-1}\times S^{m-1})}
	\left(
	\Phi_{\R^m}^{*}\circ \Phi_{\R^n}^{*}
	\right)(f)(x,y).
\end{eqnarray}

For \({\mathcal L}_{1,0,t,s}^{e,e}\), where
\(t/4<|x|<2t\) and \(|y|<s/4\). In this region we compare
\(\widetilde\chi(|z|/t)\) with \(\widetilde\chi(|x|/t)\). Accordingly, we write
\beno
&&{\mathcal L}_{1,0,t,s}^{e,e}(x,y)\\&=&\Liint_{\R^n\times\R^m}\frac{\gamma_n(x-z)}{|x-z|^{n+1}}
\widetilde{\chi}(\left|\frac{z}{t}\right|)\frac{\gamma_m(y-w)}{|y-w|^{m+1}}
\frac{\Omega_{e,e}(z',w')}{|z|^n|w|^m}\widetilde{\chi}(\left|\frac{w}{s}\right|)dzdw\\
&=&	J_{1,0,t,s}^{1}
+
J_{1,0,t,s}^{2}
+
J_{1,0,t,s}^{3}.
\eeno
Here the main term is
\begin{eqnarray*}
	J_{1,0,t,s}^{1}(x,y)
	&=&
	\widetilde{\chi}\left(\frac{|x|}{t}\right)
	|x|^{-n}
	\int_{\R^m}
	K_m(y-w)
	\frac{\Omega_{e,e}'(x',w')}{|w|^m}
	\widetilde{\chi}\left(\frac{|w|}{s}\right)
	\,dw,
\end{eqnarray*}
the annular cutoff error is
\begin{eqnarray*}
	J_{1,0,t,s}^{2}(x,y)
	&=&
	\Liint_{\{t/8\le |z|\le 3t\}\times\R^m}
	K_n(x-z)
	\left[
	\widetilde{\chi}\left(\frac{|z|}{t}\right)
	-
	\widetilde{\chi}\left(\frac{|x|}{t}\right)
	\right]
	\\
	&&\qquad\qquad\cdot
	K_m(y-w)
	\frac{\Omega_{e,e}(z',w')}{|z|^n|w|^m}
	\widetilde{\chi}\left(\frac{|w|}{s}\right)
	\,dzdw,
\end{eqnarray*}
and \(J_{1,0,t,s}^{3}\) denotes the remaining smoother terms, corresponding to
\(|z|<t/8\) and \(|z|>3t\).

By \eqref{est:E1}, with \((X,Z,r,l)=(y,w,s,m)\), we have
$
|y-w|\sim |w|
$
on the support of \(\widetilde\chi(|w|/s)\). Since $t/2<|x|<2t$ on the support of the main term, its
radial integral is bounded by a directional maximal average. Thus
\[
\sup_{t,s>0}
\left|
J_{1,0,t,s}^{1}*f(x,y)
\right|
\le
C_{n,m}
\int_{S^{n-1}}
\left\|
\Omega_{e,e}'(u',\cdot)
\right\|_{L^1(S^{m-1})}
\left(
\Phi_{\R^m}^{*}\circ
M_{u'}^{\R^n}
\right)(f)(x,y)
\,du' .
\]

For \(J_{1,0,t,s}^{2}\), applying the annular cutoff estimate
\eqref{est:E4} in the \(x\)-variable and \eqref{est:E1} in the \(y\)-variable
gives
\[
\left|
J_{1,0,t,s}^{2}(x,y)
\right|
\le
C_{n,m}
\Liint_{\{t/8\le |z|\le3t,\ |w|\ge s/2\}}
\frac{|\Omega_{e,e}(z',w')|}
{t^{n+1}|x-z|^{n-1}|w|^{2m}}
\,dzdw .
\]
Putting \(z=tu\), \(u\in A_n\), and integrating in the radial variable of
\(w\), we obtain
\[
\left|
J_{1,0,t,s}^{2}(x,y)
\right|
\le
C_{n,m}
\int_{A_n}
\left\|
\Omega_{e,e}(u',\cdot)
\right\|_{L^1(S^{m-1})}
\Phi_t^{\R^n}(x-tu)
\Phi_s^{\R^m}(y)
\,du .
\]
Consequently,
\[
\sup_{t,s>0}
\left|
J_{1,0,t,s}^{2}*f(x,y)
\right|
\le
C_{n,m}
\int_{A_n}
\left\|
\Omega_{e,e}(u',\cdot)
\right\|_{L^1(S^{m-1})}
\left(
\Phi_{\R^m}^{*}\circ
\Phi_{u,\R^n}^{*}
\right)(f)(x,y)
\,du .
\]

The smoother term satisfies
\[
\left|
J_{1,0,t,s}^{3}(x,y)
\right|
\le
C_{n,m}
\|\Omega_{e,e}\|_{L^1(S^{n-1}\times S^{m-1})}
\Phi_t^{\R^n}(x)
\Phi_s^{\R^m}(y).
\]
Therefore
\begin{eqnarray}\label{es:ee10s}
	{\mathcal T}_{1,0,*}^{e,e}(f)(x,y)
	&\le&
	C_{n,m}
	\int_{S^{n-1}}
	\left\|
	\Omega_{e,e}'(u',\cdot)
	\right\|_{L^1(S^{m-1})}
	\left(
	\Phi_{\R^m}^{*}\circ
	M_{u'}^{\R^n}
	\right)(f)(x,y)
	\,du'
	\nonumber\\
	&&+
	C_{n,m}
	\int_{A_n}
	\left\|
	\Omega_{e,e}(u',\cdot)
	\right\|_{L^1(S^{m-1})}
	\left(
	\Phi_{\R^m}^{*}\circ
	\Phi_{u,\R^n}^{*}
	\right)(f)(x,y)
	\,du
	\nonumber\\
	&&+
	C_{n,m}
	\|\Omega_{e,e}\|_{L^1(S^{n-1}\times S^{m-1})}
	\left(
	\Phi_{\R^m}^{*}\circ
	\Phi_{\R^n}^{*}
	\right)(f)(x,y).
\end{eqnarray}

The estimate for \({\mathcal L}_{0,1,t,s}^{e,e}\) is symmetric, using
\eqref{est:E1} in the \(x\)-variable and \eqref{est:E4} in the \(y\)-variable, we obtain
\begin{eqnarray}\label{es:ee01s}
	{\mathcal T}_{0,1,*}^{e,e}(f)(x,y)
	&\le&
	C_{n,m}
	\int_{S^{m-1}}
	\left\|
	\Omega_{e,e}''(\cdot,v')
	\right\|_{L^1(S^{n-1})}
	\left(
	\Phi_{\R^n}^{*}\circ
	M_{v'}^{\R^m}
	\right)(f)(x,y)
	\,dv'
	\nonumber\\
	&&+
	C_{n,m}
	\int_{A_m}
	\left\|
	\Omega_{e,e}(\cdot,v')
	\right\|_{L^1(S^{n-1})}
	\left(
	\Phi_{\R^n}^{*}\circ
	\Phi_{v,\R^m}^{*}
	\right)(f)(x,y)
	\,dv
	\nonumber\\
	&&+
	C_{n,m}
	\|\Omega_{e,e}\|_{L^1(S^{n-1}\times S^{m-1})}
	\left(
	\Phi_{\R^m}^{*}\circ
	\Phi_{\R^n}^{*}
	\right)(f)(x,y).
\end{eqnarray}

For  \({\mathcal L}_{1,1,t,s}^{e,e}\), where
\(t/4<|x|<2t\) and \(s/4<|y|<2s\). In this region both variables are
annular. We compare
\(\widetilde\chi(|z|/t)\) with \(\widetilde\chi(|x|/t)\), and
\(\widetilde\chi(|w|/s)\) with \(\widetilde\chi(|y|/s)\).

The principal term is
\[
\widetilde\chi\left(\frac{|x|}{t}\right)
\widetilde\chi\left(\frac{|y|}{s}\right)
{\mathcal K}_{e,e}(x,y),
\]
and the two finite radial integrals are controlled by
\[
C_{n,m}
\Liint_{S^{n-1}\times S^{m-1}}
|\omega_{e,e}(u',v')|
M_{u',v'}(f)(x,y)
\,du'dv'.
\]

The error terms are divided into three types. First, the terms in which only
the \(y\)-cutoff is varied contain the partial transform \(\Omega_{e,e}'\).
By \eqref{est:E4} in the \(y\)-variable, these terms are bounded by
\[
C_{n,m}
\Liint_{S^{n-1}\times A_m}
|\Omega_{e,e}'(u',v')|
\left(
\Phi_{v,\R^m}^{*}\circ
M_{u'}^{\R^n}
\right)(f)(x,y)
\,du'dv .
\]
Similarly, by \eqref{est:E4} in the \(x\)-variable, the terms in which only the
\(x\)-cutoff is varied contain \(\Omega_{e,e}''\), and are bounded by
\[
C_{n,m}
\Liint_{A_n\times S^{m-1}}
|\Omega_{e,e}''(u',v')|
\left(
\Phi_{u,\R^n}^{*}\circ
M_{v'}^{\R^m}
\right)(f)(x,y)
\,dudv' .
\]

Second, the term in which both annular cutoffs are varied contains the original
spherical function \(\Omega_{e,e}\). Denote this term by
\[
D_{1,1,t,s}(x,y).
\]
Applying \eqref{est:E4} in both variables gives
\[
\left|
D_{1,1,t,s}(x,y)
\right|
\le
C_{n,m}
\Liint_{A_n\times A_m}
|\Omega_{e,e}(u',v')|
\Phi_t^{\R^n}(x-tu)
\Phi_s^{\R^m}(y-sv)
\,dudv .
\]
Consequently,
\[
\sup_{t,s>0}
\left|
D_{1,1,t,s}*f(x,y)
\right|
\le
C_{n,m}
\Liint_{A_n\times A_m}
|\Omega_{e,e}(u',v')|
\left(
\Phi_{u,\R^n}^{*}\circ
\Phi_{v,\R^m}^{*}
\right)(f)(x,y)
\,dudv .
\]

Finally, the smoother remainders, corresponding to the regions where at least
one variable lies outside the annuli
\[
A_n=\{u\in\R^n:1/8\le |u|\le 3\},
\qquad
A_m=\{v\in\R^m:1/8\le |v|\le 3\},
\]
are estimated as in the \((1,0)\) and \((0,1)\) regions. They are bounded by
\[
C_{n,m}
\int_{S^{n-1}}
\left\|
\Omega_{e,e}'(u',\cdot)
\right\|_{L^1(S^{m-1})}
\left(
\Phi_{\R^m}^{*}\circ
M_{u'}^{\R^n}
\right)(f)(x,y)
\,du'
\]
\[
+
C_{n,m}
\int_{S^{m-1}}
\left\|
\Omega_{e,e}''(\cdot,v')
\right\|_{L^1(S^{n-1})}
\left(
\Phi_{\R^n}^{*}\circ
M_{v'}^{\R^m}
\right)(f)(x,y)
\,dv'
\]
\[
+C_{n,m}\int_{A_n}
\|\Omega_{e,e}(u',\cdot)\|_{L^1(S^{m-1})}
\left(\Phi_{u,\mathbb R^n}^{*}\circ\Phi_{\mathbb R^m}^{*}\right)(f)(x,y)\,du
\]
\[
+C_{n,m}\int_{A_m}
\|\Omega_{e,e}(\cdot,v')\|_{L^1(S^{n-1})}
\left(\Phi_{\mathbb R^n}^{*}\circ\Phi_{v,\mathbb R^m}^{*}\right)(f)(x,y)\,dv
\]
\[
+
C_{n,m}
\|\Omega_{e,e}\|_{L^1(S^{n-1}\times S^{m-1})}
\left(
\Phi_{\R^m}^{*}\circ
\Phi_{\R^n}^{*}
\right)(f)(x,y).
\]

Combining these estimates, we obtain
\begin{eqnarray}\label{es:ee11s}
	{\mathcal T}_{1,1,*}^{e,e}(f)(x,y)
	&\le&
	C_{n,m}
	\Liint_{S^{n-1}\times S^{m-1}}
	|\omega_{e,e}(u',v')|
	M_{u',v'}(f)(x,y)
	\,du'dv'
	\nonumber\\
	&&+
	C_{n,m}
	\int_{S^{n-1}}
	\left\|
	\Omega_{e,e}'(u',\cdot)
	\right\|_{L^1(S^{m-1})}
	\left(
	\Phi_{\R^m}^{*}\circ
	M_{u'}^{\R^n}
	\right)(f)(x,y)
	\,du'
	\nonumber\\
	&&+
	C_{n,m}
	\int_{S^{m-1}}
	\left\|
	\Omega_{e,e}''(\cdot,v')
	\right\|_{L^1(S^{n-1})}
	\left(
	\Phi_{\R^n}^{*}\circ
	M_{v'}^{\R^m}
	\right)(f)(x,y)
	\,dv'
	\nonumber\\
	&&+
	C_{n,m}
	\Liint_{S^{n-1}\times A_m}
	|\Omega_{e,e}'(u',v')|
	\left(
	\Phi_{v,\R^m}^{*}\circ
	M_{u'}^{\R^n}
	\right)(f)(x,y)
	\,du'dv
	\nonumber\\
	&&+
	C_{n,m}
	\Liint_{A_n\times S^{m-1}}
	|\Omega_{e,e}''(u',v')|
	\left(
	\Phi_{u,\R^n}^{*}\circ
	M_{v'}^{\R^m}
	\right)(f)(x,y)
	\,dudv'
	\nonumber\\
	&&+
	C_{n,m}
	\Liint_{A_n\times A_m}
	|\Omega_{e,e}(u',v')|
	\left(
	\Phi_{u,\R^n}^{*}\circ
	\Phi_{v,\R^m}^{*}
	\right)(f)(x,y)
	\,dudv
	\nonumber\\
&&+C_{n,m}\int_{A_n}
\|\Omega_{e,e}(u',\cdot)\|_{L^1(S^{m-1})}
\left(\Phi_{u,\mathbb R^n}^{*}\circ\Phi_{\mathbb R^m}^{*}\right)(f)(x,y)\,du
\nonumber\\
&&+C_{n,m}\int_{A_m}
\|\Omega_{e,e}(\cdot,v')\|_{L^1(S^{n-1})}
\left(\Phi_{\mathbb R^n}^{*}\circ\Phi_{v,\mathbb R^m}^{*}\right)(f)(x,y)\,dv
\nonumber\\
	&&+
	C_{n,m}
	\|\Omega_{e,e}\|_{L^1(S^{n-1}\times S^{m-1})}
	\left(
	\Phi_{\R^m}^{*}\circ
	\Phi_{\R^n}^{*}
	\right)(f)(x,y).
\end{eqnarray}

We next consider \({\mathcal L}_{1,\infty,t,s}^{e,e}\), where
\(t/4<|x|<2t\) and \(|y|\ge 2s\). In this region the first variable is
annular, while the second variable is away from the support of the cutoff. We
compare \(\widetilde\chi(|z|/t)\) with \(\widetilde\chi(|x|/t)\), and replace
\(\widetilde\chi(|w|/s)\) by \(1\), treating
\(\widetilde\chi(|w|/s)-1\) as a cutoff error.

The principal term contains the fully transformed spherical function
\(\omega_{e,e}\), and is controlled by
\[
C_{n,m}
\Liint_{S^{n-1}\times S^{m-1}}
|\omega_{e,e}(u',v')|
\left(M_{u'}^{\R^n}\circ H_{v'}^{*,\R^m}\right)(f)(x,y)
\,du'dv' .
\]
The \(y\)-cutoff error with the first variable already transformed is
controlled by \eqref{est:E2} in the \(y\)-variable; it contains
\(\Omega_{e,e}'\), and gives
\[
C_{n,m}
\int_{S^{n-1}}
\left\|
\Omega_{e,e}'(u',\cdot)
\right\|_{L^1(S^{m-1})}
\left(
\Phi_{\R^m}^{*}\circ
M_{u'}^{\R^n}
\right)(f)(x,y)
\,du' .
\]
The annular cutoff error in the \(x\)-variable is controlled by
\eqref{est:E4}; with the second variable already transformed, it contains
\(\Omega_{e,e}''\).
Its non-annular part is controlled by
\[
C_{n,m}
\int_{S^{m-1}}
\left\|
\Omega_{e,e}''(\cdot,v')
\right\|_{L^1(S^{n-1})}
\left(
\Phi_{\R^n}^{*}\circ
H_{v'}^{*,\R^m}
\right)(f)(x,y)
\,dv',
\]
while its annular part is controlled by
\[
C_{n,m}
\Liint_{A_n\times S^{m-1}}
|\Omega_{e,e}''(u',v')|
\left(
\Phi_{u,\R^n}^{*}\circ
H_{v'}^{*,\R^m}
\right)(f)(x,y)
\,dudv' .
\]
Finally, the remaining terms contain the original spherical function
\(\Omega_{e,e}\), and are bounded by
\begin{eqnarray*}
C_{n,m}
\int_{A_n}
\left\|
\Omega_{e,e}(u',\cdot)
\right\|_{L^1(S^{m-1})}
\left(
\Phi_{u,\R^n}^{*}\circ
\Phi_{\R^m}^{*}
\right)(f)(x,y)
\,du\\
+
C_{n,m}
\|\Omega_{e,e}\|_{L^1(S^{n-1}\times S^{m-1})}
\left(
\Phi_{\R^m}^{*}\circ
\Phi_{\R^n}^{*}
\right)(f)(x,y).
\end{eqnarray*}
Combining these bounds, we obtain
\begin{eqnarray}\label{es:ee1is}
	{\mathcal T}_{1,\infty,*}^{e,e}(f)(x,y)
	&\le&
	C_{n,m}
	\Liint_{S^{n-1}\times S^{m-1}}
	|\omega_{e,e}(u',v')|
	\left(M_{u'}^{\R^n}\circ H_{v'}^{*,\R^m}\right)(f)(x,y)
	\,du'dv'
	\nonumber\\
	&&+
	C_{n,m}
	\int_{S^{n-1}}
	\left\|
	\Omega_{e,e}'(u',\cdot)
	\right\|_{L^1(S^{m-1})}
	\left(
	\Phi_{\R^m}^{*}\circ
	M_{u'}^{\R^n}
	\right)(f)(x,y)
	\,du'
	\nonumber\\
	&&+
	C_{n,m}
	\int_{S^{m-1}}
	\left\|
	\Omega_{e,e}''(\cdot,v')
	\right\|_{L^1(S^{n-1})}
	\left(
	\Phi_{\R^n}^{*}\circ
	H_{v'}^{*,\R^m}
	\right)(f)(x,y)
	\,dv'
	\nonumber\\
	&&+
	C_{n,m}
	\Liint_{A_n\times S^{m-1}}
	|\Omega_{e,e}''(u',v')|
	\left(
	\Phi_{u,\R^n}^{*}\circ
	H_{v'}^{*,\R^m}
	\right)(f)(x,y)
	\,dudv'
	\nonumber\\
	&&+
	C_{n,m}
	\int_{A_n}
	\left\|
	\Omega_{e,e}(u',\cdot)
	\right\|_{L^1(S^{m-1})}
	\left(
	\Phi_{u,\R^n}^{*}\circ
	\Phi_{\R^m}^{*}
	\right)(f)(x,y)
	\,du
	\nonumber\\
	&&+
	C_{n,m}
	\|\Omega_{e,e}\|_{L^1(S^{n-1}\times S^{m-1})}
	\left(
	\Phi_{\R^m}^{*}\circ
	\Phi_{\R^n}^{*}
	\right)(f)(x,y).
\end{eqnarray}

By symmetry with the preceding case, using \eqref{est:E2} in the
\(x\)-variable and \eqref{est:E4} in the \(y\)-variable, we obtain
\begin{eqnarray}\label{es:eei1s}
	{\mathcal T}_{\infty,1,*}^{e,e}(f)(x,y)
	&\le&
	C_{n,m}
	\Liint_{S^{n-1}\times S^{m-1}}
	|\omega_{e,e}(u',v')|
	\left(M_{v'}^{\R^m}\circ H_{u'}^{*,\R^n}\right)(f)(x,y)
	\,du'dv'
	\nonumber\\
	&&+
	C_{n,m}
	\int_{S^{m-1}}
	\left\|
	\Omega_{e,e}''(\cdot,v')
	\right\|_{L^1(S^{n-1})}
	\left(
	\Phi_{\R^n}^{*}\circ
	M_{v'}^{\R^m}
	\right)(f)(x,y)
	\,dv'
	\nonumber\\
	&&+
	C_{n,m}
	\int_{S^{n-1}}
	\left\|
	\Omega_{e,e}'(u',\cdot)
	\right\|_{L^1(S^{m-1})}
	\left(
	\Phi_{\R^m}^{*}\circ
	H_{u'}^{*,\R^n}
	\right)(f)(x,y)
	\,du'
	\nonumber\\
	&&+
	C_{n,m}
	\Liint_{S^{n-1}\times A_m}
	|\Omega_{e,e}'(u',v')|
	\left(
	\Phi_{v,\R^m}^{*}\circ
	H_{u'}^{*,\R^n}
	\right)(f)(x,y)
	\,du'dv
	\nonumber\\
	&&+
	C_{n,m}
	\int_{A_m}
	\left\|
	\Omega_{e,e}(\cdot,v')
	\right\|_{L^1(S^{n-1})}
	\left(
	\Phi_{v,\R^m}^{*}\circ
	\Phi_{\R^n}^{*}
	\right)(f)(x,y)
	\,dv
	\nonumber\\
	&&+
	C_{n,m}
	\|\Omega_{e,e}\|_{L^1(S^{n-1}\times S^{m-1})}
	\left(
	\Phi_{\R^m}^{*}\circ
	\Phi_{\R^n}^{*}
	\right)(f)(x,y).
\end{eqnarray}

Combining the estimates for all
\((i,j)\in\{0,1,\infty\}^2\),  \eqref{teeiis}-\eqref{es:eei1s},
we obtain
\[
\sum_{\alpha,\beta\in\{o,e\}}
{\mathcal T}_{*}^{\alpha,\beta}(f)(x,y)
\le
{\mathcal M}_{\Omega}(f)(x,y),
\]
where \({\mathcal M}_{\Omega}\) is a finite sum of the maximal operators appearing
on the right-hand sides of \eqref{17}, \eqref{27}, \eqref{28}, and \eqref{teeiis}-\eqref{es:eei1s},
with coefficients
\[
|\omega_{\alpha,\beta}|,
\qquad
|\Omega_{\alpha,\beta}|,
\qquad
|\Omega_{e,e}'|,
\qquad
|\Omega_{e,e}''|.
\]
More precisely, every term in \({\mathcal M}_{\Omega}\) is a composition of
directional maximal Hilbert transforms, one-parameter maximal functions, or
translated radial maximal functions, all of which are bounded on
\(L^p(\R^n\times\R^m)\), \(1<p<\infty\), with bounds independent of the
directions and of the translation parameters in the compact annuli
\(A_n\) and \(A_m\).

Finally, by Lemma \ref{lemma7}, all $\omega_{\alpha,\beta}$ belong to $L^1$, and  the even-even partial transforms $\Omega_{e,e}'$ and $\Omega_{e,e}''$ also belong to $L^1(S^{n-1}\times S^{m-1})$:
$$
\sum_{\alpha,\beta}\|\omega_{\alpha,\beta}\|_1+\|\Omega_{e,e}'\|_1+\|\Omega_{e,e}''\|_1+\|\Omega\|_1\le C\|\Omega\|_{H^1(S^{n-1}\times S^{m-1}) }.
$$
Hence Minkowski's
inequality and the \(L^p\)-boundedness of the above maximal operators imply
\[
\sum_{\alpha,\beta\in\{o,e\}}
\left\|
{\mathcal T}_{*}^{\alpha,\beta}(f)
\right\|_{L^p(\R^n\times\R^m)}
\le
C_{p,n,m}
\|\Omega\|_{H^1(S^{n-1}\times S^{m-1})}
\|f\|_{L^p(\R^n\times\R^m)}.
\]

Recall that
\[
T_{*}^{\alpha,\beta}(f)
=
{\mathcal T}_{*}^{\alpha,\beta}(f^{\alpha,\beta}),
\qquad
f^{\alpha,\beta}
=
\left(
\mathfrak R_{\R^n}^{\alpha}
\otimes
\mathfrak R_{\R^m}^{\beta}
\right)f,
\]
up to harmless signs inside the absolute value. Since the Riesz transforms are
bounded on \(L^p\), it follows that
\[
\sum_{\alpha,\beta\in\{o,e\}}
\left\|
T_{*}^{\alpha,\beta}(f)
\right\|_{L^p(\R^n\times\R^m)}
\le
C_{p,n,m}
\|\Omega\|_{H^1(S^{n-1}\times S^{m-1})}
\|f\|_{L^p(\R^n\times\R^m)}.
\]

We now estimate the local terms. Fix a parity component and, for
$u'\in S^{n-1}$, put $b_{u'}(v')=\Omega_{\alpha,\beta}(u',v')$.
By the separate cancellation condition, $b_{u'}$ has mean zero on
$S^{m-1}$ for almost every $u'\in S^{n-1}$.
Let
\[
B_{s,b}g(y)=\int_{\mathbb R^m}
\frac{b(v')}{|v|^m}\widetilde\chi(|v|/s)g(y-v)\,dv,
\qquad B_b^*g=\sup_{s>0}|B_{s,b}g|.
\]
The one-parameter maximal singular integral estimate for mean-zero
$H^1$ kernels gives
\[
\|B_b^*g\|_{L^p(\mathbb R^m)}
\le C_{p,m}\|b\|_{H^1(S^{m-1})}\|g\|_{L^p(\mathbb R^m)};
\]
see \cite[Introduction]{GraS} and the references there.
Indeed, the smooth truncations are a signed superposition of sharp truncations, since
$\widetilde\chi(r)=\int_{1/2}^1\widetilde\chi'(a)
\mathbf1_{\{r>a\}}\,da$.
The kernel of $T_*^{\alpha,\beta,1}$ has $t/2<|u|<t$ in its
first variable. Polar coordinates and $dr/r\le2\,dr/t$ therefore give
\[
T_*^{\alpha,\beta,1}f(x,y)
\le2\int_{S^{n-1}}
M_{u'}^{\mathbb R^n}\big(B_{b_{u'}}^*f\big)(x,y)\,du',
\]
where $B_{b_{u'}}^*$ acts only in $y$.
Letting the first Poisson parameter tend to zero gives, for almost
every $(u',v')$,
\[
P_{S^{m-1}}^+(b_{u'})(v')
\le P^+(\Omega_{\alpha,\beta})(u',v').
\]
Hence, by integration,
\[
\int_{S^{n-1}}\|b_{u'}\|_{H^1(S^{m-1})}\,du'
\le\|\Omega_{\alpha,\beta}\|_{H^1(S^{n-1}\times S^{m-1})}.
\]
Minkowski's inequality and the uniform $L^p$ bound for $M_{u'}^{\mathbb R^n}$
now imply
\[
\|T_*^{\alpha,\beta,1}f\|_p
\le C_{p,n,m}\|\Omega_{\alpha,\beta}\|_{H^1}\|f\|_p.
\]
The second local term is treated by interchanging the variables.
Since the parity projections contract the product Hardy norm,
\[
\sum_{\alpha,\beta\in\{o,e\}}\sum_{j=1}^2
\|T_*^{\alpha,\beta,j}f\|_p
\le C_{p,n,m}\|\Omega\|_{H^1}\|f\|_p.
\]
Moreover, the last term in \eqref{tomega*} satisfies
\begin{eqnarray*}
	\begin{aligned}
		&\left\|
		\Liint_{S^{n-1}\times S^{m-1}}
		|\Omega_{\alpha,\beta}(u',v')|
		M_{u',v'}(f)
		\,du'dv'
		\right\|_{L^p(\R^n\times\R^m)}
		\\
		&\qquad\le
		C_p
		\|\Omega_{\alpha,\beta}\|_{L^1(S^{n-1}\times S^{m-1})}
		\|f\|_{L^p(\R^n\times\R^m)},
	\end{aligned}
\end{eqnarray*}
since \(M_{u',v'}\) (see \eqref{13}) is dominated by the composition of the two
one-dimensional Hardy--Littlewood maximal operators in the directions
\(u'\) and \(v'\), and hence is uniformly bounded on \(L^p\).

Combining these estimates with \eqref{tomega*}, we conclude that
\[
\|T_{\Omega}^{*}(f)\|_{L^p(\R^n\times\R^m)}
\le
C_{p,n,m}
\|\Omega\|_{H^1(S^{n-1}\times S^{m-1})}
\|f\|_{L^p(\R^n\times\R^m)}.
\]
This proves the estimate for smooth kernels. For general $\Omega$,
take smooth probability rotation averages $\Omega_\varepsilon$.
They preserve separate cancellation, converge in $L^1$, and satisfy
$\|\Omega_\varepsilon\|_{H^1}\le\|\Omega\|_{H^1}$.
For Schwartz $f$, every fixed truncation converges pointwise; taking
the supremum over truncations and applying Fatou's lemma gives the
same maximal estimate for $\Omega$. The principal value distributions
also converge on Schwartz functions by separate cancellation, so the
bound for $T_\Omega$ passes to the limit by duality. Density then gives
both estimates on $L^p$.
\endpf

	\section{Proof of Theorem \ref{thm2}}\label{Marcinkiewicz integrals on product spaces}

We first recall some basic vector-valued Littlewood--Paley estimates. For
\(\sigma\in L^1(\R)\), set
\[
\sigma_t(x)=t^{-1}\sigma(x/t),
\qquad t>0,
\]
and define
\begin{eqnarray}\label{40}
	\left\{
	\begin{array}{l}
		g_{\sigma}(f)(x)
		=
		\left|T_{\sigma}(f)(x)\right|_{\mathcal H},
		\\[4pt]
		T_{\sigma}(f)(x)
		=
		\{\sigma_t*f(x)\}_{t>0}
		=
		K_{\sigma}*f(x),
		\\[4pt]
		{\mathcal H}
		=
		L^2(\R_+,\frac{dt}{t}),
		\qquad
		K_{\sigma}(x)
		=
		\{\sigma_t(x)\}_{t>0}.
	\end{array}
	\right.
\end{eqnarray}
The following is a standard consequence of the vector-valued Calderón--Zygmund
theory; see \cite{BCP,ST0}.

	\begin{Lemma}\label{lemma8}
	If
	\beno
	\|\sigma\|_1<\infty,~~\int_{\R}\sigma(t)~dt=0,~|K_{\sigma}(x)|_{\mathcal H}\le B|x|^{-1},\\
	\int_{|x|>2|y|}\left|K_{\sigma}(x-y)-K_{\sigma}(x)\right|_{\mathcal H}\,dx
\le B\quad(y\ne0),\\
\sup_{\xi\ne0}\int_0^\infty
|\widehat\sigma(t\xi)|^2\,\frac{dt}{t}\le B^2,
	\eeno
	then $\|g_{\sigma}(f)\|_{L^{1,\infty}}\le C\left(\|\sigma\|_1+B\right)\|f\|_1$, $\|g_{\sigma}(f)\|_p\le C_p(\|\sigma\|_1+B)\|f\|_p$, $1<p<+\infty$.
\end{Lemma}
The last hypothesis gives the $L^2$ estimate by Plancherel's theorem;
the kernel estimates then give weak $(1,1)$ and $L^p$ boundedness by
the vector-valued Calder\'on--Zygmund theorem.

We shall also use a product version. For \(\zeta\in L^1(\R\times\R)\), set
\[
\zeta_{t,s}(x,y)
=
t^{-1}s^{-1}\zeta(x/t,y/s),
\qquad t,s>0,
\]
and define
\begin{eqnarray}\label{41}
	\left\{
	\begin{array}{l}
		\widetilde g_{\zeta}(f)(x,y)
		=
		\left|
		\widetilde T_{\zeta}(f)(x,y)
		\right|_{\widetilde{\mathcal H}},
		\\[4pt]
		\widetilde T_{\zeta}(f)(x,y)
		=
		\{\zeta_{t,s}*f(x,y)\}_{t,s>0}
		=
		\widetilde K_{\zeta}*f(x,y),
		\\[4pt]
		\widetilde{\mathcal H}
		=
		L^2((\R_+)^2,\frac{dt}{t}\frac{ds}{s}),
		\qquad
		\widetilde K_{\zeta}(x,y)
		=
		\{\zeta_{t,s}(x,y)\}_{t,s>0}.
	\end{array}
	\right.
\end{eqnarray}

\begin{Lemma}\label{lemma9}
	Suppose that \(\|\zeta\|_1<\infty\), that \(\zeta\) is odd in both variables,
	and that there exists \(\alpha\in(0,1]\) such that
	\[
	\left|
	\widetilde K_{\zeta}(x,y)
	\right|_{\widetilde{\mathcal H}}
	\le
	B|x|^{-1}|y|^{-1},
	\qquad x\neq0,\ y\neq0,
	\]
	\[
	\left|
	\widetilde K_{\zeta}(x+h,y)
	-
	\widetilde K_{\zeta}(x,y)
	\right|_{\widetilde{\mathcal H}}
	\le
	B
	\frac{|h|^\alpha}{|x|^{1+\alpha}}
	|y|^{-1},
	\qquad |x|\ge2|h|,
	\]
	\[
	\left|
	\widetilde K_{\zeta}(x,y+k)
	-
	\widetilde K_{\zeta}(x,y)
	\right|_{\widetilde{\mathcal H}}
	\le
	B
	|x|^{-1}
	\frac{|k|^\alpha}{|y|^{1+\alpha}},
	\qquad |y|\ge2|k|,
	\]
	and
	\[
	\begin{aligned}
		&
		\left|
		\widetilde K_{\zeta}(x+h,y+k)
		-
		\widetilde K_{\zeta}(x+h,y)
		-
		\widetilde K_{\zeta}(x,y+k)
		+
		\widetilde K_{\zeta}(x,y)
		\right|_{\widetilde{\mathcal H}}
		\\
		&\qquad\le
		B
		\frac{|h|^\alpha |k|^\alpha}
		{|x|^{1+\alpha}|y|^{1+\alpha}},
		\qquad
		|x|\ge2|h|,\ |y|\ge2|k|.
	\end{aligned}
	\]
	Then, for \(1<p<\infty\),
	\[
	\|\widetilde g_{\zeta}(f)\|_p
	\le
	C_p(\|\zeta\|_1+B)\|f\|_p.
	\]
\end{Lemma}

{\bf Proof:}
For an odd function $a\in L^1(\mathbb R)$, we first note that
\[
\int_0^\infty |\widehat a(\xi)|^2\,\frac{d\xi}{\xi}
\le C\int_0^\infty r|a(r)|^2\,dr.
\]
For smooth $a$ compactly supported away from zero, this follows from
the identity
\[
\int_0^\infty\sin(r\xi)\sin(q\xi)\,\frac{d\xi}{\xi}
=\frac12\log\frac{r+q}{|r-q|},\qquad r,q>0,\quad r\ne q,
\]
interpreted as an Abel limit. Indeed, setting $r=e^u$ and $q=e^v$
reduces the estimate to Young's inequality for $e^ua(e^u)$ and the
integrable kernel $\log\coth(|u-v|/2)$. Approximation gives the
estimate whenever its right-hand side is finite.

The size assumption gives
\[
\int_0^\infty\int_0^\infty rs|\zeta(r,s)|^2\,dr\,ds
=\|\widetilde K_\zeta(1,1)\|_{\widetilde{\mathcal H}}^2
\le B^2.
\]
Applying the preceding estimate in each variable, using the separate
oddness of $\zeta$, and rescaling the frequencies, we obtain
\[
\sup_{\xi\eta\ne0}\int_0^\infty\int_0^\infty
|\widehat\zeta(t\xi,s\eta)|^2\,\frac{dt}{t}\frac{ds}{s}
\le CB^2.
\]
Plancherel's theorem therefore gives the $L^2$ bound. The asserted
$L^p$ estimate now follows from the Hilbert-valued product
Calder\'on--Zygmund theorem, using the preceding $L^2$ bound together with
the stated size, separate difference, and mixed difference estimates; see
\cite{FS2,Fernandez1987}.
\endpf

	Now, for the Marcinkiewicz integral on \(\R^n\times\R^m\), we have
	\[
	\mu_{\Omega}(f)
	=
	\widetilde g_{\psi^{\Omega}}(f),
	\]
	see \eqref{marcinkiewicz} and \eqref{4} for the definitions.  By the parity decomposition, it suffices to prove the estimate for each parity
	component separately.  We present the proof in the even-even case, which is the
	only case requiring Riesz transforms in both variables. The remaining parity
	components are treated in the same way, with fewer transformed variables.

	 Since in the present argument \(\Omega\) is even in both variables, we use the
	 following local notation for the fully and partially transformed spherical
	 functions:
\[	\omega(x',y')
	=
	\widetilde R_{S^{n-1}}
	\otimes
	\widetilde R_{S^{m-1}}(\Omega)(x',y'),
	\]
	\[
	\Omega'(x',y')
	=
	\widetilde R_{S^{n-1}}(\Omega(\cdot,y'))(x'),
	\qquad
	\Omega''(x',y')
	=
	\widetilde R_{S^{m-1}}(\Omega(x',\cdot))(y').
	\]

	Let
	\ba\label{decomp:chi}
	\left\{
	\begin{array}{l}
		\chi(t)=\chi_{(0,1)}(t),\\
		\lambda\in C_c^\infty((0,+\infty)),
		\qquad
		\supp\lambda\subset[1,2],
		\qquad
		\int_0^\infty\lambda(t)\,dt=1,\\
		\rho(t)=\chi(t)-\lambda(t).
	\end{array}
	\right.
	\ea

	For smooth \(\Omega\), \(u'\in S^{n-1}\), and
	\(v'\in S^{m-1}\), define
	\ba\label{42}
	\left\{
	\begin{array}{l}
		\sigma^{(1,u',v')}(t)
		=
		|t|^{n-1}
		{\mathfrak R}_{\R^n}
		\left(
		\lambda(|\cdot|)|\cdot|^{1-n}
		\Omega(\cdot,v')
		\right)(tu'),
		\\[6pt]
		\sigma^{(2,u',v')}(s)
		=
		|s|^{m-1}
		{\mathfrak R}_{\R^m}
		\left(
		\lambda(|\circ|)|\circ|^{1-m}
		\Omega(u',\circ)
		\right)(sv'),
		\\[6pt]
		\zeta^{(u',v')}(t,s)
		=
		|t|^{n-1}|s|^{m-1}
		\left(
		{\mathfrak R}_{\R^n}
		\otimes
		{\mathfrak R}_{\R^m}
		\right)
		\left(
		\frac{\lambda(|\cdot|)}{|\cdot|^{n-1}}
		\frac{\lambda(|\circ|)}{|\circ|^{m-1}}
		\Omega(\cdot,\circ)
		\right)(tu',sv').
	\end{array}
	\right.
	\ea
	For general \(\Omega\), these profiles are understood through the
	almost-everywhere radial representatives constructed in Lemma \ref{lemma10}.
	Then \(\sigma^{(1,u',v')}\) is odd in \(t\),
	\(\sigma^{(2,u',v')}\) is odd in \(s\), and
	\(\zeta^{(u',v')}\) is odd both in \(t\) and in \(s\).
	Together with the integrable radial bounds in Lemma \ref{lemma10}, this
	gives, for almost every \((u',v')\),
	\[
	\int_{\R}\sigma^{(1,u',v')}(t)\,dt=0,
	\qquad
	\int_{\R}\sigma^{(2,u',v')}(s)\,ds=0,
	\]
	and
	\[
	\int_{\R}
	\zeta^{(u',v')}(t,s)\,dt=0,\quad(\forall s\in\R),\quad\int_{\R}\zeta^{(u',v')}(t,s)\,ds=0,\quad(\forall t\in\R).
	\]
	Thus \(\sigma^{(1,u',v')}\) is \(\mathbb C^n\)-valued,
	\(\sigma^{(2,u',v')}\) is \(\mathbb C^m\)-valued, and
	\(\zeta^{(u',v')}\) is \(\mathbb C^n\otimes\mathbb C^m\)-valued. Their
	products with
	\[
	\mathfrak R_{\mathbb R^n}f,\qquad
	\mathfrak R_{\mathbb R^m}f,\qquad
	(\mathfrak R_{\mathbb R^n}\otimes\mathfrak R_{\mathbb R^m})f
	\]
	below are understood through the corresponding Euclidean or Hilbert--Schmidt
	contractions.

	Let
	\begin{eqnarray*}
		\Omega_{k}^{*}(x',y')
		&=&
		\sup_{r\in\R,\ j=0,1}
		(1+|r|)^{2+j}
		\left|
		\partial_{r}^{j}
		\left(
		\sigma^{(k,x',y')}(r)
		\right)
		\right|,
		\qquad k=1,2,
		\\
		\Omega^{*}(x',y')
		&=&
		\sup_{t,s\in\R,\ i,j=0,1}
		(1+|t|)^{2+j}
		(1+|s|)^{2+i}
		\left|
		\partial_t^j\partial_s^i
		\left(
		\zeta^{(x',y')}(t,s)
		\right)
		\right|.
	\end{eqnarray*}
Here the absolute values are the Euclidean norms for
\(\sigma^{(1,u',v')}\) and \(\sigma^{(2,u',v')}\), and the Hilbert--Schmidt
norm for \(\zeta^{(u',v')}\).
	The following lemma holds.

	\begin{Lemma}\label{lemma10}
		For
		\[
		\Omega\in H^1(S^{n-1}\times S^{m-1})\wedge \eqref{cancellation},
		\]
		which is even both in the first and in the second variables, there holds
		\[
		\left\|
		\sum_{k=1}^{2}|\Omega_k^*|
		+
		|\Omega^*|
		\right\|_{L^1(S^{n-1}\times S^{m-1})}
		\le
		C_{n,m}
		\|\Omega\|_{H^1(S^{n-1}\times S^{m-1})}.
		\]
	\end{Lemma}
	{\bf Proof:}
	We first assume that \(\Omega\) is smooth. We begin with a one-parameter
	estimate which will be used in each spherical variable.

	Let \(a\) be a smooth mean-zero function on \(S^{l-1}\), \(l\ge2\), and set
\begin{eqnarray*}
	V_{l,t}a(\theta)
	&=&
	t^{l-1}{\mathfrak R}_{\mathbb R^l}
	\big(
	\lambda(|\cdot|)|\cdot|^{1-l}a(\cdot)
	\big)(t\theta),\qquad t>0.
\end{eqnarray*}
	For \(j=0,1\), write
	\[
	U_{l,t}^{(j)}a
	=
	(1+t)^{2+j}\partial_t^jV_{l,t}a.
	\]

	We first consider \(1/2\le t\le4\). 
	We write
	\[V_{l,t}a(\theta)=\operatorname{p.v.}\int_{\R^l}K_l(\theta-z)\frac{a(z')}{|z|^l}{\lambda(t|z|)}|z|\,dz.\]
	Put $
	b_j(r)=r^{j+1}\lambda^{(j)}(r).
$
	By scaling the Riesz kernel and differentiating in \(t\), we obtain
	\begin{equation}\label{43}
		\partial_t^jV_{l,t}a(\theta)
		=
		t^{-1-j}\,\operatorname{p.v.}\int_{\mathbb R^l}
		K_l(\theta-z)
		\frac{a(z')}{|z|^l}
		b_j(t|z|)\,dz .
	\end{equation}
	We split
	\[
	b_j(t|z|)
	=
	b_j(t)
	+
	\big[b_j(t|z|)-b_j(t)\big].
	\]
	The first term gives the restricted Riesz transform. Hence
	\begin{equation}\label{eq:radial-profile-decomposition}
		U_{l,t}^{(j)}a
		=
		c_j(t)\widetilde R_{S^{l-1}}a
		+
		E_{l,t}^{(j)}a,
	\end{equation}
	where
	\[
	c_j(t)
	=
	(1+t)^{2+j}t^{-1-j}b_j(t).
	\]
	Since \(1/2\le t\le4\), the function \(c_j(t)\) is uniformly bounded.

	It remains to estimate the error term. Its angular kernel is
	\[
	e_{l,t}^{(j)}(\theta,\eta)
	=
	(1+t)^{2+j}t^{-1-j}
	\int_0^\infty
	\big[
	K_l(\theta-r\eta)
	-
	\mathbf 1_{(0,1/4)}(r)K_l(\theta)
	\big]
	\big[
	b_j(tr)-b_j(t)
	\big]
	\frac{dr}{r}.
	\]
	The subtraction of \(K_l(\theta)\) is allowed since \(\int a=0\).
	For \(r<1/4\) and \(r>4\), the kernel is uniformly bounded. Thus it is enough
	to consider \(1/2\le r\le2\). If
	\[
	d=|\theta-\eta|,
	\]
	then
	\[
	|\theta-r\eta|^2
	=
	(1-r)^2+rd^2,
	\]
	and, by the smoothness of \(b_j\),
	\[
	|b_j(tr)-b_j(t)|
	\le
	C|r-1|.
	\]
	Therefore
	\[
	|e_{l,t}^{(j)}(\theta,\eta)|
	\le
	C\left(
	1+
	\int_{1/2}^2
	\frac{|r-1|}
	{\big((r-1)^2+d^2\big)^{l/2}}
	\,dr
	\right).
	\]
	Consequently,
	\[
	|e_{l,t}^{(j)}(\theta,\eta)|
	\le
	CJ_l(d),
	\]
	where
	\[
	J_l(d)
	=
	\begin{cases}
		1+\log^+(1/d),& l=2,\\[3pt]
		1+d^{2-l},& l\ge3.
	\end{cases}
	\]
	Since \(J_l(|\theta-\eta|)\) is integrable on \(S^{l-1}\), uniformly in
	\(\theta\), we have
	\begin{equation}\label{eq:angular-remainder-L1}
		|E_{l,t}^{(j)}a(\theta)|
		\le
		C\mathcal B_l(|a|)(\theta),
		\qquad
		\|\mathcal B_l b\|_1
		\le
		C_l\|b\|_1,
	\end{equation}
	where
	\[
	\mathcal B_l b(\theta)
	=
	\int_{S^{l-1}}
	J_l(|\theta-\eta|)b(\eta)\,d\eta .
	\]

	We next consider the ranges \(t\le1/2\) and \(t\ge4\). If
	\(0<t\le1/2\), then the support of \(\lambda\) gives
	\[
	|t\theta-z|\ge\frac12.
	\]
	Thus direct differentiation gives
	\[
	|\partial_t^jV_{l,t}a|
	\le
	Ct^{l-1-j}\|a\|_1,
	\qquad j=0,1.
	\]
	If \(a\) is even, then
	\[
	F_a(x):=\mathfrak R_{\mathbb R^l}
	\bigl(\lambda(|\cdot|)|\cdot|^{1-l}a\bigr)(x)
	\]
	is odd, so \(F_a(0)=0\). Since $\lambda$ is supported in
	\(1\le |z|\le2\), differentiation under the integral gives
	\[
	\sup_{|x|\le1/2}|\nabla F_a(x)|\le C\|a\|_1,
	\qquad |F_a(t\theta)|\le Ct\|a\|_1.
	\]
	Using \(V_{l,t}a(\theta)=t^{l-1}F_a(t\theta)\) and differentiating,
	we obtain
	\[
	|\partial_t^jV_{l,t}a(\theta)|
	\le Ct^{l-j}\|a\|_1,
	\qquad j=0,1,\quad 0<t<1/2.
	\]
	Since \(l\ge2\), its odd extension is \(C^1\) at zero,
	with value and derivative both zero.

	If \(t\ge4\), we use the cancellation of \(a\). Since
	\(1\le |z|\le2\) on the support of \(\lambda\),
	\[
	|\nabla^hK_l(t\theta-z)-\nabla^hK_l(t\theta)|
	\le
	C|z|t^{-l-h-1},
	\qquad h=0,1.
	\]
	It follows that
	\[
	|\partial_t^jV_{l,t}a|
	\le
	Ct^{-2-j}\|a\|_1,
	\qquad j=0,1.
	\]
	Thus \eqref{eq:radial-profile-decomposition} and
	\eqref{eq:angular-remainder-L1} hold for every \(t>0\), with
	\(|c_j(t)|\le C\). For \(t\notin[1/2,4]\), we simply take \(c_j(t)=0\).

	We now apply this one-parameter estimate to the functions in \eqref{42}.
	For \(t,s>0\),
	\[
	(1+t)^{2+j}(1+s)^{2+i}
	\partial_t^j\partial_s^i
	\zeta^{(u',v')}(t,s)
	=
	\big(
	U_{n,t}^{(j)}
	\otimes
	U_{m,s}^{(i)}
	\big)\Omega(u',v').
	\]
	Using \eqref{eq:radial-profile-decomposition} in both variables, we obtain four
	types of terms. The term containing two restricted Riesz transforms is
	controlled by \(|\omega|\). The two terms containing one restricted Riesz
	transform and one error term are controlled by
	\(\mathcal B_m(|\Omega'|)\) and
	\(\mathcal B_n(|\Omega''|)\), respectively. The remaining term is controlled
	by
	\(\mathcal B_n\mathcal B_m(|\Omega|)\). Hence
	\begin{equation}\label{eq:Omega-star-pointwise}
		\Omega^*
		\le
		C\Big(
		|\omega|
		+
		\mathcal B_m(|\Omega'|)
		+
		\mathcal B_n(|\Omega''|)
		+
		\mathcal B_n\mathcal B_m(|\Omega|)
		\Big).
	\end{equation}
	In the same way,
	\[
	\Omega_1^*
	\le
	C\Big(
	|\Omega'|+\mathcal B_n(|\Omega|)
	\Big),
	\]
	and
	\[
	\Omega_2^*
	\le
	C\Big(
	|\Omega''|+\mathcal B_m(|\Omega|)
	\Big).
	\]
	The estimates near \(t=0\) obtained above show that the same bounds hold for
	the odd extensions to all real parameters.

	Integrating these inequalities and using
	\eqref{eq:angular-remainder-L1}, we obtain
	\[
	\|\Omega_1^*\|_1
	+
	\|\Omega_2^*\|_1
	+
	\|\Omega^*\|_1
	\le
	C\Big(
	\|\Omega\|_1
	+
	\|\Omega'\|_1
	+
	\|\Omega''\|_1
	+
	\|\omega\|_1
	\Big).
	\]
	By Lemma \ref{lemma7},
	\[
	\|\Omega\|_1
	+
	\|\Omega'\|_1
	+
	\|\Omega''\|_1
	+
	\|\omega\|_1
	\le
	C
	\|\Omega\|_{H^1(S^{n-1}\times S^{m-1})}.
	\]
	This proves the desired estimate for smooth \(\Omega\).

	Finally, let \(\Omega\) be general. Choose smooth rotation averages
	\(\Omega_\nu\) which preserve the evenness and the separate cancellation
	conditions. By the rotation covariance of the restricted Riesz transforms,
	\[
	\Omega_\nu\to\Omega,\qquad
	\Omega_\nu'\to\Omega',\qquad
	\Omega_\nu''\to\Omega'',\qquad
	\omega_{\Omega_\nu}\to\omega
	\]
	in \(L^1\). Choose a subsequence for which the sum of the four $L^1$-norms
	of successive differences is finite. Applying the smooth estimate
	to these differences and using Tonelli's theorem gives, outside
	one null set of angular pairs, uniform convergence of the profiles
	and all derivatives occurring in the defining suprema, with their
	radial weights. The limits retain these derivatives and agree
	with \eqref{42} by distributional convergence. Fatou's lemma
	therefore gives
	\[
	\|\Omega_1^*\|_1
	+
	\|\Omega_2^*\|_1
	+
	\|\Omega^*\|_1
	\le
	C
	\|\Omega\|_{H^1(S^{n-1}\times S^{m-1})}.
	\]
	This completes the proof.
	\endpf

	We now prove Theorem \ref{thm2} in the even-even case. By the decomposition
	\[
	\chi=\rho+\lambda,
	\]
	the evenness of \(\Omega\) in both variables, and the definitions in
	\eqref{42}, we obtain
	\begin{eqnarray}\label{decomp:marcinkernel}
	&&	\psi_{t,s}^{\Omega}* f(x,y)\\
		&=&
		\frac14
		\Liint_{S^{n-1}\times S^{m-1}}
		\Omega(u',v')
		\Liint_{\R^2}
		\rho_t(t')\rho_s(s')
		f(x-t'u',y-s'v')
		\,dt'ds'\,du'dv'
		\nonumber\\
		&&
		-\frac14
		\Liint_{S^{n-1}\times S^{m-1}}
		\Liint_{\R^2}
		\sigma_t^{(1,u',v')}(t')\cdot
		{\mathfrak R}_{\R^n}f(x-t'u',y-s'v')
		\rho_s(s')
		\,dt'ds'\,du'dv'
		\nonumber\\
		&&
		-\frac14
		\Liint_{S^{n-1}\times S^{m-1}}
		\Liint_{\R^2}
		\rho_t(t')
		\sigma_s^{(2,u',v')}(s')\cdot
		{\mathfrak R}_{\R^m}f(x-t'u',y-s'v')
		\,dt'ds'\,du'dv'
		\nonumber\\
		&&
		+\frac14
		\Liint_{S^{n-1}\times S^{m-1}}
		\Liint_{\R^2}
		\zeta_{t,s}^{(u',v')}(t',s')
		\cdot
		\left(
		{\mathfrak R}_{\R^n}\otimes{\mathfrak R}_{\R^m}
		\right)f(x-t'u',y-s'v')
		\,dt'ds'\,du'dv' .
	\end{eqnarray}
	Here the dot denotes the natural vector or tensor contraction. The factor
	\(1/4\) comes from extending the two positive radial variables to the whole
	real line and using the evenness of \(\Omega\), while the signs follow from
$
	\mathfrak R_{\R^l}\cdot\mathfrak R_{\R^l}=-Id,
	\, l=n,m.
$

	For fixed \(u'\in S^{n-1}\) and \(v'\in S^{m-1}\), write
$$
	\R^n=L(u')\oplus L(u')^\perp,
	\qquad
	\R^m=L(v')\oplus L(v')^\perp.
$$
	Thus
	\[
	x=x_{u'}u'+x^\perp,
	\qquad
	y=y_{v'}v'+y^\perp,
	\]
	where \(x_{u'},y_{v'}\in\R\), \(x^\perp\in L(u')^\perp\), and
	\(y^\perp\in L(v')^\perp\). Define the slice
	\[
	f_{x^\perp,y^\perp}^{u',v'}(a,b)
	=
	f(x^\perp+au',\,y^\perp+bv'),
	\qquad (a,b)\in\R^2.
	\]
	We let \(K_\rho'\) and \(K_\rho''\) denote the vector-valued kernels \(K_\rho\)
	acting on the first and second one-dimensional variables, respectively.
	Similarly, \(K_{\sigma^{(1,u',v')}}'\) acts on the first variable and
	\(K_{\sigma^{(2,u',v')}}''\) acts on the second variable.

	Taking the \(\widetilde{\mathcal H}\)-norm in \eqref{decomp:marcinkernel} and
	using the triangle inequality, we get
	\begin{eqnarray}\label{45}
	&&	\mu_\Omega(f)(x,y)\nonumber\\
		&\le&
		\frac14
		\Liint_{S^{n-1}\times S^{m-1}}
		|\Omega(u',v')|
		\left|
		K_\rho'K_\rho''
		*
		f_{x^\perp,y^\perp}^{u',v'}
		(x_{u'},y_{v'})
		\right|_{\widetilde{\mathcal H}}
		\,du'dv'
		\nonumber\\
		&&+
		\frac14
		\Liint_{S^{n-1}\times S^{m-1}}
		\left|
		K_{\sigma^{(1,u',v')}}'K_\rho''
		*
		({\mathfrak R}_{\R^n}f)_{x^\perp,y^\perp}^{u',v'}
		(x_{u'},y_{v'})
		\right|_{\widetilde{\mathcal H}}
		\,du'dv'
		\nonumber\\
		&&+
		\frac14
		\Liint_{S^{n-1}\times S^{m-1}}
		\left|
		K_\rho'K_{\sigma^{(2,u',v')}}''
		*
		({\mathfrak R}_{\R^m}f)_{x^\perp,y^\perp}^{u',v'}
		(x_{u'},y_{v'})
		\right|_{\widetilde{\mathcal H}}
		\,du'dv'
		\nonumber\\
		&&+
		\frac14
		\Liint_{S^{n-1}\times S^{m-1}}
		\left|
		\widetilde K_{\zeta^{(u',v')}}
		*
		\big(
		({\mathfrak R}_{\R^n}\otimes{\mathfrak R}_{\R^m})f
		\big)_{x^\perp,y^\perp}^{u',v'}
		(x_{u'},y_{v'})
		\right|_{\widetilde{\mathcal H}}
		\,du'dv' .
	\end{eqnarray}

For a one-dimensional kernel family \(K_\sigma\), we write
\(T_\sigma'\) and \(T_\sigma''\) for the corresponding vector-valued operators
acting on the first and second variables, respectively. Thus, for
\(F\in L^p(\R^2)\),
\[
T_\sigma'F(a,b)
=
\left\{
\int_{\R}\sigma_t(a-\alpha)F(\alpha,b)\,d\alpha
\right\}_{t>0},
\]
and
\[
T_\sigma''F(a,b)
=
\left\{
\int_{\R}\sigma_t(b-\beta)F(a,\beta)\,d\beta
\right\}_{t>0}.
\]
Similarly,
\[
K_{\sigma_1}'K_{\sigma_2}''*F(a,b)
=
\left\{
\iint_{\R^2}
(\sigma_1)_t(a-\alpha)(\sigma_2)_s(b-\beta)
F(\alpha,\beta)\,d\alpha d\beta
\right\}_{t,s>0}.
\]

We shall use the following lemma.

\begin{Lemma}\label{lemma11}
	Let \(1<p<\infty\). For almost every fixed pair
	\((u',v')\in S^{n-1}\times S^{m-1}\), the following estimates hold:
	\[
	\left\|
	K_\rho'K_\rho''*F
	\right\|_{L^p(\R^2;\widetilde{\mathcal H})}
	\le
	C_p\|F\|_{L^p(\R^2)},
	\]
	\[
	\left\|
	K_{\sigma^{(1,u',v')}}'K_\rho''*F
	\right\|_{L^p(\R^2;\widetilde{\mathcal H})}
	\le
	C_p\,\Omega_1^*(u',v')\,
	\|F\|_{L^p(\R^2)},
	\]
	\[
	\left\|
	K_\rho'K_{\sigma^{(2,u',v')}}''*F
	\right\|_{L^p(\R^2;\widetilde{\mathcal H})}
	\le
	C_p\,\Omega_2^*(u',v')\,
	\|F\|_{L^p(\R^2)},
	\]
	and
	\[
	\left\|
	\widetilde T_{\zeta^{(u',v')}}F
	\right\|_{L^p(\R^2;\widetilde{\mathcal H})}
	\le
	C_p\,\Omega^*(u',v')\,
	\|F\|_{L^p(\R^2)}.
	\]
	The estimates are understood componentwise in the finite-dimensional
	vector- or tensor-valued cases.
\end{Lemma}

{\bf Proof:}
We first verify the Fourier hypothesis in Lemma \ref{lemma8}.
For the even extension of $\rho$, put $A=1$; for
$\kappa=\sigma^{(k,u',v')}$, put $A=\Omega_k^*(u',v')$.
In either case, with $\kappa=\rho$ in the first case,
\[
\int_{\mathbb R}\kappa=0,\qquad
\int_{\mathbb R}|x|^{1/2}|\kappa(x)|\,dx
+\operatorname{Var}(\kappa)\le CA.
\]
Indeed, $\rho$ is compactly supported and of bounded variation, while
the profile and derivative bounds give this estimate for $\sigma$.
Cancellation and integration by parts therefore give, respectively,
\[
|\widehat\kappa(\xi)|
\le CA\min\{|\xi|^{1/2},|\xi|^{-1}\},\qquad
\sup_{\xi\ne0}\int_0^\infty
|\widehat\kappa(t\xi)|^2\,\frac{dt}{t}\le CA^2.
\]
The fixed even kernel $\rho$ also satisfies the size and H\"ormander
conditions in Lemma \ref{lemma8}. Hence
\[
\|T_\rho G\|_{L^p(\R;\mathcal H)}
\le
C_p\|G\|_{L^p(\R)}.
\]

For \(\sigma^{(1,u',v')}\), the definition of \(\Omega_1^*\) gives
\[
|\sigma^{(1,u',v')}(r)|
\le
\Omega_1^*(u',v')(1+|r|)^{-2},
\qquad
\left|
\frac{d}{dr}\sigma^{(1,u',v')}(r)
\right|
\le
\Omega_1^*(u',v')(1+|r|)^{-3}.
\]
Since \(\sigma^{(1,u',v')}\) is odd, it has vanishing integral on \(\R\). The
preceding two estimates imply
\[
\|\sigma^{(1,u',v')}\|_1
+
B_{\sigma^{(1,u',v')}}
\le
C\,\Omega_1^*(u',v'),
\]
where \(B_{\sigma^{(1,u',v')}}\) denotes the size and smoothness constant in
Lemma \ref{lemma8}. Therefore Lemma \ref{lemma8} yields
\[
\|T_{\sigma^{(1,u',v')}}G\|_{L^p(\R;\mathcal H)}
\le
C_p\,\Omega_1^*(u',v')\,
\|G\|_{L^p(\R)}.
\]
The same argument gives
\[
\|T_{\sigma^{(2,u',v')}}G\|_{L^p(\R;\mathcal H)}
\le
C_p\,\Omega_2^*(u',v')\,
\|G\|_{L^p(\R)}.
\]

Similarly, by the definition of \(\Omega^*\),
\[
\left|
\partial_t^j\partial_s^i
\zeta^{(u',v')}(t,s)
\right|
\le
\Omega^*(u',v')
(1+|t|)^{-2-j}(1+|s|)^{-2-i},
\qquad i,j=0,1.
\]
Using cancellation in a low-frequency variable and integration by parts in a high-frequency variable, we obtain the desired estimates in the four frequency regimes; when both frequencies are large, the mixed derivative estimate is used.
These bounds give
\[
|\widehat{\zeta^{(u',v')}}(\xi,\eta)|
\le C\Omega^*(u',v')
\min\{|\xi|^{1/2},|\xi|^{-1}\}
\min\{|\eta|^{1/2},|\eta|^{-1}\}.
\]
Consequently,
\[
\sup_{\xi\eta\ne0}\int_0^\infty\int_0^\infty
|\widehat{\zeta^{(u',v')}}(t\xi,s\eta)|^2
\,\frac{dt}{t}\frac{ds}{s}
\le C\Omega^*(u',v')^2,
\]
which yields the product $L^2$ estimate by Plancherel's theorem.
Together with the separate cancellations, the profile estimates also imply
\[
\|\zeta^{(u',v')}\|_1
+
B_{\zeta^{(u',v')}}
\le
C\,\Omega^*(u',v'),
\]
where \(B_{\zeta^{(u',v')}}\) is the product Calderón--Zygmund constant in
Lemma \ref{lemma9}. Hence Lemma \ref{lemma9} gives
\[
\|\widetilde T_{\zeta^{(u',v')}}F\|_{L^p(\R^2;\widetilde{\mathcal H})}
\le
C_p\,\Omega^*(u',v')\,
\|F\|_{L^p(\R^2)}.
\]

The estimates for
\(K_\rho'K_\rho''\),
\(K_{\sigma^{(1,u',v')}}'K_\rho''\), and
\(K_\rho'K_{\sigma^{(2,u',v')}}''\) follow by applying the preceding
one-parameter estimates successively in the two variables, using their standard
Hilbert-valued extensions. This proves the lemma.
\endpf

	We now finish the proof of Theorem \ref{thm2}. By \eqref{45}, Minkowski's
	inequality, and Lemma \ref{lemma11} applied on each two-dimensional slice, we
	obtain
	\begin{eqnarray*}
		\|\mu_\Omega(f)\|_{L^p(\R^n\times\R^m)}
		&\le&
		C_p
		\bigg(
		\|\Omega\|_{L^1(S^{n-1}\times S^{m-1})}
		+
		\|\Omega_1^*\|_{L^1(S^{n-1}\times S^{m-1})}
		\\
		&&\qquad\qquad
		+
		\|\Omega_2^*\|_{L^1(S^{n-1}\times S^{m-1})}
		+
		\|\Omega^*\|_{L^1(S^{n-1}\times S^{m-1})}
		\bigg)
		\\
		&&\qquad\qquad\cdot
		\left(
		\|f\|_{L^p}
		+
		\|\mathfrak R_{\R^n}f\|_{L^p}
		+
		\|\mathfrak R_{\R^m}f\|_{L^p}
		+
		\|(\mathfrak R_{\R^n}\otimes\mathfrak R_{\R^m})f\|_{L^p}
		\right).
	\end{eqnarray*}
	Here the \(L^p\)-norms on the right-hand side are taken over
	\(\R^n\times\R^m\). Since the Riesz transforms are bounded on \(L^p\),
	\(1<p<\infty\), this gives
	\[
	\|\mu_\Omega(f)\|_{L^p(\R^n\times\R^m)}
	\le
	C_p
	\left(
	\|\Omega\|_1
	+
	\|\Omega_1^*\|_1
	+
	\|\Omega_2^*\|_1
	+
	\|\Omega^*\|_1
	\right)
	\|f\|_{L^p(\R^n\times\R^m)}.
	\]
	By Lemma \ref{lemma10}, and since
	\(\|\Omega\|_1\le C\|\Omega\|_{H^1(S^{n-1}\times S^{m-1})}\), we conclude that
	\[
	\|\mu_\Omega(f)\|_{L^p(\R^n\times\R^m)}
	\le
	C_{p,n,m}
	\|\Omega\|_{H^1(S^{n-1}\times S^{m-1})}
	\|f\|_{L^p(\R^n\times\R^m)}.
	\]
	This proves Theorem \ref{thm2} in the even-even case.

For the remaining components, put $h(r)=\operatorname{sgn}(r)\chi(|r|)$.
This odd kernel has zero integral and satisfies Lemma \ref{lemma8}.
The odd-odd case uses $h\otimes h$, so two applications of that lemma
give a bound by $C_p\|\Omega_{o,o}\|_1\|f\|_p$. In the odd-even case,
split only the second radial factor as $\chi=\rho+\lambda$. The two
kernels are $h\otimes\rho$ and $h\otimes\sigma^{(2,u',v')}$, acting
on $f$ and $\mathfrak R_{\mathbb R^m}f$, respectively, with the profile
in \eqref{42} formed from $\Omega_{o,e}$. The one-variable profile
estimate in the proof of Lemma \ref{lemma10} bounds the coefficient
integrals by $C(\|\Omega_{o,e}\|_1+\|\omega_{o,e}\|_1)$.
The even-odd case follows by interchanging the variables. Thus, for
$(\alpha,\beta)\ne(e,e)$,
\[
\|\mu_{\Omega_{\alpha,\beta}}(f)\|_p
\le C_{p,n,m}
\big(\|\Omega_{\alpha,\beta}\|_1+\|\omega_{\alpha,\beta}\|_1\big)
\|f\|_p.
\]
Lemma \ref{lemma7} and the parity decomposition complete the proof.

	\begin{Remark}\label{rmk:4.5}
		The role of the decomposition \(\chi=\rho+\lambda\) is to separate the part
		which already has cancellation from the part which does not. Indeed, the term
		involving \(\rho_t\rho_s\) has cancellation in both variables, while the
		one-parameter kernels associated with \(\lambda\) satisfy the size and
		smoothness estimates in Lemma \ref{lemma8} but do not have cancellation. For
		the remaining terms involving \(\lambda\), we introduce
		\(\sigma^{(1,u',v')}\), \(\sigma^{(2,u',v')}\), and
		\(\zeta^{(u',v')}\) by applying Riesz transforms. These auxiliary kernels are
		odd in their corresponding radial variables, and hence have the cancellation
		required by Lemmas \ref{lemma8} and \ref{lemma9}. This explains the
		decomposition used in the proof above.
	\end{Remark}

	\section{Characterizations of \texorpdfstring{$H^1(S^{l-1})$}{H1 on the sphere}}\label{sec:h1one}
We recall and compare several characterizations of $H^1(S^{l-1})$,
defined in \eqref{def:h1manifoldd}. We use the results in
\cite{CL,Colzanithesis,Colzani,chineseequivalence,2022MMV,RW,AL}.
The symbol \(\nabla\) denotes the Euclidean gradient in \(\mathbb R^l\),
whereas \(\nabla_{S^{l-1}}\) denotes the intrinsic gradient on the
sphere \(S^{l-1}\). When a kernel \(K(\theta,\eta)\) is defined on
\(S^{l-1}\times S^{l-1}\), we write \(\nabla_\theta\) and
\(\nabla_\eta\) for the intrinsic derivatives with respect to the first
and the second spherical variables, respectively. Higher order
derivatives such as
\(\nabla_\theta^\alpha\nabla_\eta^\beta K(\theta,\eta)\) are understood
locally in coordinate charts on the sphere. Since \(S^{l-1}\) is compact,
the corresponding estimates are independent of the chosen charts up to
constants.

Let \(C^{S^{l-1}}\) denote the Poisson extension on the unit ball
\(B^l=\{z\in\R^l:|z|<1\}\):
\[
C^{S^{l-1}}(f)(z)
=
\int_{S^{l-1}} f(u')p(z,u')\,du',
\qquad
p(z,u')
=
c_l\frac{1-|z|^2}{|z-u'|^l}.
\]
Here \(c_l>0\) is the normalizing dimensional constant. We recall the following
classical maximal and square functions on \(S^{l-1}\):
\begin{eqnarray}
	C_{+}^{S^{l-1}}(f)(z')
	&=&
	\sup_{0<r<1}
	\left|
	C^{S^{l-1}}(f)(rz')
	\right|,
	\nonumber\\
	C_{*}^{S^{l-1}}(f)(z')
	&=&
	\sup_{w\in\Theta_{z'}^{S^{l-1}}}
	\left|
	C^{S^{l-1}}(f)(w)
	\right|,
	\label{def:nontangentialmax}\\
	S_{\gamma}^{S^{l-1}}(f)(z')
	&=&
	\left(
	\int_{\Theta_{z',\gamma}^{S^{l-1}}}
	\left|
	\nabla_w C^{S^{l-1}}(f)(w)
	\right|^2
	\frac{dw}{(1-|w|)^{l-2}}
	\right)^{1/2},
	\qquad
	1/2\le\gamma<1.
	\nonumber
\end{eqnarray}
Here \(\Theta_{z',\gamma}^{S^{l-1}}\) denotes the convex hull of
\(\gamma B^l\) and \(z'\). We write
\[
\Theta_{z'}^{S^{l-1}}
=
\Theta_{z',1/2}^{S^{l-1}},
\qquad
S^{S^{l-1}}(f)
=
S_{1/2}^{S^{l-1}}(f).
\]

The conjugate Riesz transform \({\mathcal R}_{S^{l-1}}\) on \(S^{l-1}\) is
defined by the principal value integral
\begin{eqnarray}\label{def:conjugate}
	{\mathcal R}_{S^{l-1}}(f)(x')
	=
	\operatorname{p.v.}\int_{S^{l-1}}H(x',y')f(y')\,dy',
\end{eqnarray}
where
\[
H(x',y')
=
c_l\,l\,
\left(
y'-(y'\cdot x')x'
\right)
\int_0^1
\frac{1-r^2}{|y'-rx'|^{l+2}}
\,dr .
\]
See \cite{RW}. Here \(y'-(y'\cdot x')x'\) is the tangential component of
\(y'\) at \(x'\).

For \(f\in L^1(S^{l-1})\), the principal value above is first understood in
the sense of distributions. More precisely, for
\(\varphi\in C^\infty(S^{l-1};\mathbb R^l)\), we define
\[
\left\langle
{\mathcal R}_{S^{l-1}}f,\varphi
\right\rangle
=
\int_{S^{l-1}}
f(y')
\left[
\operatorname{p.v.}\int_{S^{l-1}}
H(x',y')\cdot\varphi(x')\,dx'
\right]dy'.
\]
If this distribution is represented by an \(L^1\)-vector field, we still denote
that representative by \({\mathcal R}_{S^{l-1}}f\).

We shall also use the following abstract notation. Let \(D\) be a complete
Riemannian manifold with non-negative Ricci curvature. Let \(P_t^D\) be the
Poisson semigroup on \(D\), and define
\[
\Gamma_x^D
=
\{(y,t)\in D\times\R_+:\rho_D(x,y)<t\},
\]
where \(\rho_D\) is the geodesic distance on \(D\). We define
\begin{eqnarray}\label{def:maninar}
	\begin{array}{lll}
		P_D^*(f)(x)
		&=&
		\displaystyle
		\sup_{(y,t)\in\Gamma_x^D}
		|P_t^D f(y)|,
		\\[6pt]
		A_D(f)(x)
		&=&
		\displaystyle
		\left(
		\int_{\Gamma_x^D}
		\left|
		\nabla_D^\perp P_t^D f(y)
		\right|^2
		\frac{t\,dt\,d_Dy}{|B_D(x,t)|}
		\right)^{1/2},
		\\[8pt]
		{\mathfrak R}_D(f)
		&=&
		\displaystyle
		\nabla_D(-\Delta_D)^{-1/2}f,
	\end{array}
\end{eqnarray}
where $
\nabla_D^\perp=(\nabla_D,\partial_t).$
The following is the spherical case of Chen--Luo's characterization.

\begin{Proposition}\label{proposition3}(\cite[Theorems 1 and 3, p.~24]{CL})
	Let \(f\in L_0^1(S^{l-1})\).
	Then
	\[
	\|f\|_{H^1}
	\cong
	\|P_{S^{l-1}}^*(f)\|_1
	\cong
	\|A_{S^{l-1}}(f)\|_1
	\cong
	\|{\mathfrak R}_{S^{l-1}}'(f)\|_1
	\cong
	\|f\|_{H_a^1},
	\]
	where
	\begin{eqnarray}\label{def:rieszp}
		{\mathfrak R}_D'
		=
		({\mathfrak R}_D,Id),
	\end{eqnarray}
where \(H_a^1\) is the atomic Hardy space recalled below.
\end{Proposition}

We also recall the atomic Hardy space on \(S^{l-1}\); see \cite{Colzani}.
Fix \(r_*>0\) smaller than the injectivity radius of \(S^{l-1}\). A function
\(a\in L^\infty(S^{l-1})\) is called an \(H^1\)-atom if one of the following
conditions holds. (i) There exists a geodesic ball
\(B=B_{S^{l-1}}(z_0',r)\), \(0<r<r_*\), such that
\[
\operatorname{supp}a\subset B,
\qquad
\int_{S^{l-1}}a(z')\,dz'=0,
\qquad
\|a\|_{L^\infty(S^{l-1})}\le |B|^{-1}.
\]
(ii) There exists a geodesic ball
\(B=B_{S^{l-1}}(z_0',r_*)\) such that
\[
\operatorname{supp}a\subset B,
\qquad
\|a\|_{L^\infty(S^{l-1})}\le |B|^{-1}.
\]
In this case no cancellation condition is imposed.

The atomic Hardy space \(H_a^1(S^{l-1})\) consists of all
\(f\in L^1(S^{l-1})\) admitting a representation
\[
f=\sum_{j=1}^{\infty}\lambda_j a_j
\qquad\text{in }L^1(S^{l-1}),
\]
where \(a_j\) are \(H^1\)-atoms and
$
\sum_{j=1}^{\infty}|\lambda_j|<\infty.
$
Its norm is
\[
\|f\|_{H_a^1(S^{l-1})}
=
\inf
\left\{
\sum_{j=1}^{\infty}|\lambda_j|:
f=\sum_{j=1}^{\infty}\lambda_j a_j
\right\}.
\]

For later reference, we introduce the augmented Riesz transforms
\[
\widetilde R_{S^{l-1}}'
=
(\widetilde R_{S^{l-1}},Id),
\qquad
{\mathcal R}_{S^{l-1}}'
=
({\mathcal R}_{S^{l-1}},Id),
\qquad
{\mathfrak R}_{S^{l-1}}'
=
({\mathfrak R}_{S^{l-1}},Id),
\]
where \(Id\) denotes the identity operator. Thus, for example,
\[
\|{\mathcal R}_{S^{l-1}}'(f)\|_1
\cong
\|{\mathcal R}_{S^{l-1}}(f)\|_1+\|f\|_1.
\]

Proposition \ref{proposition3} gives the spectral Poisson and atomic
characterizations. The ball-Poisson radial and non-tangential maximal
functions have equivalent $L^1$ norms \cite{Colzani}; for the spherical
Riesz characterization, see \cite{RW}. For mean-zero data,
\cite[Theorem 1.1]{chineseequivalence} gives
\[
\|C_*^{S^{l-1}}f\|_1\cong\|S^{S^{l-1}}f\|_1.
\]
The mean-zero condition ensures $C^{S^{l-1}}f(0)=0$, as used in that
proof. Propositions \ref{thm:tworiesz} and \ref{prop:tworadial} below
give explicit comparisons of the Riesz transforms and the two radial
maximal functions. Together these yield the following equivalences.

	\begin{Theorem}\label{thm:onecharacter}
	For \(f\in L_0^1(S^{l-1})\), we have
	\begin{eqnarray}\label{eq:onecharacter}
		\begin{aligned}
			\|f\|_{H_a^1}
			&\cong
			\|C_{+}^{S^{l-1}}(f)\|_{1}
			\cong
			\|C_{*}^{S^{l-1}}(f)\|_{1}
			\cong
			\|S^{S^{l-1}}(f)\|_1
			\cong
			\|\widetilde{R}_{S^{l-1}}'(f)\|_1 \\
			\cong
			\|{\mathcal R}_{S^{l-1}}'(f)\|_{1}
			&\cong
			\|{\mathfrak R}_{S^{l-1}}'(f)\|_{1}
			\cong
			\|A_{S^{l-1}}(f)\|_{1}
			\cong
			\|P_{S^{l-1}}^*(f)\|_{1}
			\cong
			\|P_{S^{l-1}}^{+}(f)\|_{1}.
		\end{aligned}
	\end{eqnarray}
\end{Theorem}

\begin{Remark}
	For related work on spherical Hardy kernels and singular integrals,
see \cite{RR,2000stefanove}.
\end{Remark}
{\bf Proof of Theorem \ref{thm:onecharacter}:}
Combine Proposition \ref{proposition3} and the ball-Poisson
maximal and area equivalences above with Propositions \ref{thm:tworiesz}
and \ref{prop:tworadial}. This gives \eqref{eq:onecharacter}.
\endpf

We shall also use the Hilbert-valued form of the one-parameter
area--Riesz equivalences. It follows directly from the scalar results by
randomization: for a finite orthonormal expansion
$g=\sum_{j=1}^J g_j e_j$, apply the scalar equivalence to
$\sum_j\varepsilon_j g_j$ and average over independent Rademacher signs.
The Khintchine inequalities identify the averages of the norms of the
linear Riesz transforms and of the linear cone-valued square-function
operators with the corresponding Hilbert-valued norms. The constants
are independent of $J$; finite-rank approximation gives the result for
an arbitrary Hilbert space. In particular, for Hilbert-valued mean-zero
$g$,
\begin{equation}\label{eq:hilbert-area-riesz}
\|S^{S^{l-1}}g\|_1
\cong\|\mathcal R_{S^{l-1}}'g\|_1
\cong\|\mathfrak R_{S^{l-1}}'g\|_1\cong\|A_{S^{l-1}}g\|_1..
\end{equation}
The same argument extends \cite[Theorem 3]{CL} to Hilbert-valued
mean-zero data; this extension is used in Proposition \ref{proposition13}.

	For the Riesz comparison, introduce the boundary operator
	\begin{equation}\label{eq:xif}
		\Xi(f)(x')
		=
		f(x')x'
		+
		{\mathcal R}_{S^{l-1}}f(x').
	\end{equation}
	This operator is associated with the conjugate harmonic system in the unit ball:
	if \(h\) solves the Neumann problem in \(B^l\) with boundary normal derivative
	\(f\), then formally
	\[
	\nabla h(x')
	=
	f(x')x'
	+
	{\mathcal R}_{S^{l-1}}f(x')
	=
	\Xi(f)(x').
	\]

	\begin{Remark}
		Other variants of Riesz transforms on the sphere are discussed in
		\cite{A1,AL,St1}.
	\end{Remark}

	\begin{Proposition}\label{thm:tworiesz}
		For \(f\in L_0^1(S^{l-1})\), one has
		\begin{eqnarray}\label{eq:tworiesz}
			\|{\mathcal R}_{S^{l-1}}'(f)\|_1
			\cong
			\|\widetilde R_{S^{l-1}}'(f)\|_1
			\cong
			\|{\mathfrak R}_{S^{l-1}}'(f)\|_1.
		\end{eqnarray}
		Moreover, whenever these equivalent quantities are finite,
\(\Xi(f)\in L^1(S^{l-1};\mathbb C^l)	\)	 and
		\[
		\|\Xi(f)\|_1
		\cong
		\|{\mathcal R}_{S^{l-1}}'(f)\|_1.
		\]
	\end{Proposition}

	{\bf Proof of Proposition \ref{thm:tworiesz}:}
	The equivalence between the augmented restricted and conjugate Riesz
	transforms follows from \cite[Theorem 2]{RW}. Thus it remains to compare
	the conjugate Riesz transform \({\mathcal R}_{S^{l-1}}\) with the spectral
	Riesz transform \({\mathfrak R}_{S^{l-1}}\).

	Let \(Y_k\) be a spherical harmonic of degree \(k\ge1\), and put
	\[
	a=l-2.
	\]
	Since
	\[
	-\Delta_{S^{l-1}}Y_k=k(k+a)Y_k,
	\]
	the spectral Riesz transform satisfies
	\[
	{\mathfrak R}_{S^{l-1}}Y_k
	=
	\frac{\nabla_{S^{l-1}}Y_k}{\sqrt{k(k+a)}}.
	\]
	On the other hand, the harmonic solution of the Neumann problem with
	boundary datum \(Y_k\) is
	\[
	h(r,\theta)=\frac{r^k}{k}Y_k(\theta).
	\]
	Hence its tangential boundary gradient gives
	\[
	{\mathcal R}_{S^{l-1}}Y_k
	=
	\frac1k\nabla_{S^{l-1}}Y_k.
	\]
	Therefore
	\begin{equation}\label{connect:RR}
		\left(
		{\mathcal R}_{S^{l-1}}
		-
		{\mathfrak R}_{S^{l-1}}
		\right)Y_k
		=
		\left(
		\frac1k-\frac1{\sqrt{k(k+a)}}
		\right)
		\nabla_{S^{l-1}}Y_k.
	\end{equation}
	In particular, when \(l=2\), we have \(a=0\), and the two transforms coincide.

	For \(a\ge0\) and \(0<r<1\), define
	\[
	w_a(r)
	=
	1-
	\frac1\pi
	\int_0^1
	\frac{r^{au}}{\sqrt{u(1-u)}}\,du.
	\]
	We shall use the identity
	\begin{equation}\label{eq:beta-identity}
		\frac1\pi
		\int_0^1
		\frac{du}{(k+au)\sqrt{u(1-u)}}
		=
		\frac1{\sqrt{k(k+a)}}.
	\end{equation}
	Indeed, after the change of variables
	$
	u=\frac{1-\cos\vartheta}{2},
	$
	the left-hand side of \eqref{eq:beta-identity} becomes
	\[
	\frac1\pi
	\int_0^\pi
	\frac{d\vartheta}
	{k+\frac a2-\frac a2\cos\vartheta}
	=
	\frac1{\sqrt{k(k+a)}}.
	\]
	Consequently,
	\[
	\begin{split}
		\int_0^1 r^{k-1}w_a(r)\,dr
		&=
		\frac1k
		-
		\frac1\pi
		\int_0^1
		\frac1{\sqrt{u(1-u)}}
		\left(
		\int_0^1r^{k+au-1}\,dr
		\right)du
		\\
		&=
		\frac1k-\frac1{\sqrt{k(k+a)}}.
	\end{split}
	\]

	Write
$
	C_rg(\theta)
	=
	C^{S^{l-1}}g(r\theta),
\,0<r<1.
$
	Since
$
	C_rY_k=r^kY_k,
$
	we have
	\[
	\nabla_{S^{l-1}}C_rY_k
	=
	r^k\nabla_{S^{l-1}}Y_k.
	\]
	Hence \eqref{connect:RR} gives
	\begin{equation}\label{eq:spectral-riesz-difference}
		\left(
		{\mathcal R}_{S^{l-1}}
		-
		{\mathfrak R}_{S^{l-1}}
		\right)g
		=
		\int_0^1
		w_a(r)\nabla_{S^{l-1}}C_rg\,\frac{dr}{r}
	\end{equation}
	for finite spherical harmonic sums \(g\) of mean zero, and therefore also
	for smooth mean-zero \(g\).

	We next show that the operator on the right-hand side of
	\eqref{eq:spectral-riesz-difference} is bounded on \(L^1(S^{l-1})\).
	First,
	\[
	0\le w_a(r)\le1,
	\qquad 0<r<1.
	\]
	Moreover, for \(1/2\le r<1\),
	\[
	w_a(r)
	=
	\frac1\pi
	\int_0^1
	\frac{1-r^{au}}{\sqrt{u(1-u)}}\,du
	\le
	C_a(1-r).
	\]
	Direct differentiation of the ball Poisson kernel gives
	\[
	\sup_{\eta\in S^{l-1}}
	\int_{S^{l-1}}
	\left|
	\nabla_\theta p(r\theta,\eta)
	\right|
	\,d\theta
	\le
	\begin{cases}
		C_l r,
		&0<r\le1/2,\\[3pt]
		C_l(1-r)^{-1},
		&1/2<r<1.
	\end{cases}
	\]
	Therefore, if
	\[
	D_lg
	=
	\int_0^1
	w_a(r)\nabla_{S^{l-1}}C_rg\,\frac{dr}{r},
	\]
	then, by Fubini's theorem,
	\[
	\begin{split}
		\|D_lg\|_1
		&\le
		C_l\|g\|_1
		\left[
		\int_0^{1/2}
		w_a(r)\frac{r}{r}\,dr
		+
		\int_{1/2}^1
		w_a(r)\frac{1}{r(1-r)}\,dr
		\right]\le
		C_l\|g\|_1.
	\end{split}
	\]
	Thus \(D_l\) extends to a bounded operator on \(L^1(S^{l-1})\).

	We now pass from smooth functions to general
	\(g\in L_0^1(S^{l-1})\). For \(0<\rho<1\), the function \(C_\rho g\) is
	smooth and has mean zero, so
	\[
	\left(
	{\mathcal R}_{S^{l-1}}
	-
	{\mathfrak R}_{S^{l-1}}
	\right)C_\rho g
	=
	D_l(C_\rho g).
	\]
	Since
	\[
	C_\rho g\longrightarrow g
	\qquad\text{in }L^1(S^{l-1})
	\]
	as \(\rho\uparrow1\), and \(D_l\) is bounded on \(L^1\),
	\[
	D_l(C_\rho g)\longrightarrow D_lg
	\qquad\text{in }L^1(S^{l-1}).
	\]
	Passing to the limit in the sense of distributions gives
	\begin{equation}\label{eq:Riesz-difference-L1}
		{\mathcal R}_{S^{l-1}}g
		-
		{\mathfrak R}_{S^{l-1}}g
		=
		D_lg.
	\end{equation}

	It follows from \eqref{eq:Riesz-difference-L1} that
	\[
	\|{\mathcal R}_{S^{l-1}}g\|_1
	\le
	\|{\mathfrak R}_{S^{l-1}}g\|_1
	+
	C_l\|g\|_1,
	\]
	and conversely,
	\[
	\|{\mathfrak R}_{S^{l-1}}g\|_1
	\le
	\|{\mathcal R}_{S^{l-1}}g\|_1
	+
	C_l\|g\|_1.
	\]
	After adjoining the identity operator, we obtain
	\[
	\|{\mathcal R}_{S^{l-1}}'(g)\|_1
	\cong
	\|{\mathfrak R}_{S^{l-1}}'(g)\|_1.
	\]
	Together with \cite[Theorem 2]{RW}, this proves
	\eqref{eq:tworiesz}.

	Finally, \({\mathcal R}_{S^{l-1}}g(\theta)\) is tangent to
	\(S^{l-1}\) at \(\theta\). Hence
	\[
	\theta\cdot
	{\mathcal R}_{S^{l-1}}g(\theta)
	=
	0,
	\]
	and therefore
	\[
	\begin{split}
		|\Xi(g)(\theta)|^2
		&=
		\left|
		g(\theta)\theta
		+
		{\mathcal R}_{S^{l-1}}g(\theta)
		\right|^2
		\\
		&=
		|g(\theta)|^2
		+
		|{\mathcal R}_{S^{l-1}}g(\theta)|^2.
	\end{split}
	\]
	Thus, with the finite-dimensional norm convention used in this paper,
	\[
	\|\Xi(g)\|_1
	\cong
	\|{\mathcal R}_{S^{l-1}}'(g)\|_1.
	\]
	This completes the proof.
	\endpf



For the approximation argument in Proposition \ref{prop:mulriszarea},
we record that $\mathcal R_{S^{l-1}}$ is of weak type $(1,1)$.
Indeed, its kernel satisfies, with $d=|\theta-\eta|$,
\[
|H(\theta,\eta)|\le C_l d^{1-l},\qquad
|\nabla_\theta H(\theta,\eta)|+|\nabla_\eta H(\theta,\eta)|
\le C_l d^{-l}.
\]
The spherical harmonic calculation above gives
$\|\mathcal R_{S^{l-1}}g\|_2\le\sqrt{l-1}\,\|g\|_2$,
and the transform annihilates constants. The weak $(1,1)$ estimate
therefore follows from the Calder\'on--Zygmund theorem on the sphere;
see \cite{KV1}.

				Next, we establish the equivalence between the \(L^1\)-norms of the two
				radial maximal functions.

		\begin{Proposition}\label{prop:tworadial}
			Let
			\[
			K_t^C(\theta,\eta)
			=
			c_l\frac{1-r^2}{|\eta-r\theta|^l},
			\qquad t=1-r,\quad 0<t\le1.
			\]
			Then \(\{K_t^C\}_{0<t\le1}\) is an approximation of the identity of exponent
			\(1\) on \(S^{l-1}\) in the sense of Martini--Meda--Vallarino. Consequently, for every $f\in L^1(S^{l-1})$,
			\[
			\|C_{+}^{S^{l-1}}(f)\|_{1}
			\cong
			\|P_{S^{l-1}}^{+}(f)\|_1.
			\]
				\end{Proposition}
Clearly, \[ \int_{S^{l-1}}K_t^C(\theta,\eta)\,d\eta=1. \]

				{\bf Proof of Proposition \ref{prop:tworadial}:}
				Write $\rho=\rho_{S^{l-1}}$ and $d=l-1$. We verify the
approximation-of-the-identity conditions in
\cite[Definitions 4.1--4.2]{2022MMV} with exponent $1$:
size, diagonal lower bound, regularity, and an integrable global remainder.
For regularity we use the sufficient gradient bound in
\cite[Remark 4.4]{2022MMV},
\[
|\nabla_yK(t,x,y)|\le Ct^{-d-1}
\left(1+\frac{\rho(x,y)}{t}\right)^{-d-2}.
\]

				First observe that
				\begin{eqnarray}
					|\eta-r\theta|^2
					=
					(1-r)^2+r|\eta-\theta|^2
					=
					t^2+r|\eta-\theta|^2.
				\end{eqnarray}
				Since \(S^{l-1}\) is compact and \(|\eta-\theta|\sim\rho(\eta,\theta)\)
				uniformly on \(S^{l-1}\), it follows that
				\begin{eqnarray}\label{es:basic}
					|\eta-r\theta|
					\sim
					t+\rho(\theta,\eta),
					\qquad 0<t\le1.
				\end{eqnarray}
				Indeed, for \(0<t\le1/2\) this follows directly from \(r\ge1/2\), while
				for \(1/2<t\le1\) both sides are uniformly comparable to \(1\). Since
				\(1-r^2\le3t\), we get
				\begin{eqnarray}\label{es:ktsize}
					K_t^C(\theta,\eta)
					\lesssim
					\frac{t}{(t+\rho(\theta,\eta))^l}
					\lesssim
					t^{-(l-1)}
					\left(1+\frac{\rho(\theta,\eta)}{t}\right)^{-(l-1)-1}.
				\end{eqnarray}
				This is the required size estimate with exponent \(\gamma=1\).

				On the diagonal,
				\begin{eqnarray}\label{es:kerneldiag}
					K_t^C(\theta,\theta)
					=
					c_l(1+r)(1-r)^{-l+1}
					=
					c_l(1+r)t^{-l+1}
					\gtrsim
					t^{-(l-1)}.
				\end{eqnarray}
				This gives the diagonal lower bound.

				Now we prove the regularity estimate. Set
			$
				Q(\theta,\eta)=|\eta-r\theta|^2.
	$
				Then
				\[
				K_t^C(\theta,\eta)
				=
				c_l(1-r^2)Q(\theta,\eta)^{-l/2}.
				\]
				If \(Y\) is a unit tangent vector field in the \(\eta\)-variable, then
				\[
				YQ=2(\eta-r\theta)\cdot Y\eta,
				\]
				and therefore
				\begin{eqnarray}
					|YQ|\lesssim|\eta-r\theta|=Q^{1/2}.
				\end{eqnarray}
				It follows that
				\[
				|YK_t^C(\theta,\eta)|
				\lesssim
				(1-r^2)Q^{-(l+1)/2}.
				\]
				Using \(1-r^2\le3t\) and \eqref{es:basic}, we obtain
				\begin{eqnarray}\label{es:gradient}
					|\nabla_{\eta}K_t^C(\theta,\eta)|
					\lesssim
					\frac{t}{(t+\rho(\theta,\eta))^{l+1}}
					\lesssim
					t^{-(l-1)-1}
					\left(1+\frac{\rho(\theta,\eta)}{t}\right)^{-(l-1)-2}.
				\end{eqnarray}
				This gives the required regularity.

				It remains to fit the kernel into the local-global decomposition required
				in Definition 4.2 of \cite{2022MMV}. Choose \(\epsilon>0\) smaller than half the
				injectivity radius of \(S^{l-1}\), and let
				\(\chi\in C^\infty([0,\pi])\) satisfy
				\[
				0\le\chi\le1,\qquad
				\chi(s)=1\quad(0\le s\le\epsilon),
				\qquad
				\chi(s)=0\quad(s\ge2\epsilon).
				\]
				Set
				\[
				K_t^{(1)}(\theta,\eta)
				=
				\chi(\rho(\theta,\eta))K_t^C(\theta,\eta),
				\]
				and
				\[
				K_t^{(2)}(\theta,\eta)
				=
				(1-\chi(\rho(\theta,\eta)))K_t^C(\theta,\eta).
				\]
				The local part \(K_t^{(1)}\) is supported where
				\(\rho(\theta,\eta)\le2\epsilon\). Estimates \eqref{es:ktsize},
				\eqref{es:kerneldiag}, and \eqref{es:gradient} show that \(K_t^{(1)}\)
				satisfies the size, diagonal lower bound, and regularity estimates.
				The derivative falling on the cutoff is bounded by
				\[
				K_t^C(\theta,\eta)
				|\nabla_{\eta}(\chi(\rho(\theta,\eta)))|
				\lesssim_{\epsilon}
				\frac{t}{(t+\rho(\theta,\eta))^{l}}
				\lesssim_{\epsilon}
				\frac{t}{(t+\rho(\theta,\eta))^{l+1}},
				\]
				because \(\nabla\chi(\rho(\theta,\eta))\) is supported where
				\(\epsilon\le\rho(\theta,\eta)\le2\epsilon\). Hence \(K_t^{(1)}\) is a
				local approximation of the identity of exponent \(1\) on \(S^{l-1}\).

				For the global part, on the support of \(K_t^{(2)}\) one has
				\(\rho(\theta,\eta)\ge\epsilon\). Therefore, by \eqref{es:ktsize},
				\[
				|K_t^{(2)}(\theta,\eta)|
				\lesssim_{\epsilon}
				1
				\]
				uniformly for \(0<t\le1\). Hence
				\[
				\operatorname*{ess\,sup}_{\eta\in S^{l-1}}
				\int_{S^{l-1}}
				\sup_{0<t\le1}|K_t^{(2)}(\theta,\eta)|\,d\theta
				<\infty.
				\]
				This is precisely the uniform integrability condition for the global part.
				Thus
				\[
				K_t^C=K_t^{(1)}+K_t^{(2)}
				\]
				is an approximation of the identity of exponent \(1\) on \(S^{l-1}\).

By \cite[Proposition 5.3]{2022MMV}, with $\alpha=1/2$, the Poisson
semigroup $\{P_t^{S^{l-1}}\}_{0<t\le1}$ is also an approximation of the
identity of exponent $1$. Thus
\cite[Theorem 4.15 and Corollary 5.4]{2022MMV} give
			\begin{eqnarray}\label{eq:localradialcompare}
				\|f\|_1+\|C_{+}^{S^{l-1}}(f)\|_1
				\cong
				\|f\|_1+
				\left\|
				\sup_{0<t\le1}|P_t^{S^{l-1}}f|
				\right\|_1 .
			\end{eqnarray}
				Since both \(K_t^C\) and \(P_t^{S^{l-1}}\) are normalized approximations to
				the identity, their maximal functions dominate \(|f|\) a.e. Therefore the
				\(L^1\)-terms can be absorbed, and
				\[
				\|C_{+}^{S^{l-1}}(f)\|_1
				\cong
				\left\|
				\sup_{0<t\le1}|P_t^{S^{l-1}}(f)|
				\right\|_1.
				\]

				Finally, we pass from the local radial maximal function to the full radial
			maximal function
			\[
			P_{S^{l-1}}^{+}(f)
			=
			\sup_{t>0}|P_t^{S^{l-1}}(f)|.
			\]
			For \(t>1\),
		$
			P_t^{S^{l-1}}
			=
			P_1^{S^{l-1}}P_{t-1}^{S^{l-1}}.
		$
			The operator \(P_1^{S^{l-1}}\) has a smooth bounded kernel on the compact
			sphere, and \(P_{t-1}^{S^{l-1}}\) is \(L^1\)-contractive. Hence
			\[
			\sup_{t>1}|P_t^{S^{l-1}}(f)(x)|
			\lesssim
			\|f\|_1.
			\]
			Thus
			\[
			\left\|
			\sup_{t>1}|P_t^{S^{l-1}}f|
			\right\|_1
			\lesssim
			\|f\|_1
			\lesssim
			\left\|
			\sup_{0<t\le1}|P_t^{S^{l-1}}f|
			\right\|_1.
			\]
			Together with \eqref{eq:localradialcompare}, this proves
			\[
			\|C_{+}^{S^{l-1}}(f)\|_1
			\cong
			\|P_{S^{l-1}}^{+}(f)\|_1.
			\]
			\endpf

	\section{Proof of Theorem \ref{thm3}}\label{sec:product}

	In this section the product Hardy space \(H^1(N\times M)\) is scalar-valued.
	Hilbert-valued estimates will only be used through one-parameter results after
	one of the two variables has been fixed.

	We now turn to the proof of Theorem \ref{thm3}. We first record some product
	estimates in a general geometric setting, which will later be applied to
	\[
	N=S^{n-1},
	\qquad
	M=S^{m-1}.
	\]
	Let \(N\) and \(M\) be complete Riemannian manifolds with non-negative Ricci
	curvature. For the auxiliary manifold estimates below, assume also
\begin{equation}\label{eq:unit-ball-volume-lower-bounds}
\inf_{x\in N}|B_N(x,1)|>0,
\qquad
\inf_{y\in M}|B_M(y,1)|>0.
\end{equation}
These bounds hold for compact factors, in particular for the spheres
to which the estimates will be applied.
	Let \(P_t^N\) and \(P_s^M\) be the Poisson semigroups on \(N\) and
	\(M\), respectively. For \(x\in N\) and \(y\in M\), set
	\[
	\Gamma_x^N
	=
	\{(u,t)\in N\times\R_+:\rho_N(x,u)<t\},
	\qquad
	\Gamma_y^M
	=
	\{(v,s)\in M\times\R_+:\rho_M(y,v)<s\}.
	\]
	For \(f\in L^1(N\times M)\), define
	\begin{eqnarray}\label{def:prona}
		\begin{aligned}
			P^+(f)(x,y)
			&=
			\sup_{t,s>0}
			\left|
			(P_t^N\otimes P_s^M)(f)(x,y)
			\right|,
			\\
			P^*(f)(x,y)
			&=
			\sup_{\substack{(u,t)\in\Gamma_x^N\\ (v,s)\in\Gamma_y^M}}
			\left|
			(P_t^N\otimes P_s^M)(f)(u,v)
			\right|,
			\\
			A(f)(x,y)
			&=
			\left(
			\Liint_{\Gamma_x^N\times\Gamma_y^M}
			\left|
			\left(
			\nabla_N^\perp\otimes\nabla_M^\perp
			\right)
			(P_t^N\otimes P_s^M)(f)(u,v)
			\right|^2
			\frac{t\,dt\,d_Nu}{|B_N(x,t)|}
			\frac{s\,ds\,d_Mv}{|B_M(y,s)|}
			\right)^{1/2}.
		\end{aligned}
	\end{eqnarray}
	Here
	\[
	\nabla_N^\perp=(\nabla_N,\partial_t),
	\qquad
	\nabla_M^\perp=(\nabla_M,\partial_s),
	\]
	and the norm in the definition of \(A(f)\) is the natural Hilbert norm in the
	finite-dimensional tensor product fiber. Recall that the product Hardy norm is
	defined by
	\[
	\|f\|_{H^1(N\times M)}
	=
	\|P^+(f)\|_{L^1(N\times M)}.
	\]

	The following estimate is a direct consequence of the product area-maximal
	estimate and the sub-mean inequality.

	\begin{Proposition}\label{proposition5}
		Let \(N\) and \(M\) be complete Riemannian manifolds with
		non-negative Ricci curvature, and suppose that
		\[
		\inf_{x\in N}|B_N(x,1)|>0,
		\qquad
		\inf_{y\in M}|B_M(y,1)|>0.
		\]
		Then, for every \(f\in H^1(N\times M)\),
		\[
		\|A(f)\|_{L^1(N\times M)}
		\le
		C\|f\|_{H^1(N\times M)}.
		\]
	\end{Proposition}
	
	{\bf Proof:}
	The product area--maximal estimate follows directly from
	\cite[Theorem 4, p.~275]{CWn1}. Indeed, the function
	\[
	U(u,t;v,s)
	=
	(P_t^N\otimes P_s^M)f(u,v)
	\]
	is multiharmonic in the sense of that theorem. Moreover, in the
	notation of \cite{CWn1},
	\[
	V_{(x,y)}((1,1))
	=
	|B_N(x,1)|\,|B_M(y,1)|,
	\]
	and hence
	\[
	\inf_{(x,y)\in N\times M}V_{(x,y)}((1,1))
	=
	\left(\inf_{x\in N}|B_N(x,1)|\right)
	\left(\inf_{y\in M}|B_M(y,1)|\right)>0.
	\]
	Thus the volume hypothesis of that theorem is satisfied. Taking
	\(d=2\), \(p=1\), and \(\bar a=\bar b=(1,1)\), we obtain
	\[
	\|A(f)\|_{L^1(N\times M)}
	\le
	C\|P^*(f)\|_{L^1(N\times M)}.
	\]
	
	It remains to compare the non-tangential maximal function \(P^*(f)\)
	with the radial maximal function \(P^+(f)\). Fix \(0<q<1\).
	The local \(L^q\)-to-\(L^\infty\) estimate for harmonic functions,
	applied successively in the two cylinders, gives
	\[
	|U(u,t;v,s)|^q
	\le
	\frac{C_q}
	{|B_N(x,ct)|\,|B_M(y,cs)|}
	\int_{B_N(x,ct)}
	\int_{B_M(y,cs)}
	(P^+(f)(z,w))^q\,d_Mw\,d_Nz
	\]
	whenever
	\[
	(u,t)\in\Gamma_x^N,
	\qquad
	(v,s)\in\Gamma_y^M,
	\]
	where \(c>1\) is fixed. Here the averages in the vertical variables
	arising from the local estimate are absorbed by the supremum in
	\(P^+(f)\). Taking the supremum over the two cones yields
	\[
	P^*(f)(x,y)
	\le
	C_q
	\left[
	M_NM_M\big((P^+(f))^q\big)(x,y)
	\right]^{1/q}.
	\]
	Since non-negative Ricci curvature implies the doubling property,
	\(M_N\) and \(M_M\) are bounded on \(L^{1/q}\). Consequently,
	\[
	\begin{aligned}
		\|P^*(f)\|_{L^1(N\times M)}
		&\le
		C_q
		\left\|
		M_NM_M\big((P^+(f))^q\big)
		\right\|_{L^{1/q}(N\times M)}^{1/q}
		\\
		&\le
		C
		\left\|(P^+(f))^q\right\|_{L^{1/q}(N\times M)}^{1/q}
		=
		C\|P^+(f)\|_{L^1(N\times M)}.
	\end{aligned}
	\]
	Finally,
	\[
	\|P^+(f)\|_{L^1(N\times M)}
	=
	\|f\|_{H^1(N\times M)},
	\]
	and the result follows.
	\endpf

	From this point on, set
	\[
	N=S^{n-1},
	\qquad
	M=S^{m-1}.
	\]
	We next record the corresponding Riesz transform characterization on the
	product sphere.
	For a manifold \(D\), write
	\[
	{\mathfrak R}_{D}'
	=
	({\mathfrak R}_{D},Id),
	\]
	where \(Id\) denotes the identity operator. Thus
	\[
	({\mathfrak R}_{N}'\otimes{\mathfrak R}_{M}')(f)
	=
	\big(
	f,\,
	{\mathfrak R}_{N}f,\,
	{\mathfrak R}_{M}f,\,
	({\mathfrak R}_{N}\otimes{\mathfrak R}_{M})f
	\big),
	\]
	with the natural finite-dimensional norm.

	\begin{Proposition}\label{proposition13}
		Let \(f\in L^1(N\times M)\) satisfy the separate cancellation
		conditions
		\[
		\int_N f(x,y)\,d_Nx=0
		\quad\text{for a.e. }y\in M,
		\qquad
		\int_M f(x,y)\,d_My=0
		\quad\text{for a.e. }x\in N.
		\]
		Then
		\begin{equation}\label{eq:product-riesz-general}
			\begin{aligned}
				\|f\|_{H^1(N\times M)}
				&=
				\|P^+(f)\|_{L^1(N\times M)}
				\cong
				\|P^*(f)\|_{L^1(N\times M)}
				\\
				&\cong
				\|A(f)\|_{L^1(N\times M)}
				\cong
				\left\|
				(\mathfrak R_N'\otimes\mathfrak R_M')f
				\right\|_{L^1(N\times M)}.
			\end{aligned}
		\end{equation}
	\end{Proposition}
	
	{\bf Proof:}
	We first assume that \(f\) is smooth. The product area--maximal
	estimate and the sub-mean argument in Proposition
	\ref{proposition5} give
	\[
	\|A(f)\|_1
	\lesssim
	\|P^*(f)\|_1
	\lesssim
	\|P^+(f)\|_1,
	\]
	while \(P^+(f)\le P^*(f)\) pointwise.
	
	Set
	\[
	G
	=
	(\mathfrak R_N'\otimes\mathfrak R_M')f.
	\]
	Let \(A_M\) denote the one-parameter area function acting in the
	\(M\)-variable. Applying the Hilbert-valued one-parameter
	area--Riesz characterization first in the \(N\)-variable and then in
	the \(M\)-variable gives
	\[
	\begin{aligned}
		\|A(f)\|_1
		&\cong
		\|A_M(f)\|_1
		+
		\|A_M(\mathfrak R_Nf)\|_1
		\\
		&\cong
		\|f\|_1
		+
		\|\mathfrak R_Nf\|_1
		+
		\|\mathfrak R_Mf\|_1
		+
		\|(\mathfrak R_N\otimes\mathfrak R_M)f\|_1
		\\
		&\cong
		\|G\|_1.
	\end{aligned}
	\]
	The separate cancellation conditions remove the constant modes in
	each application. They are preserved by transforms acting in the
	other variable because the corresponding operators commute.
	
	It remains to estimate the maximal function by \(G\). For \(D=N,M\),
	let
	\[
	L_D^{(1)}
	=
	d_Dd_D^*+d_D^*d_D
	\]
	be the nonnegative Hodge Laplacian on one-forms, and write
	\[
	\vec P_\tau^D
	=
	e^{-\tau\sqrt{L_D^{(1)}}}.
	\] 
	Define
	\[
	\begin{aligned}
		U_{00}(t,s)
		&=
		(P_t^N\otimes P_s^M)f,
		\\
		U_{10}(t,s)
		&=
		(\vec P_t^N\otimes P_s^M)(\mathfrak R_Nf),
		\\
		U_{01}(t,s)
		&=
		(P_t^N\otimes\vec P_s^M)(\mathfrak R_Mf),
		\\
		U_{11}(t,s)
		&=
		(\vec P_t^N\otimes\vec P_s^M)
		\big((\mathfrak R_N\otimes\mathfrak R_M)f\big),
	\end{aligned}
	\]
	and put
	\[
	\mathcal U=(U_{00},U_{10},U_{01},U_{11}).
	\]
	For \(u=P_\tau^Dg\) and
	\(\omega=\vec P_\tau^D\mathfrak R_Dg\), functional calculus gives
	\[
	d_D\omega=0,
	\qquad
	\partial_\tau\omega=-d_Du,
	\qquad
	\partial_\tau u=-d_D^*\omega.
	\]
	Thus \(\mathcal U\) is a conjugate system in each variable.
	Its boundary value is \(G\), and its
	scalar--scalar component is
	\[
	U_{00}(t,s)
	=
	(P_t^N\otimes P_s^M)f.
	\]
	
	Choose
	\[
	\max\left\{
	\frac{n-2}{n-1},
	\frac{m-2}{m-1}
	\right\}
	<q<1.
	\]
	The conjugate Poisson identities, the Bochner formula, and the
	refined Kato inequality show that \(|\mathcal U|^q\) is subharmonic
	separately in the two cylinders; see also
	\cite[proof of Theorem 3, pp.~34--35]{CL} for the one-parameter argument.
	The separate cancellation conditions ensure that the corresponding
	Poisson extensions decay as either semigroup parameter tends to infinity.
	Since \(N\) and \(M\) are compact, the maximum
	principle can therefore be applied successively to give
	\[
	|\mathcal U(t,s)(x,y)|^q
	\le
	(P_t^N\otimes P_s^M)(|G|^q)(x,y).
	\]
	Taking the supremum over the product non-tangential regions and using
	the standard domination of the positive non-tangential Poisson
	maximal operators by the Hardy--Littlewood maximal operators, we get
	\[
	P^*(f)
	\le
	C\left[M_NM_M(|G|^q)\right]^{1/q}.
	\]
	Since \(1/q>1\),
	\[
	\begin{aligned}
		\|P^*(f)\|_1
		&\le
		C\left\|M_NM_M(|G|^q)\right\|_{L^{1/q}}^{1/q}
		\\
		&\le
		C\left\||G|^q\right\|_{L^{1/q}}^{1/q}
		=
		C\|G\|_1.
	\end{aligned}
	\]
	Combining the preceding estimates proves
	\eqref{eq:product-riesz-general} for smooth \(f\).
	
	The area--Riesz equivalence above is valid for general \(f\), since it
	follows directly from the Hilbert-valued one-parameter theorem. It
	remains only to extend the maximal estimate. If \(\|G\|_1<\infty\), let
	\(f_\varepsilon\) be smooth probability rotation averages. Rotation
	covariance, with the natural orthogonal actions on one-form and tensor
	components, gives \(G_\varepsilon\to G\) in \(L^1\), while the scalar
	Poisson extensions converge locally uniformly. The smooth estimate and
	Fatou's lemma therefore yield
	\[
	\|P^*(f)\|_1\le C\|G\|_1.
	\]
	If \(\|G\|_1=\infty\), this estimate is automatic. This completes the
	proof.
	\endpf

	Applying Proposition \ref{thm:tworiesz} successively in the two
	variables, componentwise for finite-dimensional vector-valued data,
	gives the following corollary.
\begin{Corollary}\label{cor:5.3}
		For every \(f\in L_0^1(S^{n-1}\times S^{m-1})\),
		\begin{eqnarray}\label{eq:product-sphere-riesz}
			\|f\|_{H^1}
			&\cong&
			\|P^*(f)\|_{1}
			\cong
			\|A(f)\|_{1}
			\nonumber\\
			&\cong&
			\left\|
		\left(	{\mathfrak R}_{S^{n-1}}'
			\otimes
			{\mathfrak R}_{S^{m-1}}'\right)(f)
			\right\|_{1}
			\nonumber\\
			&\cong&
			\left\|
		\left(	{\mathcal R}_{S^{n-1}}'
			\otimes
			{\mathcal R}_{S^{m-1}}'\right)(f)
			\right\|_{1}
			\nonumber\\
			&\cong&
			\left\|
			\left(\widetilde R_{S^{n-1}}'
			\otimes
			\widetilde R_{S^{m-1}}'\right)(f)
			\right\|_{1}.
		\end{eqnarray}
	\end{Corollary}

	It remains to connect the product Hardy-space quantities above with the
	ball-Poisson quantities on the product sphere. For
	\(f\in L^1(S^{n-1}\times S^{m-1})\), write
	\[
	C(f)(u,v)
	=
	\left(
	C^{S^{n-1}}\otimes C^{S^{m-1}}
	\right)(f)(u,v),
	\qquad
	u\in B^n,\quad v\in B^m,
	\]
	for the product ball-Poisson extension.

	We define the radial maximal function, the non-tangential maximal function,
	and the product Lusin area integral by
	\begin{eqnarray}\label{def:procs}
		\begin{aligned}
			C_{+}(f)(x',y')
			&=
			\sup_{0<r,s<1}
			\left|
			C(f)(rx',sy')
			\right|,
			\\
			C_{*}(f)(x',y')
			&=
			\sup_{(u,v)\in\Theta_{x',y'}}
			\left|
			C(f)(u,v)
			\right|,
			\\
			S(f)(x',y')
			&=
			\left(
			\Liint_{\Theta_{x',y'}}
			\left|
			\left(
			\nabla_u\otimes\nabla_v
			\right)
			C(f)(u,v)
			\right|^2
			\frac{du}{(1-|u|)^{n-2}}
			\frac{dv}{(1-|v|)^{m-2}}
			\right)^{1/2}.
		\end{aligned}
	\end{eqnarray}
	Here
	\[
	\Theta_{x',y'}
	=
	\Theta_{x'}^{S^{n-1}}\times \Theta_{y'}^{S^{m-1}},
	\]
	and \(\nabla_u\otimes\nabla_v\) means that the Euclidean gradients are taken in
	both ball variables. The norm is the natural tensor norm.

	In view of Corollary \ref{cor:5.3}, it remains to compare these
	ball-Poisson quantities with the product Riesz-transform quantities. This will
	be done in the following two propositions. 
	\begin{Proposition}\label{prop:mulriszarea}
		For \(f\in L_0^1(S^{n-1}\times S^{m-1})\), one has
		\begin{eqnarray}\label{equiv:drs}
			\left\|
		\left(	{\mathcal R}_{S^{n-1}}'
			\otimes
			{\mathcal R}_{S^{m-1}}'\right)(f)
			\right\|_1
			\cong
			\|S(f)\|_1 .
		\end{eqnarray}
	\end{Proposition}

	\begin{Proposition}\label{prop:areamaximal}
		For \(f\in L_0^1(S^{n-1}\times S^{m-1})\), one has
		\begin{eqnarray}\label{equiv:areamaximal}
			\|C_+(f)\|_1
			\cong
			\|C_*(f)\|_1
			\cong
			\|S(f)\|_1 .
		\end{eqnarray}
	\end{Proposition}

	\begin{Remark}
		Related results for product upper half-spaces can be found in
		\cite{Chenshuo,Sato}. In the present setting, we prove the required estimate
		directly by using the product conjugate harmonic system associated with the
		ball Poisson kernels.
	\end{Remark}

	In the two proofs below, whenever a reduction to smooth data is needed,
	we use smooth probability rotation averages \(f_\varepsilon\) on the
	two spheres. These averages preserve separate cancellation and converge
	to \(f\) in \(L^1\). Rotation covariance and Minkowski's inequality give
	contraction of the relevant area, maximal, and Riesz quantities, while
	the Poisson extensions and their relevant derivatives converge locally
	uniformly. When a transformed component is not known a priori to belong
	to \(L^1\), the one-parameter weak \((1,1)\) estimates and Fatou's lemma
	first place the single-transform limits in \(L^1\). Rotation covariance
	then gives their \(L^1\)-convergence, after which the same weak-type
	argument identifies the mixed transform. A final application of Fatou's
	lemma passes the estimates to \(f\). We use this regularization without
	further comment.

	{\bf Proof of Proposition \ref{prop:mulriszarea}:}
	We first assume that \(f\) is smooth and satisfies the separate
	cancellation conditions. The proof is obtained by applying the
	Hilbert-valued one-parameter equivalence \eqref{eq:hilbert-area-riesz}
	successively in the two spherical variables.

	For convenience, if \(g=g(x',y')\) is scalar- or finite-dimensional
	vector-valued, write
	\[
	S_m(g)(x',y')
	=
	\left(
	\int_{\Theta_{y'}^{S^{m-1}}}
	\left|
	\nabla_v C^{S^{m-1}}(g(x',\cdot))(v)
	\right|^2
	\frac{dv}{(1-|v|)^{m-2}}
	\right)^{1/2}.
	\]
	Thus \(S_m\) is simply the one-parameter area integral in the second
	variable.

	We first apply \eqref{eq:hilbert-area-riesz} in the first variable.
	Fix \(y'\in S^{m-1}\), and regard
	\[
	x'
	\longmapsto
	\left[
	v\longmapsto
	\nabla_v C^{S^{m-1}}(f(x',\cdot))(v)
	\right]
	\]
	as a function on \(S^{n-1}\) with values in the Hilbert space
	\[
	L^2\left(
	\Theta_{y'}^{S^{m-1}},
	\frac{dv}{(1-|v|)^{m-2}};
	\mathbb C^m
	\right).
	\]
	The cancellation of \(f\) in the first variable gives
	\[
	\int_{S^{n-1}}
	\nabla_v C^{S^{m-1}}(f(x',\cdot))(v)\,dx'
	=
	\nabla_v C^{S^{m-1}}
	\left(
	\int_{S^{n-1}}f(x',\cdot)\,dx'
	\right)(v)
	=
	0.
	\]
	Hence the Hilbert-valued one-parameter result applies.

	The constants are independent of \(y'\). Integrating in \(y'\) and
	commuting the operators in different variables, we obtain
	\begin{equation}\label{eq:first-iteration-area}
	\|S(f)\|_1
	\cong
	\|S_m(f)\|_1
	+
	\|S_m({\mathcal R}_{S^{n-1}}f)\|_1,
	\end{equation}
	where \({\mathcal R}_{S^{n-1}}\) acts only in the first variable.

	We now apply \eqref{eq:hilbert-area-riesz} in the second variable.
	For \(f(x',\cdot)\), this gives
	\[
	\|S_m(f)\|_1
	\cong
	\|f\|_1
	+
	\|
	(Id\otimes{\mathcal R}_{S^{m-1}})f
	\|_1.
	\]
	The same argument applies to
	\(({\mathcal R}_{S^{n-1}}\otimes Id)f\). Indeed, the cancellation of
	\(f\) in the second variable implies
	\[
	\int_{S^{m-1}}
	({\mathcal R}_{S^{n-1}}f)(x',y')\,dy'
	=
	{\mathcal R}_{S^{n-1}}
	\left(
	\int_{S^{m-1}}f(\cdot,y')\,dy'
	\right)(x')
	=
	0.
	\]
	Therefore,
	\[
	\begin{aligned}
		\|S_m({\mathcal R}_{S^{n-1}}f)\|_1
		\cong\,
		&
		\|
		({\mathcal R}_{S^{n-1}}\otimes Id)f
		\|_1
+
		\|
		({\mathcal R}_{S^{n-1}}
		\otimes
		{\mathcal R}_{S^{m-1}})f
		\|_1.
	\end{aligned}
	\]
	Combining the preceding estimates, we obtain
	\begin{equation}\label{eq:product-area-four-components}
		\begin{aligned}
			\|S(f)\|_1
			\cong\,
			&
			\|f\|_1
			+
			\|
			(Id\otimes{\mathcal R}_{S^{m-1}})f
			\|_1
			\\
			&+
			\|
			({\mathcal R}_{S^{n-1}}\otimes Id)f
			\|_1
			+
			\|
			({\mathcal R}_{S^{n-1}}
			\otimes
			{\mathcal R}_{S^{m-1}})f
			\|_1\\
			&\cong 	\left\|
			\left(
			{\mathcal R}_{S^{n-1}}'
			\otimes
			{\mathcal R}_{S^{m-1}}'
			\right)f
			\right\|_1.
		\end{aligned}
	\end{equation}
	This proves \eqref{equiv:drs} for smooth \(f\).
	
	The regularization described above extends
	\eqref{eq:product-area-four-components} to general \(f\), and hence
	proves \eqref{equiv:drs}.
	\endpf

{\bf Proof of Proposition \ref{prop:areamaximal}:}
Let \(M_n\) and \(M_m\) denote the Hardy--Littlewood maximal
operators on the two spheres. Choose
\[
\max\left\{\frac{n-1}{n},\frac{m-1}{m}\right\}<q<1.
\]

We first compare the radial and non-tangential maximal functions.
For a harmonic function \(u\) in \(B^d\), the interior
\(L^q\)-to-\(L^\infty\) estimate gives
\[
|u(z)|^q
\le
\frac{C_q}{\delta^d}
\int_{B(z,c\delta)}|u(w)|^q\,dw,
\qquad
\delta=1-|z|,
\]
where \(0<c<1\) is fixed and sufficiently small. If
\(z\in\Theta_\xi\), then the directions of the points in
\(B(z,c\delta)\setminus\{0\}\) lie in a spherical cap centered at
\(\xi\) of radius at most \(C\delta\), and each corresponding radial
interval has length at most \(C\delta\). Hence, with
\[
H_u(\theta)=\sup_{0<r<1}|u(r\theta)|^q,
\]
polar coordinates give
\[
|u(z)|^q
\le
C_q\delta^{1-d}
\int_{B_{S^{d-1}}(\xi,C\delta)}
H_u(\theta)\,d\theta
\le
C_qMH_u(\xi).
\]
Taking the supremum over \(z\in\Theta_\xi\), we obtain
\[
\sup_{z\in\Theta_\xi}|u(z)|^q
\le
C_q
M\left(
\sup_{0<r<1}|u(r\,\cdot)|^q
\right)(\xi).
\]
Applying this estimate successively to \(C(f)(x,y)\) in the two
variables yields
\[
C_*(f)(x',y')
\le
C_q
\left[
M_nM_m\bigl(C_+(f)^q\bigr)(x',y')
\right]^{1/q}.
\]
Since \(1/q>1\),
$
\|C_*(f)\|_1
\lesssim
\|C_+(f)\|_1.
$
The reverse inequality is immediate, and therefore
\begin{equation}\label{equiv:radialnontan}
	\|C_*(f)\|_1\cong\|C_+(f)\|_1.
\end{equation}

The product area--maximal estimate
\cite[Theorem 1.2]{chineseequivalence}, applied with \(p=1\) and
\(\alpha=\beta=\pi/6\), gives
\begin{equation}\label{eq:area_by_max}
	\|S(f)\|_1\lesssim\|C_*(f)\|_1.
\end{equation}
It remains to prove
\begin{equation}\label{eq:maxnonriesz}
	\|C_*(f)\|_1
	\lesssim
	\|(\mathcal R_{S^{n-1}}'
	\otimes\mathcal R_{S^{m-1}}')f\|_1.
\end{equation}

Assume first that \(f\) is smooth and has separate mean zero. Define
\[
v(x,y)
=
\int_0^1\int_0^1
(C^{S^{n-1}}\otimes C^{S^{m-1}})f(\rho x,\sigma y)
\,\frac{d\rho}{\rho}\frac{d\sigma}{\sigma},
\qquad
F=\nabla_x\otimes\nabla_yv.
\]
Separate cancellation gives
\[
(C^{S^{n-1}}\otimes C^{S^{m-1}})f(0,y)
=
(C^{S^{n-1}}\otimes C^{S^{m-1}})f(x,0)
=
0,
\]
so \(v\) is smooth in \(B^n\times B^m\), including at the origins.

Using separate cancellation, we may subtract the values of the
Poisson kernels at the origins before differentiating. We obtain
\begin{equation}\label{FF}
	F(x,y)
	=
	\int_{S^{n-1}\times S^{m-1}}
	f(\omega',z')
	\mathcal F_n(x,\omega')
	\otimes
	\mathcal F_m(y,z')
	\,d\omega'\,dz',
\end{equation}
where, for \(d\ge2\), \(\zeta\ne0\), and \(\eta\in S^{d-1}\),
\[
\begin{aligned}
	\mathcal F_d(\zeta,\eta)
	=c_d\Bigg[
	&\left(
	\frac{1-|\zeta|^2}{|\eta-\zeta|^d}-1
	\right)
	\frac{\zeta'}{|\zeta|}
	\\
	&+
	d\int_0^1
	\frac{1-r^2|\zeta|^2}{|\eta-r\zeta|^{d+2}}
	\bigl(\eta-(\eta\cdot\zeta')\zeta'\bigr)\,dr
	\Bigg].
\end{aligned}
\]
The integral in \eqref{FF} extends smoothly across \(x=0\) and \(y=0\).

Since \(v\) is harmonic separately, every column of \(F\) is a
conjugate harmonic system in \(x\), and every row is one in \(y\):
\begin{equation}\label{eq:product-conjugate-system}
	\begin{gathered}
		\partial_{x_i}F_{jk}
		=
		\partial_{x_j}F_{ik},
		\qquad
		\sum_{j=1}^n\partial_{x_j}F_{jk}=0,
		\\
		\partial_{y_a}F_{jb}
		=
		\partial_{y_b}F_{ja},
		\qquad
		\sum_{k=1}^m\partial_{y_k}F_{jk}=0.
	\end{gathered}
\end{equation}
The Euler identities and separate cancellation give
\[
\begin{aligned}
	\langle x\otimes y,F(x,y)\rangle
	&=
	(x\cdot\nabla_x)(y\cdot\nabla_y)v(x,y)
=
	(C^{S^{n-1}}\otimes C^{S^{m-1}})f(x,y).
\end{aligned}
\]
Consequently,
\begin{equation}\label{eq:scalar-by-F}
	\left|
	(C^{S^{n-1}}\otimes C^{S^{m-1}})f(x,y)
	\right|
	\le |F(x,y)|.
\end{equation}

The boundary value of \(F\) is
\[
\begin{aligned}
	\Xi(f)(x',y')
	={}&
	f(x',y')x'\otimes y'
	+
	(\mathcal R_{S^{n-1}}f)(x',y')\otimes y'
	\\
	&+
	x'\otimes(\mathcal R_{S^{m-1}}f)(x',y')
	+
	(\mathcal R_{S^{n-1}}\otimes\mathcal R_{S^{m-1}})
	f(x',y').
\end{aligned}
\]
Thus
\begin{equation}\label{poisson}
	F
	=
	(C^{S^{n-1}}\otimes C^{S^{m-1}})\Xi(f).
\end{equation}

For complex-valued \(f\), include the real and imaginary parts as
separate components. The systems
\eqref{eq:product-conjugate-system} imply the refined Kato
inequalities
\[
|\nabla_x|F||^2
\le
\frac{n-1}{n}|\nabla_xF|^2,
\qquad
|\nabla_y|F||^2
\le
\frac{m-1}{m}|\nabla_yF|^2.
\]
Indeed, the \(x\)-derivative of each column of \(F\) is a symmetric
trace-free matrix \(A\), for which
\[
|A\xi|^2
\le
\frac{n-1}{n}|A|^2|\xi|^2;
\]
summing over the columns gives the first inequality, and the second
is analogous. Therefore,
\[
\begin{aligned}
	\Delta_x|F|^q
	&=
	q|F|^{q-2}
	\left(
	|\nabla_xF|^2+(q-2)|\nabla_x|F||^2
	\right)
	\\
	&\ge
	q|F|^{q-2}
	\left(
	1-(2-q)\frac{n-1}{n}
	\right)
	|\nabla_xF|^2
	\ge0,
\end{aligned}
\]
and similarly \(\Delta_y|F|^q\ge0\). At zeros of \(F\), this follows
by applying the preceding calculation to
\((|F|^2+\varepsilon^2)^{q/2}\) and letting
\(\varepsilon\downarrow0\).

Successive Poisson majorization now gives
\[
|F(x,y)|^q
\le
(C^{S^{n-1}}\otimes C^{S^{m-1}})
\bigl(|\Xi(f)|^q\bigr)(x,y).
\]
Set
\[
F^*(x',y')
=
\sup_{(x,y)\in\Theta_{x',y'}}|F(x,y)|.
\]
By \eqref{eq:scalar-by-F},
\[
C_*(f)(x',y')\le F^*(x',y').
\]
The non-tangential maximal function of a positive product Poisson
extension is bounded by \(CM_nM_m\). Hence
\[
F^*(x',y')
\le
C
\left[
M_nM_m\bigl(|\Xi(f)|^q\bigr)(x',y')
\right]^{1/q},
\]
and consequently
\begin{equation}\label{es:friesz}
	\|C_*(f)\|_1
	\le
	\|F^*\|_1
	\le
	C\|\Xi(f)\|_1.
\end{equation}

The four terms in \(\Xi(f)\) belong respectively to the mutually
orthogonal radial--radial, tangential--radial, radial--tangential,
and tangential--tangential components. In particular,
\[
\begin{aligned}
	|\Xi(f)|^2
	={}&
	|f|^2
	+
	|\mathcal R_{S^{n-1}}f|^2
	+
	|\mathcal R_{S^{m-1}}f|^2
+
	|(\mathcal R_{S^{n-1}}
	\otimes\mathcal R_{S^{m-1}})f|^2.
\end{aligned}
\]
Therefore
\[
\|\Xi(f)\|_1
\cong
\|(\mathcal R_{S^{n-1}}'
\otimes\mathcal R_{S^{m-1}}')f\|_1,
\]
which proves \eqref{eq:maxnonriesz} for smooth \(f\).

The same regularization extends \eqref{eq:maxnonriesz} to general
\(f\).

Finally, combine \eqref{eq:maxnonriesz},
Proposition \ref{prop:mulriszarea}, and \eqref{eq:area_by_max}, and
then use \eqref{equiv:radialnontan}, to obtain
\[
\|C_+(f)\|_1
\cong
\|C_*(f)\|_1
\cong
\|S(f)\|_1.
\]
\endpf

Together with Corollary \ref{cor:5.3}, Propositions
\ref{prop:mulriszarea} and \ref{prop:areamaximal} complete the proof of
Theorem \ref{thm3}.

\section*{Acknowledgments}
J.~Chen, D.~Fan, and M.~Wang were partially supported by the National
Key R\&D Program of China under Grant No.~2022YFA1005703. Additional
support was provided by the National Natural Science Foundation of China
under Grant No.~12571109 (J.~Chen and D.~Fan), Grant No.~12371105
(D.~Fan), and Grant No.~12371100 (M.~Wang).

\end{document}